\documentclass{amsart}
\usepackage{amstext,amssymb,amsmath,amsbsy,mathrsfs,mathtools,eucal,amsfonts,amscd,exscale}
\usepackage{graphics, epstopdf}
\usepackage{verbatim}
\usepackage{appendix}
\usepackage{rotating}
\usepackage{xspace}
\usepackage{enumerate}
\usepackage{enumitem}
\usepackage{wrapfig, lipsum}

\usepackage{comment}

\usepackage{color}
\definecolor{red}{rgb}{1,0,0}

\definecolor{blue}{rgb}{0,0,1}

\usepackage{tikz}
\usetikzlibrary{arrows}

\numberwithin{equation}{section}
\numberwithin{table}{section}
\numberwithin{figure}{section}

\newcommand{\Le}{L_{\epsilon}}
\newcommand{\Ie}{I_{\epsilon}}
\newcommand{\Qe}{Q_{\epsilon}}

\newcommand{\du}[2]{\delta u (#1, #2)}

\newcommand{\ue}{u^{\epsilon}}
\newcommand{\N}[1]{\omega (#1)}

\newcommand{\bld}[1]{\boldsymbol{#1}}

\newcommand{\abs}[1]{| #1 |}
\newcommand{\Abs}[1]{\big| #1 \big|}

\newcommand{\vphi}{\varphi}

\newcommand{\dm}{\Omega}
\newcommand{\dme}{\Omega_{\epsilon}}
\newcommand{\bdry}{\partial \Omega}

\newcommand{\Th}{{\mathcal{T}_{h}}}

\newcommand{\Fh}{{\mathcal{F}_{h}}}
\newcommand{\F}{{\mathcal{F}}}
\newcommand{\Nh}{{\mathcal{N}_{h}}}
\newcommand{\Vhzero}{{\mathbb{V}_{h}^0}}
\newcommand{\Vh}{{\mathbb{V}_{h}}}

\newcommand{\dist}{\textrm{dist}\,}
\newcommand{\gradv}{\bld{\nabla}}
\newcommand{\laplace}{\bld{\Delta}}
\newcommand{\norm}[1]{\|\,#1\,\|}
\newcommand{\inftynorm}[1]{\|\,#1\,\|_{L^{\infty}(\dm)}}
\newcommand{\inftynormh}[1]{\|\,#1\,\|_{L^{\infty}(\dm_h)}}

\newcommand{\innerpro}[2]{\langle#1, #2\rangle}
\newcommand{\squishlist}{
 \begin{list}{$\bullet$}
  { \setlength{\itemsep}{0pt}
     \setlength{\parsep}{3pt}
     \setlength{\topsep}{3pt}
     \setlength{\partopsep}{0pt}
     \setlength{\leftmargin}{0.9em}
     \setlength{\labelwidth}{0.5em}
     \setlength{\labelsep}{0.5em} } }
\newcommand{\squishlisttwo}{
 \begin{list}{$\bullet$}
  { \setlength{\itemsep}{0pt}
     \setlength{\parsep}{0pt}
    \setlength{\topsep}{0pt}
    \setlength{\partopsep}{0pt}
    \setlength{\leftmargin}{2em}
    \setlength{\labelwidth}{1.5em}
    \setlength{\labelsep}{0.5em} } }
\newcommand{\squishend}{
  \end{list}  }

\newtheorem*{condition*}{Condition}

\newtheorem{theorem}{Theorem}[section]
\newtheorem{proposition}{Proposition}[section]
\newtheorem{lemma}{Lemma}[section]
\newtheorem{prop}[lemma]{Proposition}
\newtheorem{corollary}[lemma]{Corollary}
\newtheorem*{definition*}{Definition}
\theoremstyle{remark}
\newtheorem{remark}{Remark}[section]

\title{Discrete ABP Estimate and Convergence Rates for Linear
    Elliptic Equations in Non-divergence Form}

\author{Ricardo H. Nochetto}
\address{Department of Mathematics and Institute for Physical
  Science and Technology, University of Maryland, College Park}
\curraddr{}
\email{rhn@math.umd.edu}
\thanks{Both authors were partially supported by NSF Grants
  DMS-1109325 and DMS-1411808.}

\author{Wujun Zhang}
\address{    Department of Mathematics, Rutgers University}
\curraddr{}
\email{wujun@math.rutgers.edu}
\thanks{The second author was partially supported by the Brin
  Postdoctoral Fellowship of the University of Maryland}

\begin{document}
\maketitle
\keywords

\begin{abstract}
  We design a two-scale finite element method (FEM) for linear
  elliptic PDEs in non-divergence form $A(x) : D^2 u(x) =
  f(x)$ in a bounded but not necessarily convex
  domain $\Omega$ and study it in the max norm.
  The fine scale is given by the meshsize $h$ whereas the
  coarse scale $\epsilon$ is dictated by an integro-differential
  approximation of the PDE. We show that the FEM satisfies
  the discrete maximum principle (DMP) for any uniformly positive
  definite matrix $A$ provided that the mesh is face weakly
  acute. 
  We establish a discrete Alexandroff-Bakelman-Pucci (ABP) estimate
  which is suitable for finite element analysis. Its proof relies on
  a discrete Alexandroff estimate which expresses the min of a
  convex piecewise linear function in terms of the measure of its
  sub-differential, and thus of jumps of its gradient. 
  The discrete ABP estimate leads, under suitable regularity
  assumptions on $A$ and $u$, to pointwise error estimates of the form
\begin{equation*}
\inftynorm{u - u^{\epsilon}_h} \leq \, C(A,u) \, h^{2\alpha /(2 + \alpha)} \Abs{\ln h}
\qquad 0< \alpha \leq 2,
\end{equation*}
provided $\epsilon \approx h^{2/(2+\alpha)}$.
Such a convergence rate is at best of order 
$ h \Abs{\ln h}$, which turns out to be quasi-optimal.


\end{abstract}

\smallskip
\begin{center}
Communicated by Endre S{\"u}li
\end{center}

\pagestyle{myheadings}
\thispagestyle{plain}
\markboth{R.H. Nochetto and W. Zhang}{FEM for elliptic equations in non-divergence form}

%
                            \section{Introduction}\label{S:introduction}
%

We consider second order elliptic equations in non-divergence form, 
\begin{subequations}
\label{pde}
\begin{align}\label{pde_a}
Lu(x) =  A(x) : D^2 u(x) &= f(x)  \quad &&\text{in $\dm$}
\\
\label{pde_b}
u &= 0 \quad &&\text{on $\bdry$},
\end{align}
\end{subequations}
where $\dm$ denotes a bounded but not necessarily convex
domain in $\mathbb{R}^d$ ($d \ge 2$) with
$C^{1,1}$ boundary $\partial\dm$, $f \in L^d(\dm)$ and  $A$ is a
measurable $d \times d$ matrix-valued function satisfying the uniformly
ellipticity condition for a.e. $x\in\Omega$:
\begin{align}
\label{elliptic_condition}
\lambda I \leq A (x) \leq \Lambda I ,
\end{align}
for some positive constants $\lambda$  and $\Lambda$ with
moderate aspect ratio $\Lambda/\lambda\ge1$.
Moreover, we assume the vanishing Dirichlet condition
\eqref{pde_b} only for simplicity.

The elliptic PDEs \eqref{pde_a} in non-divergence form arises in
linearization processes of fully nonlinear PDEs. The latter in turn arise in
stochastic optimal control, nonlinear elasticity, fluid dynamics,
image processing, materials science, and mathematical
finance. They are thus ubiquitous in science and
engineering. \looseness=-1

The structure of \eqref{pde_a} is deceivingly simple.
For example, \eqref{pde} with forcing
$f=0$ and discontinuous coefficient $A$ given by
\[
    A(x) = I_{d \times d} + \frac{d + \alpha -2}{1-\alpha}
    \frac{x}{|x|} \otimes \frac{x}{|x|}
\]
admits two solutions in the unit ball $B_1(0)$ centered at $0$, namely
$u(x) = |x|^{\alpha} - 1$ and $u(x)= 0$, which
happen to be of class $H^2(\Omega)$ provided $d>2(2-\alpha)$ for any
$0<\alpha<1$. Several notions of solutions of \eqref{pde} are
available in the literature:
\squishlist
\item
{\it $H^2$-solutions.}  
For $d=2$, $\Omega$ convex, $A \in L^{\infty}(\dm)$ and $f \in L^2(\dm)$,
S.~N. Bernstein established the $H^2$-regularity of $u$ along with the
bound \cite{Bernstein1910}, \cite[Chapter 3, section 19]{LadyzhenskayaUraltseva68}
\begin{equation}\label{H2-bound}
      \norm{u}_{H^2(\dm)} \leq C \norm{f}_{L^2(\dm)} \, .
\end{equation}
For $d\ge2$, if the coefficient matrix
$A=(a_{ij})_{i,j=1}^d$ satisfies the Cord\`{e}s condition
\begin{equation}\label{cordes}
  (d -1 + \epsilon) \sum_{i, j= 1}^d a_{ij}^2 \leq \left( \sum_{i=1}^d a_{ii} \right)^2
\end{equation}
with $\epsilon>0$, and $\Omega$ is convex,
then there is a unique strong solution $u \in H^2(\Omega) \cap H^1_0(\Omega)$
satisfying \eqref{H2-bound};
see \cite{MaugeriPalagachevSoftova00}. This condition is valid for any
$A \in L^{\infty}(\dm)$ satisfying \eqref{elliptic_condition} for
$d=2$, thereby being consistent with \cite{Bernstein1910,LadyzhenskayaUraltseva68},
but imposes a restriction on off-diagonal elements and
the ratio $\Lambda/\lambda$ for $d > 2$.
\item
{\it Strong solutions.} For $d \geq 2$ if $A \in C(\overline{\Omega})$ and $\Omega$ is of
  class $C^{1,1}$,
  then the Calder\'on-Zygmund theory guarantees the existence and
  uniqueness of solutions $u \in W^{2}_p(\dm)$ for any $f \in
  L^p(\dm)$ along with the stability bound \cite{GilbargTrudinger01}
\begin{align}\label{CZ}
  \norm{u}_{W^2_p(\dm)} \leq C \norm{f}_{L^p(\dm)}
  \qquad \text{for $1 < p < \infty$.}
\end{align}
This theory extends to vanishing mean oscillation matrices $A \in \text{VMO}(\dm)$
with uniform VMO-modulus of continuity
\cite{ChiarenzaFrascaLongo91,ChiarenzaFrascaLongo93}; see \eqref{vmo}
below for a definition.
\item
{\it Classical solutions.}
For $d \geq 2$, if $A \in C^{0, \alpha}(\overline\dm)$ and $\Omega$ is
of class $C^{2,\alpha}$ for some $0<\alpha < 1$,
then the Schauder theory guarantees the
existence and uniqueness of a solution $u \in
C^{2,\alpha}(\overline\dm)$ for any $f\in C^{0,\alpha}(\overline\dm)$
along with the bound \cite{GilbargTrudinger01}
\begin{equation}\label{schauder}
    \norm{u}_{C^{2,\alpha}(\overline\dm)} \leq C \norm{f}_{C^{0,\alpha}(\overline\dm)}.
\end{equation}
\item
{\it Viscosity solutions:}
Weaker notions of solutions, such as $L^p$-viscosity solutions
\cite{CCKS:96} and
good solutions \cite{Jensen95}, exists to deal with
discontinuous coefficients. However, no comparison principle has been
proved except when a strong $W^2_p$ solution exists, in which case
it coincides with the $L^p$-viscosity solution for $p \ge d$.
The famous non-uniqueness
counterexample of Nadirashvili \cite{Nadirashvili97} for $d\geq 3$,
further studied by M. Safonov \cite{Safonov99},
shows that there cannot be a comparison principle for \eqref{pde} with
discontinuous coefficients.
Moreover, one can construct two sequences of regularized matrices
$\{A_k^i\}_{k=1}^\infty$ for $i=1,2$ converging to the same limit $A$
but such that the corresponding solutions $\{u_k^i\}_{k=1}^\infty$ of \eqref{pde}
converge uniformly as $k\to\infty$ to different limits $u^i$ which are both
$L^p$-viscosity solutions of \eqref{pde}.
\squishend

In contrast to an extensive numerics literature for elliptic PDEs in divergence form, the numerical approximation for PDEs in non-divergence form reduces to a few papers. 
Among these, we mention
the discrete Hessian method of O. Lakkis and T. Pryer \cite{LakkisPryer11},
  the DG methods of I. Smears and E. S{\"u}li \cite{SmearsSuli13} for
  the Cord\`{e}s condition \eqref{cordes} as well as
  the $H^1$-conforming method of X. Feng,
  L. Hennings and M. Neilan \cite{FengHenningsNeilan} and the weak
  Galerkin method of C. Wang and J. Wang \cite{WangWang}, both for
  coefficients $A\in C(\overline{\Omega})$. In \cite{FengHenningsNeilan, SmearsSuli13,
    WangWang}, the FEMs are shown to be stable in the broken
  $H^2$-seminorm via suitable discrete inf-sup conditions. Moreover,
  they prove optimal error estimates in the broken $H^2$-seminorm
  under either suitable local regularity assumptions on $u$
  \cite{FengHenningsNeilan, SmearsSuli13}
  or global ones  \cite{WangWang}.

The numerics literature is relatively larger for
fully nonlinear second order elliptic PDEs. The following papers
are somewhat related to this one: the augmented Lagrangian approach by
E.J. Dean and R. Glowinski \cite{DeanGlowinski03},
the finite element method
by M. Jensen and I. Smears \cite{JensenSmears13} and I. Smears and
E. S\"{u}li \cite{SmearsSuli14} for the Hamilton-Jacobi-Bellman (HJB) equation,
the finite difference method by J. D. Benamou, B. Froese, A. Oberman
\cite{BenamouFroeseOberman14} for optimal transportation,
and semi-Lagrangian
methods for linear and nonlinear elliptic problems by K. Debrabant and
E. R. Jakobsen \cite{DebrabantJakobsen13}, by J. F. Bonnans and
H. Zidani \cite{BonnansZidani03} and by F. Camilli and
M. Falcone \cite{CamilliFalcone95}.
The latter methods deal with two scales, the finer one being related to the
mesh and the coarser scale being dictated by a nodal
(wide stencil) finite difference operator
which ensures monotonicity and consistency; this feature
is known for finite difference approximations of elliptic PDEs
\cite{MotzkinWasow52, Kocan95} and is also
present in our finite element construction below.
We also refer to the books \cite{FlemingSoner06, KushnerDupuis01}, and
references therein, for numerical methods for the HJB equation which built on
its probabilistic interpretation.

G. Barles and P. Souganidis have proposed an abstract framework for
uniform convergence to viscosity solutions which hinges on
stability, monotonicity, and operator consistency \cite{BarlesSouganidis91}.
These properties are
tricky to enforce simultaneously. If $\Th$ is a quasi-uniform mesh of
size $h$, then we say that a discrete operator
$L_h$ is {\it monotone} if, for any two discrete
functions $u_h \leq v_h$ with equality at node $x_i$, then
\begin{equation}\label{monotonicity-0}
L_h u_h (x_i) \leq L_h v_h (x_i).
\end{equation}
We say that $L_h$ is {\it consistent} if for every $\varphi \in C^2(\dm)$, 
\begin{equation}\label{consistency-0}
  \lim_{h \rightarrow 0} L_h [{I_h} \varphi] (x_h) = A(x_0) : D^2 \varphi(x_0) 
\qquad
\text{ for all sequences $x_h \rightarrow x_0$, }
\end{equation}
where ${I_h} \varphi$ denotes the Lagrange interpolant of $\varphi$.
Consider now the centered
finite difference approximation of the Hessian using a nine-point stencil 
\[
D^{ij}_{h} u (x) =  \frac {1} {4 h^2} \Big(u(x + h e_i + h e_j) - u(x + h e_i - h e_j) 
- u(x - h e_i + h e_j) + u(x - h e_i - h e_j) \Big),
\]
which is consistent but not monotone. In fact, if 
$u_h(x + h e_1 + h e_2) = - 4 h^2$ and $u_h = 0$ at the other eight nodes,
then the discrete Hessian is 
$D^2_{h} u_h (x) = \bigl(\begin{smallmatrix}
0&-1\\ -1&0
\end{smallmatrix} \bigr).$
If $A (x) = \bigl(\begin{smallmatrix}
1&- 1/2\\ -1/2&1
\end{smallmatrix} \bigr),$
then $A(x) : D^2_h u_h(x)  = 1$ which violates \eqref{monotonicity-0}
when compared with $v_h=0$. 

The finite element Laplacian $\laplace_h$
for any piecewise linear function $v_h$ is given by
\begin{align}\label{discrete-laplace}
\laplace_h v_h (x_i) := - \left(\int_\Omega \phi_i\right)^{-1}
\int_{\dm} \nabla v_h \cdot \nabla \phi_i 
\end{align}
where $\phi_i$ is the hat function associated with node $x_i$.
On {\it weakly acute} meshes $\Th$, $\laplace_h$ satisfies
\eqref{monotonicity-0} (see \eqref{weaklyacute} and Lemma \ref{Monotonicity}),
but it might not satisfy \eqref{consistency-0} even on uniform meshes, namely
\begin{equation}\label{lack-consistency}
\laplace_h I_h u(x_i)  \quad\not\rightarrow \quad
\laplace u(x_i)\qquad\textrm{as } h\to0.
\end{equation}
To see this, we consider an example from
\cite[p. 146]{JensenSmears13}: let $\Th$ be the mesh in $\mathbb{R}^2$
consisting of four triangles whose vertices are 
$z_0 = (0, 0)$, 
$z_1 = (h, 0)$, 
$z_2 = (0, h)$,
$z_3 = (-h, 0)$,
$z_4 = (0, -h)$;
if $u(x_1,x_2) = x_1^2 + x_2^2$, then a simple calculation yields
\begin{align}\label{inconsistency}
\Delta_h I_h u (z_0) = 6 \neq 4 = \laplace u(z_0)
\qquad\forall h>0.
\end{align}
This shows that \eqref{consistency-0} is too
restrictive for finite element analysis, which was already observed
and circumvented by M.~Jensen and I.~Smears \cite{JensenSmears13}.

Regarding rates of convergence in the max norm for viscosity
solutions of fully nonlinear PDEs, 
we refer to H-J.~Kuo and N. Trudinger \cite{KuoTrudinger92},
L. Caffarelli and P. Souganidis \cite{CaffarelliSouganidis08}, 
N. V. Krylov \cite{Krylov97, Krylov00},
and  G. Barles and E. R. Jakobsen \cite{BarlesJakobsen05}. 

Our primary goal in this paper is to design a two-scale
finite element method for \eqref{pde}, which is monotone and operator
consistent, study its stability properties
and derive rates of convergence in the max norm
within the context of classical solutions, thereby requiring at least
$C^{1,1}$ domains for regularity purposes.
To this end we develop a novel technical tool
for any bounded domains, a discrete Alexandroff-Bakelman-Pucci
(ABP) estimate which mimics the continuous ABP estimate; the latter is a
conerstone in the theory of fully nonlinear elliptic PDEs. To
introduce the coarse scale $\epsilon$, let's assume for the moment that the
coefficient matrix $A$ is uniformly continuous in $\Omega$ and rewrite
\eqref{pde_a} as follows
\begin{align}\label{pde2}
A(x) : D^2 u(x) = \frac{\lambda}{2} \laplace u(x) + \Big(A(x) - 
\frac{\lambda}{2} I\Big) : D^2 u(x),
\end{align}
where the second term is still elliptic thanks to the ellipticity
condition \eqref{elliptic_condition}.
Our method hinges
on the approximation of \eqref{pde_a}, and thus of \eqref{pde2}, by a
linear {\it integro-differential} operator proposed by L. Caffarelli and
L. Silvestre in \cite{CaffarelliSilvestre10}
\begin{equation}\label{integro-differential-operator}
L_\epsilon \ue (x):= \; \frac{\lambda} {2} \laplace \ue (x)
+ \Ie \ue (x) = f(x)
\qquad \textrm{for all } x\in\dm,
\end{equation} 
where
\begin{equation}\label{Ie}
\Ie v(x) := \; \int_{\mathbb{R}^d} \frac{\delta v(x,y)}{\epsilon ^ {d+2}
  \det(M(x))} \varphi \Big(\frac { M^{-1}(x) y} {\epsilon} \Big) dy .
\end{equation}
Hereafter, 
$\varphi(y)$ is a radially symmetric function with compact support in the unit ball and $\int_{\mathbb R^d} \abs{y}^2 \varphi(y) dy = d$ where $d$ is the dimension of $\dm$,   
\begin{align}\label{M}
M(x) := \left( A(x) - \frac{\lambda}{2}I \right)^{1/2}
\end{align}
and 
$$
\delta v(x,y) := v(x+y) + v (x-y) - 2 v(x)
$$ 
is the centered second difference operator
with suitable modifications near $\bdry$; see \eqref{extension}. The operator \eqref{Ie} is a consistent approximation of $\left( A(x) - \frac{\lambda}{2} I \right) : D^2 v(x)$ in the sense that if $v$ is a quadratic polynomial, then
\begin{align}\label{quadratic}
\Ie v(x)  = \left( A(x) -  \frac{\lambda}{2}I \right) : D^2 v (x)  \qquad & \text{for all $\epsilon > 0$, } 
\end{align}
To see this, note that 
$
\delta v(x,y) = (y \otimes y) : D^2 v(x)  
$
for quadratic $v$ where $\otimes$ denotes the tensor product.
Since
\begin{align}\label{Iquadratics}
\Ie v(x) = \int_{\mathbb{R}^d} \frac{y \otimes y}{\epsilon ^ {d+2}
  \det (M(x))} \varphi \left( \frac { M^{-1}(x) y} {\epsilon} \right) dy : D^2 v(x) ,
\end{align}
by definition, the change of variable $z = M^{-1}(x) y / \epsilon$ yields 
\[
\Ie v(x) = M(x) \left ( \int_{\mathbb{R}^d} z \otimes z \varphi(z) dz \right ) M(x) : D^2 v(x) .
\]
Since $\varphi(z)$ is radially symmetric, we have $\int z_i z_j\varphi(z) dz = 0$
if $i \ne j$, as well as
\[
\int_{\mathbb{R}^d}  z_1^2 \varphi(z) dz = \cdots = \int_{\mathbb{R}^d}  z_d^2 \varphi(z) dz = \frac 1 d \int_{\mathbb{R}^d} \abs z ^2 \varphi(z) dz = 1. 
\]
We thus obtain
$
\int_{\mathbb{R}^d} z \otimes z \varphi(z) dz = I 
$ 
and 
\[
  \Ie v(x) = M(x)^2 : D^2 v(x) = \left( A(x) - \frac{\lambda}{2}I  \right) : D^2 v(x).
\]

We now consider a sequence of conforming quasi-uniform meshes $\{\Th\}$, made of shape
regular simplices, which induces polytope approximations $\dm_h$ of $\dm$
with boundary nodes of $\partial\dm_h$ lying on $\partial\dm$. Since
we assume throughout, except for section \ref{S:discrete-ABP}, that
$\partial\dm$ is at least $C^{1,1}$ there is a discrepancy between $\dm$ and
$\dm_h$ to account for. Given the technical nature of this endeavor,
which would complicate our discussion without adding substance,
we make the simplifying assumption that $\dm_h=\dm$; see subsection
\ref{subsec:Galerkinprojection} for further details.
We approximate $L_\epsilon u_\epsilon(x) = f(x)$ by
\begin{align}\label{intro:FEM}
  L_h^\epsilon u_h^\epsilon(x_i)
  :=\frac{\lambda}{2} \Delta_h \ue_h(x_i)  + \Ie \ue_h(x_i)  = f_i
\qquad\textrm{for all} \,x_i \in\Nh,
\end{align}
where $\ue_h = \sum_{x_j\in\Nh} U_j \phi_j$ is a continuous and
piecewise affine finite element function, $\Nh$ is the set
of internal nodes of $\Th$, and $f_i :=\int_\dm f \phi_i / \int_\dm \phi_i$.
The meshsize $h$ gives the fine scale of \eqref{intro:FEM} in that $\epsilon$ and $h$ satisfy $\epsilon\ge C h \abs{\ln h}^{1/2}$.
The integral $\Ie \ue_h(x_i)$ is simple to compute using quadrature 
because the kernel is smooth and $M(x)$ is evaluated at $x = x_i$. 
All the results in this paper are valid provided
the quadrature rule is locally
supported, consistent and positive; see subsection \ref{subsec:quadrature}.

We derive rates of convergence in the maximum norm for
\eqref{intro:FEM} in the context of classical solutions.  In
contrast to \cite{FengHenningsNeilan, SmearsSuli13, WangWang}, we do
not show an inf-sup condition to deal with the maximum norm. The main
difficulty is indeed to establish an alternative notion of
stability. We first prove that \eqref{intro:FEM} is a monotone FEM
provided that the meshes $\{\Th\}$ are weakly acute; see \eqref{weaklyacute}.
We next recall a fundamental stability property of \eqref{pde}, namely
the celebrated Alexandroff-Bakelman-Pucci (ABP) estimate.
The ABP estimate for \eqref{pde} reads \cite{CaffarelliCabre95, HanLin}:
\begin{equation}\label{ABP}
\sup_{\dm} u^- \leq C \Big ( \int_{\{u = \Gamma (u)\}} |f(x)|^d dx \Big)^{1/d} ,
\end{equation}
where $u^-(x) = \max\{-u(x), 0\}$ is the negative part of $u$
and $ \{u = \Gamma (u)\} $ denotes the (lower) contact set of $u$ with
its convex envelope $\Gamma (u)$; see \eqref{convexenvelop} and
\eqref{contactset}. This estimate gives a bound for $u^-$ while a
bound for the positive part $u^+$ can be derived in the same fashion by considering a concave envelope and corresponding (upper) contact set. A combination of both estimates yields stability of the $L^{\infty}$-norm of $u$ in terms of the $L^d$-norm of $f$. 
We establish Theorem \ref{discrete_ABP} (discrete ABP estimate)
\begin{equation}\label{intro:ABP}
\sup_{\dm} \, (\ue_h)^- \leq C \left( 
\sum_{\{x_i: \ue_h(x_i) = \Gamma (\ue_h)(x_i) \}} |f_i|^d \abs {\omega_i} \right)^{1/d},
\end{equation}
where $\{x_i: \ue_h(x_i) = \Gamma (\ue_h)(x_i) \}$ denotes the {\it (lower)
nodal contact set}, defined in \eqref{discrete-contactset} and
$\abs{\omega_i}$ stands for the volume of the star $\omega_i := \text{supp}
\; \phi_i$ associated with the node $x_i\in\Nh$. Note that the
nodal contact set is just a collection of nodes. The estimate
\eqref{intro:ABP} hinges on Proposition \ref{Alexandroff}
({\it discrete Alexandroff estimate}), which is of intrinsic interest.
It is worth mentioning that the estimates in section \ref{S:discrete-ABP}
do not require any regularity of the domain $\dm$ which is just
assumed to be bounded.
This undertaking is somewhat related to early work in the maximum
norm for linear elliptic PDE in divergence form by Ph. Ciarlet and
P.A. Raviart \cite{CiarletRaviart73}.

We would like to mention that a discrete ABP estimate is proved in
\cite{KuoTrudinger00} for finite differences on general meshes under the
assumption that the discrete operator is monotone. 
Compared with \cite{KuoTrudinger00}, the novelties of this paper are
the following:
\squishlist
\item We give a novel proof of discrete ABP estimate, which is more
  geometric in nature and suitable for FEM: it is
  based on a geometric characterization of the sub-differential of
  piecewise linear functions and control of its Lebesgue measure by
  jumps of the normal flux.
\item  The estimate in \cite{KuoTrudinger00} is sub-optimal when
  applied to our finite element method \eqref{intro:FEM}. In
  fact, it replaces the measure of star $\abs {\omega_i}\approx h^d$
  in \eqref{intro:ABP}, which corresponds to the fine scale $h$,
  by the volume $\approx\epsilon^d$ of a ball used to define
  \eqref{Ie}. The two estimates thus differ by a multiplicative factor
  $\epsilon/h \gg 1$, the ratio of scales, 
  which is responsible for suboptimal decay rates.
\item 
  Upon combining our discrete ABP estimate with operator consistency of
  \eqref{intro:FEM} in $L^\infty(\Omega)$, we derive pointwise
  rates of convergence under natural regularity
  requirements of $u$ in H\"older spaces, i.e. in the realm of
  classical solutions. We also exploit that operator consistency is
  measured in a discrete $L^d$ norm in \eqref{intro:ABP} to
  establish convergence rates for piecewise smooth solutions $u$.
\squishend

Our 2-scale FEM \eqref{intro:FEM}
extends to certain classes of discontinuous coefficients.
We recall that $A \in \text{VMO}(\dm)$, the space of {\it vanishing
mean oscillation} functions,
if the mean oscillation of $A$ satisfies for all $x\in\dm$
\begin{equation}\label{vmo}
  \sup_{\rho\le r}\frac 1 {|B_{\rho}(x)\cap\dm|}\int_{B_{\rho}(x)\cap\dm}
  | A(y) - A_{\rho}(x)| \, dy \, \le \eta(r)  \to 0 
  \quad
  \text{as $r \to 0$},
\end{equation}
where $A_{\rho}(x)$ is the mean-value of $A$ in a ball $B_{\rho}(x)$ of center $x$
and radius $\rho$ 
\[
  A_{\rho}(x) := \frac 1 {|B_{\rho}(x)\cap\dm|}
  \int_{B_{\rho}(x)\cap\dm} A(y) \, dy;
\]
function $\eta$ is the so-called VMO-modulus of continuity of $A$.
Since neither $A(x)$ nor $M(x)$ may be well
defined at each node $x=x_i$, and this is critical in
\eqref{intro:FEM}, we replace nodal values of $A$ at $x_i$ by the
means $\bar{A}(x_i)$ of $A$ over the star $\omega_i$ of $x_i$
\begin{align}
  \label{mean}
\bar{A}(x_i) := \frac{1}{\abs{\omega_i}} \int_{\omega_i} A(x) \, dx,
\qquad
M(x_i) := \left(\bar{A}(x_i) - \frac{\lambda}{2} I \right)^{1/2},
\end{align}
in the definition \eqref{Ie} of $\Ie \ue_h(x_i)$.
We prove uniform convergence in Corollary \ref{convergence-C2}
provided $u\in C^2(\dm)$ and $\epsilon = C_0 h |\ln h|$.
Obviously, the accuracy of the solution $\ue_h$ depends on the
approximation quality of $A$ by its mean. 
We show that if 
\begin{align}\label{Ldassumption}
\left( \sum_{x_i \in \Nh} \int_{\omega_i} \abs{ A(x) - \bar{A}(x_i)}^d
\, dx \right)^{1/d} \leq C h^{\beta}
\quad \text{and} \quad
u \in C^{2, \alpha}(\overline\dm)
\end{align}
with $\frac{2\alpha}{2+ \alpha} \leq \beta \leq \alpha$ and 
$\epsilon = C_1\left( h^2 \abs{\ln h} \right)^{\frac{1}{2 + \alpha}}
  $ for an arbitrary constant $C_1>0$, then
$$
\inftynorm { u - \ue_h } \leq C \left( h^2 | \ln{h} | 
\right)^{\frac{\alpha}{2 + \alpha}}
\|u\|_{C^{2,\alpha}(\overline{\Omega})} \, ; 
$$ 
see Corollary \ref{convergencerate-C2Holder}. 
Note that, according to \eqref{schauder},
the $C^{2, \alpha}(\overline{\Omega})$ regularity assumption on $u$ is
guaranteed by $A,f \in C^{0,\alpha}(\overline{\Omega})$
and $\dm$ being of class $C^{2,\alpha}$
\cite{Caffarelli88,GilbargTrudinger01}, which is consistent with
\eqref{Ldassumption}.
For $u \in C^{3, \alpha}(\overline{\Omega})$ instead,
we impose $\frac{2+2\alpha}{3+\alpha} \le \beta \le 1$ and
$\epsilon = C_2 h^{\frac{2}{3+\alpha}}$ 
for an arbitrary constant $C_2>0$, to
show in Corollary \ref{convergencerate-C3Holder} that
$$
\inftynorm { u - \ue_h } \leq C h^{\frac{2(1+ \alpha)}{3 + \alpha}}
|\ln{h}| \, \|u\|_{C^{3,\alpha}(\overline{\Omega})}\, .
$$
We stress that for $\alpha = 1$, we obtain a nearly linear decay rate 
$
\inftynorm{ u - \ue_h } \leq h |\ln h|,
$
which turns out to be optimal for our method.

We further allow $u$ to be piecewise $C^{2,\alpha}$ in a
collection of disjoint subdomains $\{\dm_j\}_{j=1}^J$ with Lipschitz
boundaries $\partial\dm_j$. We exploit that \eqref{intro:ABP} measures
operator consistency in a discrete $L^d$-norm to set
$\epsilon = C_3 \left( h^2 \abs{\ln h}\right)^{\frac{d}{1+2d}}$ and show
\[
\inftynorm {u - \ue_h} \leq 
C \big( h^{2} |\ln h| \big)^{\frac{1}{1+2d}}
\]
in Corollary \ref{convergencerate-C11}
without requiring that $\partial\dm_j$ aligns with the mesh $\Th$. This
accounts for a special but important class of discontinuous coefficients
\cite{Kim07,Lorenzi72}.

Our two-scale FEM is a compromise between the fine scale accuracy
provided by the discrete Laplace operator $\Delta_h$ and the
monotonicity and stability achieved at the coarse scale $\epsilon$ by
the integral operator $I_\epsilon$ in \eqref{Ie}. This also explains
why the geometric mesh restriction of weak acuteness
is unrelated to the coefficient matrix $A$ but to the identity: it guarantees monotonicity of $\Delta_h$! 
The coarser scale enhances the stability of \eqref{intro:FEM} at the cost of
additional coarser scale error which reduces the fine scale accuracy;
this is somewhat related to wide stencil techniques
\cite{BonnansZidani03,CamilliFalcone95,DebrabantJakobsen13}.
The enhanced stability enables us to establish $L^{\infty}$ estimates
based on the ABP maximum principle, instead of variational techniques
as in \cite{FengHenningsNeilan, SmearsSuli13, WangWang}.
Our method requires regularity of $u$ beyond $C^2(\overline{\Omega})$
whereas those in \cite{FengHenningsNeilan, SmearsSuli13, WangWang}
require regularity beyond $H^2(\Omega)$. It is worth stressing that,
due to the structure of the ABP estimate, such a regularity assumption
is only required to hold piecewise with discontinuities of the Hessian
$D^2u$ of $u$ not necessarily aligned with the mesh.

The rest of this paper is organized as follows. 
In Section \ref{S:approximation}, we describe the approximation \eqref{integro-differential-operator} of \eqref{pde} proposed by L. Caffarelli and L. Silvestre \cite{CaffarelliSilvestre10}. 
We introduce finite element methods and show the monotonicity property 
in Section \ref{S:FEM}. 
We next discuss the classical ABP estimate in Section \ref{S:ABP} and
apply it to derive the error estimate $\inftynorm{ u - \ue} \leq C
\epsilon^{\alpha}$ provided $u \in C^{2, \alpha}(\overline{\dm})$, where $\ue$ is the
solution of the integro-differential equation \eqref{integro-differential-operator}.
In Section
\ref{S:discrete-ABP}, we prove our discrete Alexandroff estimate, which
has some intrinsic interest and is instrumental to derive convergence
rates for the Monge-Amp\`ere
equation \cite{NochettoZhang16}. We also derive our discrete
ABP estimate, which hinges only on the mesh $\Th$ being face weakly acute. 
Utilizing the discrete ABP estimate, we establish several rates of convergence depending on solution and data regularity in Section \ref{S:error-estimate}. 
We conclude in Section \ref{S:numerics} with numerical experiments which
explore properties and limitations of our FEM.

%
       \section{Approximation of uniformly elliptic equations} \label{S:approximation}
%

In this section, we discuss the approximation proposed by
L. Caffarelli and L. Silvestre in \cite{CaffarelliSilvestre10}
for the linear elliptic PDE in non-divergence form 
\eqref{pde} by the integro-differential equation \eqref{integro-differential-operator}. 
We also propose a modification of the second difference $\du x y$ near
$\bdry$ which avoids extending the functions outside $\Omega$. 

%
                      \subsection{Integro-differential equation}
%

Let $\varphi$ be a radially symmetric function with compact support in
the unit ball and $\int_{\mathbb{R}^d} \abs x ^2 \varphi (x) = d$.
Given a continuous function $u$, we let $\Ie u$ be the integral transform
\begin{equation}\label{integral}
\Ie u(x) := 
\int_{\mathbb{R}^d} \du x y K_\epsilon(x,y) dy ,
\end{equation}
where the kernel  
\begin{equation}\label{kernel}
  K_\epsilon(x,y) :=  \frac{1}{\epsilon ^ {d+2} \text{det}(M(x))} \varphi \left( \frac { M^{-1}(x) y} {\epsilon} \right)
\end{equation}
has support contained in the ball $B_{Q \epsilon} (0)$ with radius
$Q\epsilon$ where $Q = (\Lambda - \frac 12 \lambda)^{1/2}$.
If $u$ is just defined in $\dm$, then the integral in \eqref{integral}
is problematic for values of $x$ close to $\bdry$ unless $u$ is
suitably extended outside $\dm$; an $H^1$ extension is used in
\cite{CaffarelliSilvestre10} which restricts the order of
accuracy. Our goal is to avoid an extension by suitably modifying the
definition of $\du x y$ for $x$ near $\bdry$ and at the same time
retain exactness for quadratic polynomials.
To this end, we denote the region bounded away from the boundary by
\begin{equation}\label{Omega-eps}
\dme = \{ x \in \dm: \;{\rm dist}(x , \bdry) > Q \epsilon   \},
\end{equation}
and note that the $\du x y$ is well defined only for $x \in \dme$. 
If the line connecting $x$ with either $x+y$ or $x-y$ is not contained in the domain $\dm$, let $\theta \in (0, 1)$ be the largest number such that $x \pm \theta y \in \dm$ for all $y \in B_{Q\epsilon}$, define 
\begin{align}\label{extension}
\du x y :=  \; \frac{ u(x - \theta y) + u(x + \theta y) - 2 u(x)}{\theta^2 },
\end{align}
and note that $\du x y = D^2 u(x) : (y \otimes y)$
provided $u$ is a quadratic polynomial.

We now approximate the equation \eqref{pde} by the integro-differential equation
\begin{align}
\label{pde_approx}
\Le \ue (x) = \frac{\lambda}{2} \laplace \ue (x) + \Ie \ue (x) = f(x) \qquad \text{ in $\dm$ } .
\end{align}
We refer to \cite{CaffarelliSilvestre10} for details about  
the existence and uniqueness of solution $u_\epsilon$.

%
\subsection {Rate of convergence of integral transform $\Ie u(x)$} 
%

The convergence rate of $\Ie u(x)$ depends on the regularity of the
function $u$, and is established below. 

\begin{lemma}[approximation property of $\Ie$]\label{approximation}
Let $\Ie u (x)$ be the integral transform defined by
\eqref{integral}-\eqref{extension} with $M = M(x)$ given in \eqref{M},
and let $U_{Q\epsilon}(x) := \overline{B}_{Q\epsilon}(x)\cap\overline{\Omega}$.
\begin{enumerate}
\item
If $u \in C^2(\overline{\dm})$, then
$
\Ie u(x) \rightarrow (A(x) - \frac{\lambda}{2}I) : D^2 u (x)
$
as $\epsilon \to 0$ for all $x \in \Omega$.
\item
If ${\rm dist}( x , \bdry ) \leq Q \epsilon$ and $ u \in C^{2, \alpha} 
(U_{Q\epsilon} (x) ) $ for $0 < \alpha \leq 1$, then
\begin{align*}
  \Big| \Ie u(x) - \Big(A(x) - \frac{\lambda}{2}I\Big) : D^2 u (x) \Big|
  \leq C |u|_{C^{2,\alpha}(U_{Q\epsilon}(x))}\theta^\alpha\epsilon^{\alpha}. 
\end{align*}
\item
If $x \in \dme$ and $ u \in C^{2+k,\alpha}(U_{Q\epsilon}(x))$
for $k = 0, 1$ and $0 < \alpha\leq 1$, then
\begin{align*}
\Big| \Ie u(x) - \Big(A(x) - \frac{\lambda}{2}I \Big) : D^2 u (x) \Big| 
\leq C |u|_{C^{2+k,\alpha}(U_{Q\epsilon}(x))}\epsilon^{k+ \alpha} .
\end{align*}
\end{enumerate}
\end{lemma}

\begin{proof}
We recall that $I_\epsilon u$ is exact if $u$ is quadratic,
namely \eqref{quadratic} holds.

{\bf Case (2).} If dist $(x,\partial\Omega) \le Q\epsilon$,
  then we have 
\[
u(x+ \theta y) - u(x) = \theta \abs{y} \int_0^1 D_y u(x + s \theta y) ds,
\]
where 
$
D_y u = \abs{y}^{-1} y \cdot \gradv u.
$
Hence, 
\[
\du x y 
= 
\frac { \abs y} {\theta} \int_0^1
\Big( D_y u(x + s \theta y) - D_y u(x - s \theta y) \Big) ds.
\]
Upon adding and substracting $D_y u(x)$, we obtain
\begin{equation}\label{differenceformula}
\du x y 
= \; 
 \abs y ^2 \int_0^1 \int_{0}^1 
s\Big( D^2_{yy} u(x + s t \theta y) + D^2_{yy} 
u(x - s t \theta y) \Big) dt ds .
\end{equation}
Using the following property, shown earlier for \eqref{quadratic},  
\begin{equation}\label{M2-D2}
\begin{aligned}
M(x)^2 : D^2 u (x) 
= & \int_{\mathbb{R}^d} 
\abs y ^2 D^2_{yy} u(x)  K_{\epsilon}(x,y) dy
\\
= & \int_{\mathbb{R}^d}
2 \abs y ^2 \int_0^1 \int_{0}^1 s \; D^2_{yy} u(x) dt ds \; K_{\epsilon}(x,y) dy
\end{aligned}
\end{equation}
and the H\"{o}lder continuity of $D^2_{yy} u$ in $U_{Q\epsilon}(x)$
\begin{align*}
\Abs { D^2_{yy} u(x + s t \theta y) -  D^2_{yy} u(x) }
\leq 
 |u|_{C^{2, \alpha}(U_{Q\epsilon}(x))} \theta^{\alpha} \abs y ^{\alpha},
\end{align*}
we deduce
\begin{align*}
\Abs { \Ie u(x) - M(x)^2 : D^2 u (x) } 
\leq &\;
|u|_{C^{2, \alpha}(U_{Q\epsilon}(x))} \int_{\mathbb{R}^d} \abs y ^{2+\alpha}
K_{\epsilon}(x,y) dy
\\
\leq &\;
C |u|_{C^{2, \alpha}(U_{Q\epsilon}(x))} \theta^{\alpha} \epsilon^{\alpha} .
\end{align*}

{\bf Case (3).} If dist $(x,\partial\Omega)>Q\epsilon$, then we take
$\theta = 1$ in \eqref{extension} and rewrite it as follows
\begin{align*}
\du x y  = &\;
\abs y \int_0^1 \Big( D_y u(x+ sy) -  D_y u(x- sy) \Big) ds
\\
= &\; \abs y ^2 \int_0^1 \int_{0}^1 s\Big(  D^2_{yy} u(x + s t y) 
+ D^2_{yy} u(x - s t y) \Big) dt ds ,
\end{align*}
upon adding and subtracting $D^2u(x)$.
In view of \eqref{M2-D2} with $\theta=1$ and
\[
\Abs { D^2_{yy} u(x + s t y) + D^2_{yy} u(x - s t y) - 2 D^2_{yy} u(x) } 
\leq 
2|u|_{C^{2 + k, \alpha}(U_{Q\epsilon}(x) )} \abs y ^{k + \alpha},
\]
we deduce
\begin{align*}
\Abs{ \Ie u(x) - M(x)^2 : D^2 u (x) }
\leq &\;
|u|_{C^{2 + k, \alpha}(U_{Q\epsilon}(x) )}
\int_{\mathbb{R}^d}\abs y ^{2 + k + \alpha} K_{\epsilon}(x,y) dy
\\
\leq &\;
C |u|_{C^{2+k,\alpha}(U_{Q\epsilon}(x))} \epsilon ^{k+\alpha}. 
\end{align*}

{\bf Case (1).}
Note that if $u \in C^2(\overline{\dm})$, then 
\[
\abs{ D^2_{yy} u(x + s t y) + D^2_{yy} u(x - s t y) - 2 D^2_{yy} u(x) } \to 0 
\]
as $\abs{y} \to 0$. This implies
$\Abs { \Ie u(x) - M(x)^2 : D^2 u (x) } \to 0$
as $\epsilon\to0$ for all $x\in\Omega$,
and completes the proof.
\end{proof}

%
   \section{Finite element method for the integro-differential problem}\label{S:FEM}
%
In this section, we introduce a finite element method for
\eqref{pde_approx} and show that the method is monotone 
provided that the mesh $\Th$ is weakly acute (see
\eqref{weaklyacute}).

We start with some notation.
Let $\Th =\{K\}$ be a conforming, quasi-uniform and shape-regular
partition of $\dm$ into simplices $K$ with shape
regularity constant $\sigma$. The latter is a bound for the ratio between 
the diameter of any element $K\in\Th$ and the diameter of the largest ball
inscribed in $K$.

Let $\Fh$ be the set of faces,
or equivalently of interior $(d-1)$-dimensional simplices of
$\Th$, and $\Nh$ be the set of interior nodes of $\Th$.

Let $\Vh$ be the space of continuous piecewise affine functions relative to
$\Th$, and $\Vhzero$ be its subspace with vanishing trace
\[
\Vh := \{ v \in C(\overline{\dm}) : \; v |_{K} \text{ is affine for all $K\in\Th$} \},
\quad
\Vhzero := \{ v \in \Vh : \; v|_{\bdry} = 0  \}.
\]
Given $x_i \in \Nh$, let $\phi_i$ be its hat function 
and $\omega_i = \text{supp} \;\phi_i$ be its star.

\subsection{Finite element method}
We seek a solution $\ue_h \in \Vhzero$ satisfying
\begin{align}\label{FEM}
  - \frac{\lambda}{2} \int_{\dm} \gradv \ue_h \cdot \gradv \phi_i  + \Ie \ue_h(x_i) \int_{\dm} \phi_i  = \int_{\dm} f \, \phi_i
\end{align}
for all nodes $x_i \in \Nh$, or equivalently
\begin{equation}\label{discrete_pde}
  L_h^{\epsilon} u_h^\epsilon(x_i) = \frac{\lambda}{2} \laplace_h
  \ue_h (x_i) + \Ie \ue_h (x_i) = f_i = \frac{\int_{\dm} f \phi_i }{\int_{\dm} \phi_i},
\end{equation}
where the discrete Laplacian is defined in \eqref{lack-consistency}.
We define $\Ie$ as in Section \ref{S:approximation}, namely
\begin{align}
\label{FEMintegral}
\Ie \ue_h(x_i)  = \int_{\mathbb{R}^d} \frac{\delta \ue_h (x_i,
  y)}{\epsilon ^ {d+2} {\rm det}(M(x_i))} \, \varphi \left( \frac { M^{-1}(x_i) y} {\epsilon} \right) dy 
\end{align}
where $M(x_i) = \left( \bar{A}(x_i) - \frac{\lambda}{2} I \right
)^{1/2}$. If $A(x) \in C(\dm)$, then $\bar{A}(x_i)=A(x_i)$
is well defined at every node $x_i$. Otherwise, we let
$\bar{A}(x_i)$ be the meanvalue of $A$ over $\omega_i$:
\begin{align*}
\bar{A}(x_i) = \frac{1}{\abs{\omega_i}} \int_{\omega_i} A(x) \, dx .
\end{align*}

We emphasize that the discrete formulation \eqref{FEM}
above is not obtained by simply testing \eqref{pde_approx} with a hat
function $\phi_i$,
which would lead to a term $\int_{\dm} \Ie \ue_h(x) \phi_i(x) \, dx$
instead of $\Ie \ue_h(x_i) \int_{\dm} \phi_i \, dx$.
This quadrature (mass lumping) preserves monotonicity, which
plays a crucial role in establishing the ABP estimate and the a priori
error estimates, and is much easier to implement since we only need to
evaluate $\Ie \ue_h(x_i)$ at every node $x_i$.
We deal with monotonicity in subsection \eqref{SS:dmp} and with the
computation of $\Ie \ue_h(x_i)$ in subsection \ref{subsec:quadrature}.

\subsection{Quadrature}\label{subsec:quadrature}
We briefly discuss the effect of quadrature in computing $\Ie u_h^\epsilon(x_i)$,
which renders our method fully practical. The change of
variables $z = M^{-1}(x)y/\epsilon$ yields 
\[
I_{\epsilon} u^{\epsilon}_h (x) = \int_{B_1(0)}
\frac{ \delta u_h^\epsilon (x_i, \epsilon M(x_i)z)} {\epsilon^2} \,
\varphi(z) \, dz ,
\]
where $B_1(0)$ is the unit ball in $\mathbb{R}^d$.
We thus define the quadrature formula
\[
\Qe u_h^\epsilon(x_i) := \sum_{j=1}^J
\frac{\delta u_h^\epsilon (x_i,\epsilon M(x_i) z_j)} {\epsilon^2} \varphi(z_j) w_j 
\qquad\textrm{for all } x_i\in\Nh,
\]
where the node-weight pairs $(z_j,w_j)_{j=1}^J$
satisfy the following properties \cite{Stroud71}
%
\squishlist
\item {\it local support}:   $z_j \in B_1(0)$ for all quadrature points $z_j$;
  \item {\it consistency}: $\Qe p(x_i) = \Ie p(x_i)$
    for all quadratic polynomials $p$ and $x_i\in\Nh$;
  \item {\it positivity}: $w_j > 0$ for all quadrature weights $w_j$.
\squishend
Finally, it is easy to check that operator consistency holds provided that
\[
\sum_{j=1}^J z_j \otimes z_j \varphi(z_j) \, w_j = I. 
\]

\subsection{Mesh weak acuteness and discrete maximum principle}\label{SS:dmp}
We start by recalling the definition \eqref{lack-consistency} of 
discrete Laplace operator at each node $x_i \in \Nh$, and rewrite it
upon integrating by parts elementwise
\begin{align*}
\laplace_h \ue_h (x_i) = &\; \left ( 
\int_{\dm} \phi_i \right )^{-1}
\sum_{F \;\ni x_i } \int_F J_F (\ue_h)  \phi_i ,
\end{align*}
where
\[
  J_F (\ue_h) := - n_F^+ \cdot \gradv \ue_h|_{K^+} - n_F^- \cdot \gradv \ue_h|_{K^-}
\]
denotes the jump of $\nabla\ue_h$ across the face $F\in\Fh$, $K^\pm\in\Th$ denote
the two elements sharing the face $F$ and $n_F^\pm$ the outer unit
normal vectors of $K^\pm$ on $F$.
We point out that $J_F (\ue_h)$ is the opposite of the usual jump
because it corresponds to $\laplace_h \ue_h$ rather than $-\laplace_h \ue_h$.
Since $J_F (\ue_h)$ is constant and 
\[
\int_{\dm} \phi_i = \frac{\abs{ \omega_i }}{d+1} 
\quad \text{ and } \quad 
\int_{F} \phi_i  = \frac{\abs{ F }}{d} ,
\]
provided $x_i\in F$, we get the following expression for the discrete Laplacian
\begin{align}\label{discretelaplace}
  \laplace_h \ue_h(x_i) = \frac{d+1}{d} \sum_{F\in\Fh: \, x_i\in F}
  \frac{\abs F}{\abs {\omega_i}} J_F(\ue_h).
\end{align}

We now impose restrictions on the geometry of meshes.
We say that the mesh $\Th$ is {\it weakly acute with respect to
faces}, or {\it face weakly acute} for short, if
\begin{align}\label{localweaklyacute}
\int_{\omega_F} \gradv \phi_i \cdot \gradv \phi_j \leq 0 
\quad \text{ for all $i \neq j$ and all faces $F$}
\end{align}
where $\omega_F := \cup \{ K^{\pm}\in\Th: \, F \subset K^{\pm} \}$.
We say that $\Th$ is {\it weakly acute} \cite{Ciarlet91} if
\begin{align}\label{weaklyacute}
k_{ij} := \int_{\dm} \gradv \phi_i \cdot \gradv \phi_j \leq 0 
\quad \text{ for all $i \neq j$.}
\end{align}
This condition is equivalent to \eqref{localweaklyacute} for
$d=2$ and is valid if and only if
the sum of the two angles opposite to a face (or edge) is no greater than
$\pi$ \cite{CiarletRaviart73, NochettoPaoliniVerdi91}.
On the other hand, \eqref{weaklyacute} is weaker than
\eqref{localweaklyacute} for $d>2$ because the former is obtained upon
adding the latter over all faces $F$ containing the segment that connects
nodes $x_i$ and $x_j$.
For $d=3$, the property that internal dihedral angles of tetrahedra
does not exceed $\pi/2$ implies \eqref{localweaklyacute}; we refer to
\cite{Bartels05, KorotovKrizekNeittaanmaki00}.

It is well known that monotonicity of piecewise linear
finite element methods for the Laplace equation hinges on \eqref{weaklyacute}.
We are now ready to discuss monotonicity of the
discrete operator $L_h^\epsilon$ in \eqref{discrete_pde}.

\begin{lemma}[monotonicity property of $L_h^{\epsilon}$]\label{Monotonicity}
Let $v_h$ and $w_h$ be two functions in $\Vh$, and $v_h \leq w_h$ in $\dm$ with equality attained at some node $x_i \in \Nh$. 
Then the integral operator $\Ie$ satisfies the monotonicity property
\[ 
\Ie v_h (x_i) \leq \Ie w_h (x_i).
\]
In addition, if the mesh $\Th$ satisfies \eqref{weaklyacute},
then the discrete Laplacian $\laplace_h$ satisfies the monotonicity property
\[
\laplace_h v_h(x_i)
\leq 
\laplace_h w_h(x_i) ,
\]
whence
\[
L_h^{\epsilon} v_h (x_i) \leq L_h^{\epsilon} w_h (x_i).
\]
\end{lemma}
\begin{proof}
To show the monotonicity property of $\Ie$, we note that the assumptions $v_h \leq w_h$ and $v_h(x_i) = w_h(x_i)$ imply
\[
\delta v_h (x_i, y) \leq \delta w_h (x_i, y) .
\]
The first assertion follows from the definition \eqref{integral}
of $\Ie$  and the fact $K_\epsilon(x,y)\ge0$.

On the other hand, following \cite{Ciarlet91}, we realize that 
\begin{equation}\label{ciarlet}
  - \int_{\dm} \gradv (w_h - v_h) \cdot \gradv \phi_i  = 
- \sum_{j} \big(w_h(x_j) - v_h(x_j) \big) k_{ij} \geq 0,
\end{equation}
because $k_{ij}\le0$ for $i\ne j$.
Invoking the definition \eqref{lack-consistency} of $\laplace_h$ yields
\[
\laplace_h(w_h - v_h)(x_i) \geq 0.
\]
This proves the second inequality. 
Finally, the last assertion follows from \eqref{discrete_pde}.
\end{proof}

It is worth stressing that the monotonicity property of $L_h^\epsilon$
relies solely on \eqref{weaklyacute} and is thus valid for all
matrices $A(x)$ regardless of possible anisotropies.
We mention two important consequences of the
monotonicity property: the discrete maximum principle
and the unique solvability of \eqref{discrete_pde}.
The proof of the former, as well as that of Lemma
\ref{Monotonicity}, extends to the quadrature
described in subsection \ref{subsec:quadrature} and requires no a
priori relation between the two scales $\epsilon$ and $h$.
\begin{corollary}[discrete maximum principle]\label{dmp}
Let $\Th$ satisfy \eqref{weaklyacute}. If 
  $L_h^\epsilon w_h (x_i) \geq 0 $ for all $x_i \in \Nh$,
  and $w_h \le 0$ on the boundary $\bdry$,
then $w_h \leq 0$ in $\dm$.
\end{corollary}
\begin{proof}
Given $\gamma>0$ arbitrary, we argue with the auxiliary function
$v_h := w_h + \gamma I_h\psi$, where $\psi(x) = |x|^2 - \alpha$ and
$\alpha=\alpha(\dm)>0$ is so large that $\psi\le0$ on $\partial\dm$.
Upon subtracting a linear function tangent to $\psi(x)$ at $x_i \in\Nh$,
whose discrete Laplacian
vanishes, we can assume that $\psi$ attains a minimum at $x_i$, namely
$\psi(x) = |x-x_i|^2 - \alpha$. Employing \eqref{ciarlet} to
compare $I_h\psi$ with the constant function $-\alpha$, we deduce
\[
\Delta I_h\psi(x_i) \int_\dm \phi_i = 
- \int_{\dm} \gradv I_h\psi \cdot \gradv \phi_i =
- \sum_{j\ne i} |x_j - x_j|^2 k_{ij} > 0,
\]
because $\sum_{j\ne i} k_{ij} = - k_{ii} < 0$. In addition, realizing
that $\delta I_h\psi(x_i,y)>0$ for all $y=\epsilon M(x_i) z$ with $z$
in the unit ball, we obtain $I_\epsilon[I_h\psi](x_i) > 0$, whence
$L_h^\epsilon v_h(x_i) > 0$.

Let $x_i$ be a node where $v_h$ attains an absolute positive maximum. Such a
node $x_i\in\Nh$ must be interior because $v_h\le0$ on $\partial\dm$.
Applying Lemma \ref{Monotonicity} to compare $v_h$ with the constant
function $w_h=v_h(x_i)$ we infer that $L_h^\epsilon v_h (x_i) \le 0$,
which contradicts the preceding statement. This implies $v_h \le 0$ in
$\dm$, or equivalently
\[
w_h \le -\gamma I_h\psi \le \gamma\alpha
\quad\textrm{in } \dm.
\]
Taking the limit as $\gamma\to0$ yields the asserted inequality.
\end{proof}

\begin{corollary}[uniqueness]\label{uniquesolvence}
If the mesh $\Th$ satisfies \eqref{weaklyacute},
then the linear system \eqref{discrete_pde} has a unique solution. 
\end{corollary}
\begin{proof}
  Since the equation \eqref{discrete_pde} is a square linear system, we only need to show that $\ue_h = 0$ if $f = 0$ in $\dm$. 
This statement is a direct consequence of Corollary \ref{dmp}
(discrete maximum principle).
\end{proof}

A third important consequence of the monotonicity property is the
discrete ABP estimate, which relies on \eqref{localweaklyacute}
rather than \eqref{weaklyacute} and is discussed in Section \ref{S:discrete-ABP}.
We first review its continuous counterpart in Section \ref{S:ABP}.

%
 \section{The Alexandroff-Bakelman-Pucci estimate}\label{S:ABP}
%
We start with the definition of convex envelope and sub-differential of continuous functions which is frequently used in the analysis of fully nonlinear elliptic PDEs.

\subsection{Convex envelope and sub-differential}
Let the domain $\dm$ be compactly contained in a ball $B_R$ of radius $R$ and 
$v \in C (\overline{\dm})$ with $v \geq 0$ on $\bdry$. 
Since the negative part $v^-$ of $v$ vanishes on $\bdry$, we extend
$v^-$ continuously by zero to $B_{R} \setminus \dm$.
We define, with some abuse of notation, the convex envelope of $v$ in $B_{R}$ by
\begin{align}\label{convexenvelop}
\Gamma (v) (x) := \sup_{L} \{ L(x) : L \leq -v^- \text{ in $B_{R}$, $L$
  is affine} \}\quad\forall \; x\in B_R.
\end{align}
Obviously, $\Gamma (v)$ is a convex function and
$\Gamma (v) \leq -v^- \le v$ in $\dm$. Moreover, $\Gamma (v) = 0$ on
$\partial B_R$ because dist $(\partial\Omega,\partial B_R)>0$.
In fact, for every $x\in\partial B_R$ there exists an affine function $L$ such that 
$L(z) \leq -v^- (z)$ for all  $z\in\dm$ and $L(x) = 0$, whence $\Gamma (v)(x)=0$.
The set
\begin{align}
  \label{contactset}
  \mathcal{C}^-(v) := \{ x \in B_R  : \; \Gamma (v)(x) = v(x) \} 
\end{align}
is called {\it (lower) contact set} of $v$. 
We may assume that $\mathcal{C}^-(v) \subset \dm$ unless $\Gamma (v)=0$.
In fact, if $\Gamma (v)(x) = -v^-(x)= 0$ for
some $x \in B_{R}\setminus \dm$, then the convexity of 
$\Gamma(v)$ and $\Gamma (v) = 0$ on $\partial B_R$ implies
$\Gamma (v) = 0$ in $\dm$.

Since $\Gamma (v)$ is convex its {\it subdifferential} $\gradv\Gamma (v) (x_0)$
is nonempty for all $x_0\in B_R$
\begin{equation}\label{sub-diff}
\gradv \Gamma (v) (x_0) := \{  w \in \mathbb{R}^d: \;  \innerpro w {
  x -  x_0 } + \Gamma (v)(  x_0) \leq \Gamma (v)(x) \text{ for all $ x
  \in B_{R}$} \} ,
\end{equation}
where $\innerpro \cdot \cdot$ denotes the dot product in $\mathbb R^d$.
In particular, if $x_0\in\mathcal{C}^-(v)$, then
\[
\innerpro w {x -  x_0 } + v(  x_0) \leq v(x) \quad\forall \;
w \in \gradv \Gamma (v) (x_0), \; x\in B_R.
\]

%
\subsection{Alexandroff-Bakelman-Pucci estimate and applications}
%

The classical ABP estimate is the cornerstone in the regularity theory of fully nonlinear elliptic equations. 
The estimate gives a bound for the $L^{\infty}(\dm)$-norm of 
the negative part $u^-$ of the solution $u$ to equation 
\eqref{pde} in terms of the $L^d$-norm of $f$:
\begin{align*}
  \sup_{\dm} u^- \leq C \left ( \int_{\mathcal{C}^-(u)} |f|^d \right)^{1/d} ,
\end{align*}
where $\mathcal{C}^-(u)$ is the lower contact set of $u$ in $B_{R}$ 
defined in \eqref{contactset} and $C = C(d, \lambda, \Omega)$.
We complement the ABP estimate with a modified version at the $\epsilon$-scale \cite{CaffarelliSilvestre10}.

\begin{lemma}[ABP estimate at $\epsilon$-scale \cite{CaffarelliSilvestre10}]
    \label{coarse_ABP}
If $\ue$ is a solution of \eqref{pde_approx} with $\ue \geq 0$ on the boundary $\bdry$. Then 
\[
  \sup_{\dm} (\ue)^- \leq C \left ( \int_{\mathcal{C}^-(\ue)} |f|^d \right)^{1/d} ,
\]
where $\mathcal{C}^-(\ue)$ is defined in \eqref{contactset}
and $C = C(d, \lambda, \dm)$.
\end{lemma}
We now apply Lemma \ref{coarse_ABP} to establish a rate of convergence for $\inftynorm{u - \ue}$.
\begin{lemma}
  [rate of convergence of $\inftynorm{u - \ue}$]
  \label{rate_of_ue}
 If the solution $u$ of \eqref{pde} satisfies $u \in C^{2,
    \alpha}(\overline{\dm})$ for some $0 < \alpha \leq 1$ and $\ue$ is a solution of \eqref{pde_approx}, then
there exists $C=C(d,\lambda,\Omega)$ such that
  \[
    \inftynorm{u - \ue} \leq C \epsilon^{\alpha} 
  |u|_{C^{2,\alpha}(\overline{\dm})}.
  \]
\end{lemma}
\begin{proof}
  We only need to establish a bound for the negative part of $u - \ue$ such as
  \begin{align}\label{eq:lower-bound}
\sup_{\dm} (u - \ue)^- \leq C \epsilon^{\alpha},
\end{align}
because the bound for the positive part is similar.
By Lemma \ref{approximation} (approximation property) of $\Ie$,  we have
\[
\Abs{\Le u(x) - A(x) : D^2 u (x)} \leq C \epsilon^{\alpha} 
|u|_{C^{2,\alpha}(\overline{\dm})}
\qquad
\text{ for all $x \in \dm$.}
\]
Thanks to \eqref{pde} and \eqref{pde_approx}, a simple comparison between $\Le u$ with $\Le \ue$ yields 
\[
  \Abs{\Le (u- \ue)(x)} \leq C \epsilon^{\alpha} 
  |u|_{C^{2,\alpha}(\overline{\dm})}.
\]
Invoking Lemma \ref{coarse_ABP} (ABP estimate at $\epsilon$-scale), we 
readily obtain \eqref{eq:lower-bound}.
\end{proof}

%
 \section{Discrete Alexandroff-Bakelman-Pucci estimate}\label{S:discrete-ABP}
%
The aim of this section is to establish Theorem \ref{discrete_ABP}
(discrete ABP estimate).
This and related results are of intrinsic interest and do not require
regularity of the domain $\dm$, which is just assumed to be bounded in
this section.
We recall that a discrete ABP estimate is also proved in  
\cite{KuoTrudinger00} for finite differences on general meshes 
within the abstract framework of \cite{KuoTrudinger92}. However, 
when applied to our finite element method, the estimate in
\cite{KuoTrudinger00} yields sub-optimal results because it replaces 
the measure of star $\abs {\omega_i}$ in \eqref{intro:ABP} by the much 
larger quantity $\abs{B_{\epsilon}(x_i)}$, where $B_{\epsilon}(x_i)$
stand for the set of influence of $x_i$ which, according to
\eqref{Ie}, is of size $\epsilon\gg h$.
We present a novel proof which is more geometric and suitable for
FEM. It is based on the geometric characterization of the 
sub-differential of piecewise linear functions $v_h\in\Vh$ and control of its
measure by the jumps of $\nabla v_h$.

First, we need
a definition. Given $v_h\in\Vh$ with $v_h\ge0$ on $\partial\Omega$, 
we observe that if $x$ belongs to the interior of some 
element $K\in\Th$ and to the contact set $\mathcal{C}^-(v_h)$, 
then the vertices of $K$ are also
in the contact set. This motivates the following definition of 
{\it (lower) nodal contact set} 
-- the discrete counterpart of \eqref{contactset}:
\begin{equation}\label{discrete-contactset}
\mathcal{C}^-_h(v_h) := \{ z \in \Nh : \Gamma (v_h)(z) = v_h(z) \}
\qquad\forall v_h\in\Vh,
\end{equation}
Therefore, $\mathcal{C}^-_h(v_h)$ is just a collection of nodes and $\mathcal{C}^-_h(v_h) \subset \mathcal{C}^-(v_h) \subset \dm$ unless $\Gamma (v_h)=0$.

\begin{theorem}[discrete Alexandroff-Bakelman-Pucci estimate]\label{discrete_ABP}
Let the mesh $\Th$ be shape regular and satisfy \eqref{localweaklyacute}.
Let $v_h \in \Vh$ 
with $v_h \geq 0$ on $\bdry$ satisfy 
\[
L_h^\epsilon v_h(x_i) \leq f_i 
\quad
\text{ for all $x_i \in \Nh$.}
\]
If $\mathcal{C}^-_h(v_h)$ is the nodal contact set of
\eqref{discrete-contactset}, then the discrete ABP estimate reads
\[
\sup_{\dm} v_h^- \leq C \left( \sum_{x_i \in \mathcal{C}_h^-(v_h)} \abs {f_i^+}^d \abs {\omega_i} \right)^{1/d},
\]
where the constant $C = C(\sigma, d, \lambda, \dm)$ and $\abs {\omega_i}$ denotes the volume of the star $\omega_i$.
\end{theorem}

\subsection{Local convex envelope of piecewise affine functions}\label{S:lce}
%
There are two critical issues in dealing with $\Gamma (v_h)$:
first $\Gamma (v_h)$ is not
locally defined and second $\Gamma (v_h)$ is not a piecewise affine
function subordinate to $\Th$.
To overcome the first issue, we define
the {\it local convex envelope} for any $z \in \mathcal{C}^-_h(v_h)$
\begin{align}\label{localfunction}
\Gamma_{z}(v_h) (x) = \sup_{L}
\big\{ L(x) : L \leq v_h \text{ in $\omega_z$, $L$ is affine and } L(z) = v_h(z) \big\}
\end{align}
for all $x \in \omega_z$. We wonder whether $\Gamma_{z}(v_h)\in\Vh$
and explore this question next.

\begin{lemma}[local convex envelope for $d=2$]\label{lce-d=2}
The function $\Gamma_{z}(v_h)\in\Vh$ for all $z\in\Nh$ provided $d=2$.
\end{lemma}
\begin{proof}
Given a triangle with vertices $z=0,x_1,x_2$, let $L_1,L_2\le v_h$ in
$\omega_z$ be two affine functions which satisfy, without loss of generality,
\[
L_1(z)=L_2(z)=v_h(z)=0,\quad
L_1(x_1) > L_2(x_1),\quad
L_1(x_2) < L_2(x_2).
\]
Let $L$ be the affine function which agrees with $L_1$ at $z,x_1$ and
with $L_2$ at $x_2$. Since $v_h$ is affine in $T$, we deduce
$L\le v_h$ in $T$. On the other hand, $L \le \max\{L_1,L_2\}\le v_h$ in 
$\omega_z\setminus T$ because $y=\lambda_1 x_1 + \lambda_2 x_2 \in\omega_z\setminus T$ 
entails either $\lambda_1<0$ or $\lambda_2<0$ and
\[
L(y) = \lambda_1 L(x_1) + \lambda_2 L(x_2) \le \lambda_1 L_2(x_1) + \lambda_2 L_2(x_2)
= L_2(y) \le \max\{L_1(y),L_2(y)\}
\]
if $\lambda_1<0$ or likewise if $\lambda_2<0$.
This implies that $L$ is an admissible function
in the definition of $\Gamma_{z}(v_h)$, whence $\Gamma_{z}(v_h)$ must
be affine in $T$ as asserted.
\end{proof}  
\begin{remark}[local convex envelope for $d=3$]\label{lce-d=3}
Unfortunately, Lemma \ref{lce-d=2} is false for $d=3$.
To see this, we
construct a counterexample: consider the vertices
\[
z_0 = (0,0,-1),\quad
z_1 = (-1,0,0),\quad
z_2 = (0,1,0), \quad
z_3 = (1,0,0),
\]
and tetrahedra $T_1,T_2$ to be the convex hulls of $z_0,z_1,z_2,z_3$ and
$z_0,z_1,-z_2,z_3$. If $v_h$ is piecewise affine with nodal values
$v_h(z_0) = -1, v_h(z_1)=v_h(z_3)=0$ and $v_h(\pm z_2)=-1$, then
the local convex envelope
$\Gamma_{z_0}(v_h)(x)=|x_1|-1$ is not affine in each $T_i$ for $i=1,2$.
\end{remark}
In view of \eqref{localfunction}, we let the {\em local
sub-differential} $\nabla \Gamma_z(v_h)(z)$ at $z \in\mathcal{C}^-_h(v_h)$ be
\looseness=-1
\[
\gradv \Gamma_{z}(v_h) ( z) := \big\{ w \in \mathbb{R}^d: \; \innerpro w {x -  z}
+ \Gamma_{z}(v_h)( z) \leq \Gamma_{z}(v_h)(x) \text{ for all $ x \in \omega_z$ } \big\} .
\]
Comparing with definition \eqref{sub-diff} we immediately deduce the
key property
\begin{align}\label{subdifferential:eq-1}
  \gradv \Gamma (v_h) (z) \subset \gradv \Gamma_{z} (v_h) (z)
  \qquad \forall \, z\in\mathcal{C}^-_h(v_h),
\end{align}
which will be instrumental in the subsequent derivation. In fact, all
  statements involving $\nabla\Gamma(v_h)(z)$ will be proved using
$\nabla\Gamma_{z}(v_h)(z)$ for $z\in\mathcal{C}^-_h(v_h)$ instead.
\subsection{Discrete Alexandroff estimate}
The next Alexandroff estimate for a continuous piecewise affine
function $v_h$ states that the $L^{\infty}$-norm of $v_h$ is
controlled by the Lebesgue measure of the sub-differential of its
convex envelope.

\begin{proposition}[discrete Alexandroff estimate]\label{Alexandroff}
Let $v_h \in \Vh$ with $v_h \geq 0$ on $\bdry$, and $\Gamma
  (v_h)$ be its convex envelope in $B_R$. Then
\begin{align}\label{alex}
  \sup_{\dm} v_h^- \leq C \left( \sum_{x_i \in \mathcal{C}^-_h(v_h) } \abs {\gradv \Gamma (v_h) (x_i)} \right)^{1/d},
\end{align}
where $\abs{ {\gradv \Gamma (v_h) (x_i)} }$ denotes the $d$-Lebesgue
measure of the sub-differential of $\Gamma (v_h)$ associated with the
contact node $x_i\in\mathcal{C}^-_h(v_h)$ and $C = C(d, \Omega)$. 
\end{proposition}
\begin{proof}
We proceed in four steps as follows.

\smallskip
{\bf Step 1.}
We first show that
\[
  \sup_{B_{R}} v_h^- = \sup_{B_{R}} \Gamma (v_h)^-. 
\]
Since $v_h \geq \Gamma (v_h)$, the inequality 
$
\sup_{B_{R}} v_h^- \leq \sup_{B_{R}} \Gamma (v_h)^-
$ 
is obvious. To show the reversed inequality, let 
$
\sup_{B_R}v_h^- = v_h^-(x^*)
$
for some $x^* \in B_{R}$ and let $L$ be a horizontal hyperplane touching
$v_h$ from below at $x^*$. By \eqref{convexenvelop} (definition
of convex envelope) again, we deduce 
\[
  \Gamma (v_h) (x) \geq L(x) = L(x^*) = v_h(x^*) 
  \quad \text{ for all $x \in \dm$,}
\]
whence 
$
\sup_{B_{R}} \Gamma (v_h)^- \leq v_h^-(x^*) .
$ 
Hence, to prove \eqref{alex}, we only need to show that
\begin{align*}
 \sup_{B_{R}} \Gamma (v_h)^- \leq C \left( \sum_{ x_i \in \mathcal{C}^-_h(v_h) } \abs {\gradv \Gamma (v_h) (x_i)} \right)^{1/d}.
\end{align*}

{\bf Step 2.} 
We construct a cone $K(x)$ with vertex at $x^*$ such that
\begin{align*}
  K(x^*) = - \sup_{B_{R}} \Gamma (v_h)^- = - M
  \quad \text{ and } \quad
  K(x) = 0 \text{ on $\partial B_{R}$},
\end{align*}
and assume that $M>0$ for otherwise $\Gamma(v_h) = 0$ and \eqref{alex} is trivial in
view of Step 1; thus $K(x)<0$ for all $x\in B_R$.
We note that for any vector $v \in B_{\frac{M}{2R}}(0)$, the affine function 
$
L(x) = -M + \innerpro v {x - x^*}
$
is a supporting plane of $K(x)$ at point $x^*$, namely
$L(x) \le K(x)$ for all $x\in B_R$ and $L(x^*) = K(x^*)$.  This implies
$
  \gradv K(x^*) \supset  B_{\frac{M}{2R}}(0),
$
whence
\[
\abs{\gradv K(x^*)} \geq C(d)\left( \frac{M}{R} \right)^d.
\]

{\bf Step 3.}
We claim that
\begin{equation}\label{cone-subgrad}
  \gradv K(x^*) \subset \cup \big\{ \gradv \Gamma (v_h)(x_i)  :
  x_i \in \mathcal{C}^-_h(v_h) \big\} .
\end{equation}
This is equivalent to showing that for any supporting plane $L(x)$ of
$K(x)$ at $x = x^*$, there is a parallel supporting plane
$\tilde{L}(x)$ for $\Gamma (v_h) (x)$ at some contact node $y$, namely
$y\in\mathcal{C}_h^-(v_h)$.
Consider the function $v_h(x) - L(x)$, and observe that $ v_h(x) \geq 0$ 
on $\partial \dm$ and $v_h(x^*) = K(x^*) = L(x^*)$, whence
\begin{align*}
    v_h(x) - L(x) &\; \geq \; K(x) - L(x) 
    \geq 0 \qquad \text{on $\partial  \dm$},
  \\
    v_h(x^*) - L(x^*) &\; = \; K(x^*) - L(x^*) = 0.
\end{align*}
We infer that 
$
v_h - L 	
$
attains a non-positive minimum inside $\dm$ at $y$. Hence, 
$\widetilde{L}(x) := L(x) + v_h(y) - L(y)$ is a parallel supporting 
plane for $ v_h(x)$ at $y$. Since, according to  \eqref{convexenvelop},
every supporting plane of $v_h$ is a supporting plane of
$\Gamma(v_h)$, we find that
$\widetilde{L}(x) \leq \Gamma (v_h)(x)\leq  v_h(x)$ with equality
at $x=y$. 
The function $v_h - L$, being piecewise affine in $\dm$,
attains its minimum at a node of $\Th$, whence $y \in \mathcal{C}^-_h (v_h)$.

\smallskip
{\bf Step 4.}
Computing Lebesgue measures in \eqref{cone-subgrad} yields
\[
 C(d)\left( \frac{M}{R} \right)^d  \leq 
 \abs{ \gradv K(x^*) } \leq \sum_{x_i \in \mathcal{C}^-_h(v_h) } \abs{ \gradv \Gamma (v_h)(x_i)  }.
\]
Finally, \eqref{alex} follows from a simple algebraic manipulation.
\end{proof}

In view of Proposition \ref{Alexandroff} (discrete Alexandroff estimate) and
\eqref{subdifferential:eq-1}, to prove Theorem \ref{discrete_ABP}
(discrete ABP estimate), we intend to relate $\abs{ {\gradv
    \Gamma_{x_i} (v_h) (x_i)} }$ with the discrete Laplacian at
the contact node $x_i$, namely to show 
\[ 
\Abs{ {\gradv \Gamma_{x_i} (v_h) (x_i)} } \leq C \big( \laplace_h v_h
(x_i) \big) ^d \abs{\omega_i} 
\quad 
\text{ for all $x_i \in \mathcal{C}^-_h(v_h)$,}
\] 
for $C=C(d,\lambda,\Omega,G)$ where $G$ is a geometric constant
defined below in \eqref{const-G}.
This entails estimating $\abs{ {\gradv \Gamma_{x_i} (v_h) (x_i)} }$ 
in terms of the jumps $J_F(\Gamma_{x_i} (v_h))$ across faces $F$
containing $x_i$ according to \eqref{discretelaplace}.
This is precisely our next task.

\subsection{Sub-differential of convex piecewise linear functions}
Given $ x_i \in \mathcal{C}^-_h(v_h)$, let $\{z_j\}_{j=1}^m = \omega_i \cap \Nh$ be the set of nodes connected with $x_i$ and $ \Gamma_{x_{i}} (v_h)$ be the local convex envelope defined in \eqref{localfunction}. 
Without loss of generality, we assume $ x_i = 0$ and $ \Gamma_{x_{i}} (v_h) \geq 0$ in $\omega_{i}$ with equality at node $ x_i$ only. 
We further assume $\Gamma_{x_{i}} (v_h)\in\Vh$ and
simplify the notation in this subsection upon writing
\begin{equation*}\label{gamma}
\gamma( x) := \Gamma_{x_{i}} (v_h)( x),
\quad
\gamma(0)=0,
\quad 
\gradv \gamma ( 0)  := \gradv \Gamma_{x_{i}} (v_h)( x_i),
\quad 
\omega := \omega_i.
\end{equation*}

Our goal in this section is to show the following proposition.
Let $\F(0)$ denote the set of $(d-1)$-dim simplices (faces) containing the origin.
\begin{prop}[estimate of $\abs{ \gradv \gamma ( 0)}$]\label{estimate_of_subdifferential}
Let $\gamma$ be a convex piecewise affine function on a star $\omega$
centered at the origin.
Then there is a constant $C = C(d)$ such that
\[
\abs{ \gradv \gamma ( 0) } \leq C \left( \sum_{F \in \F(0)} J_F(\gamma) \right)^{d} .
\]
\end{prop}
We first point out that the jump $J_F(\gamma)$ across face $F$ has a sign. 
\begin{lemma}[sign of $J_F(\gamma)$]
  \label{sign}
  If $\gamma$ is a convex function in $\omega$, then $J_F(\gamma) \geq 0$ for all faces $F \in \F(0)$.
\end{lemma}
\begin{proof}
Let $\{K^\pm\} \subset \omega$ be the elements sharing $F$ and
$\N F := K^+ \cup K^-$. If $n_F$ is the normal vector of $F$ pointing from $K^+$ to $K^-$, then $J_F(\gamma)$ reads
  \[
J_F(\gamma) = \gradv \gamma|_{K^-} \cdot n_F - \gradv \gamma|_{K^+}  \cdot n_F.
  \]
Take a point $ x \in F$ and $\epsilon > 0$ sufficiently small such
that $ x \pm \epsilon n_F \in \N F$. Since $\gamma ( x)$ is piecewise affine and convex, we have
\[
J_F (\gamma) =  \gradv \gamma|_{K^-} \cdot n_F - \gradv \gamma|_{K^+}  \cdot n_F 
= \frac{  \gamma( x + \epsilon n_F) + \gamma( x - \epsilon n_F) - 2 \gamma(x)}{\epsilon} \geq 0,
\]
which is the asserted inequality.
\end{proof}

It is easy to see that $\gradv \gamma(0)$ is always a convex set.
Since $ \gamma( x) $ is a piecewise linear function on $\omega$, we have a more precise characterization.
\begin{lemma}
  [characterization of $\gradv \gamma ( 0)$]
  \label{characterization}
  The local sub-differential  $\gradv \gamma ( 0)$ is a convex polytope determined by the intersection of the half-spaces
  \[
    S_j := \{  w \in \mathbb{R}^d : \innerpro w { {z}_j} \leq \gamma(z_j) \}
    \quad 1 \leq j \leq m.
  \]
  Moreover, a vector $ w$ is in the interior of $\gradv \gamma (0)$ if and only if all the inequalities
\[
  \innerpro w {{z}_j} < \gamma({z}_j)
    \quad 1 \leq j \leq m
\]
hold strictly.
\end{lemma}

\begin{proof}
Since $ \gamma(x) $ is a piecewise affine function, any vector $ w$ is in the sub-differential $\gradv \gamma ( 0) $ if and only if
$
\innerpro w  {z_j} \leq \gamma( z_j)  \text{ for all $ z_j$}.
$
Therefore, the sub-differential $\gradv \gamma ( 0)$ is determined by the intersection of the half-spaces $S_j$ for $1\leq j \leq m$.
If $\innerpro w {z_j} < \gamma({z}_j)$ for all $ 1 \leq j \leq m$, then for $\epsilon > 0$ sufficiently small such that 
\[
  \epsilon |{z}_j| \leq \gamma ({z}_j) - \innerpro w {z_j} 
  \quad 1 \leq j \leq m,
\]
we deduce 
\[
\innerpro { w + v} {z_j} \leq \gamma ({z}_j),
\]
for all $ v$ in the small $B_{\epsilon}(0)$ with radius $\epsilon$ and centered at $0$, 
whence $ w +  v \in \gradv \gamma ( 0)$. This implies that $ w$ is in the interior of $\gradv \gamma ( 0)$.  The argument can be reversed to prove the equivalence. 
\end{proof}
Lemma \ref{characterization} immediately leads to two important consequences. First, if $\gamma(x) \geq 0$ with equality only at the origin, i.e. $\gamma(z_j) >0$ for all $1 \leq j \leq m$, then the vector $ 0$ is in the interior of $\gradv \gamma(0)$. This implies that $\gradv \gamma( 0)$ has a non-empty interior and is thus $d$-dimensional.
Second, a vector $ w$ is on the boundary $\partial \gradv \gamma ( 0)$
of the sub-differential $\gradv \gamma ( 0)$ 
if and only if equality holds for at least one of the inequalities
\[
  \innerpro w {z_j} \leq \gamma({z}_j)
    \quad 1 \leq j \leq m.
\]
The second consequence gives a characterization of $\partial \gradv \gamma (0)$ which motivates us to introduce the notion of dual set below.

Let $T$ be an $n$-dim simplex with $0 \leq n \leq d$ such that $ 0 \in T$. We define 
\begin{equation}\label{NT}
  \N T = \cup \{ K \subset \omega : \; T \subset K , K \text{ $d$-dim simplex}\},
\end{equation}
and the {\it dual set} $T^*$ of $T$ with respect to a convex piecewise affine function $\gamma$
\begin{align}\label{dualset}
T^* = \{ {w} \in \gradv \gamma ( 0): \; \innerpro {w} {z} = \gamma ( z) \; \forall  z \in T \}.
\end{align}
\begin{lemma}
  [geometry of $T^*$]
  \label{geometry_of_T*}
  The dual set $T^*$ of an $n$-dim simplex $T$ is a convex polytope contained in the $(d -n)$-dim plane
  \[
    P = \{ w \in \mathbb{R}^d: \; \innerpro  w  z = \gamma ( z) \quad \forall  z \in T \},
  \]
  which happens to be orthogonal to $T$.
\end{lemma}
\begin{proof}
  It is obvious that $T^*$ is a subset of $P$. Moreover, in view of Lemma \ref{characterization} (characterization of $\gradv \gamma ( 0)$) and the definition \eqref{dualset}, we realize that
  $
    T^* = \cap_{j=1}^m S_j \cap P
  $
  which means that $T^*$ is a convex polytope bounded by the half-spaces $\{ S_j \}, 1\leq j \leq m$.

To show that $P$ is orthogonal to $T$ we see that, given arbitrary $ w_1,  w_2 \in P$,  
$
 \innerpro {w_1 -  w_2} z = 0 
$
for all $ z \in T$. This proves the claim.
\end{proof}

The geometry of $T^*$ is rather simple in two dimensions as the following example illustrates. 
\\
$\bullet$ { \bf Case I  ($0$-dim simplex):} If $T = \{0\}$, then $\N T = \omega$. It is easy to check by definition that $T^*$ is nothing but the sub-differential $\gradv \gamma( 0)$.
\\
$\bullet$   { \bf Case II  ($2$-dim simplex):}
If $T = K$ is an element, then $\N K = K$. If a vector $ w \in K^*$,
then the equality $\innerpro {w} {z} = \gamma( z)$ for all $ z \in K$ implies
${w} = \gradv \gamma|_K$.
It is easy to check that the constant gradient $\gradv \gamma|_K $ is in the sub-differential $\gradv \gamma( 0)$ by using the convexity of function $\gamma$. Hence,  we conclude that $K^*$ consists of one vector $\gradv \gamma|_K$ only.
\\
$\bullet$  { \bf Case III  ($1$-dim simplex):}
Finally, we consider the most complicated case by taking $T = F$ which
is a face with two vertices $ 0,  z_1$. Then, $\N F$ consists of two elements  $K^\pm$ sharing the face $F$.
Lemma \ref{geometry_of_T*} implies $F^*$ is contained in the line 
\[
  \{ 
  w \in \mathbb{R}^2, \; \innerpro  w {z_1} = \gamma ( z_1)
\}
\]
It is easy to check that  $ \innerpro   {\gradv \gamma |_{K^\pm}} {z_1} = \gamma ( z_1)$ which implies that the two constant gradients $ \gradv \gamma |_{K^\pm} \in F^*$.
We claim that 
\begin{align}
  \label{claim2}
  F^*
  \text{ is the line segment joining the  two vectors $\gradv \gamma|_{K^\pm}$.}
\end{align}
Moreover, Lemma \ref{characterization} gives us the following characterization of $\partial \gradv \gamma ( 0)$ 
\begin{align}
  \label{claim3}
  \partial \gradv \gamma ( 0) = \cup \{ F^* ,\;  0 \in F \} ,
\end{align}
namely the boundary of $\gradv \gamma (0)$ is made of straight segments joining $\gradv \gamma |_{K}$ on consecutive triangles $K$ clockwise.
Both claims are proved in Proposition \ref{face_subdifferential} below in a more general setting which holds for any space dimensions.
Figure \ref{F:dual-set} depicts a face $T=[z_1,z_3]$ and its dual
  set $T^*$ for $d=2$.

\begin{figure}[h!]
\begin{tikzpicture} 
  \coordinate (z1) at (0,0);
  \coordinate (z2) at (0.5,3);
  \coordinate (z3) at (3,2.5);
  \coordinate (z4) at (3,-0.5);
  \coordinate (T)  at (2.1,1.75);
  \coordinate (K+)  at (1.2,2);
  \coordinate (K-)  at (2,0.5);
  \coordinate (o) at (7,0);
  \coordinate (zz3) at (10,2.5);
  \coordinate (p1) at (7.25,3.5);
  \coordinate (p2) at (11,-1);
  \coordinate (g-) at (9.75,0.5);
  \coordinate (g+) at (8.0,2.6);
  \coordinate (T*)  at (8.7,2.2);
  \draw (z1)  [below]  node {$z_1=0$};
  \draw (z2)  [above left]  node {$z_2$};
  \draw (z3)  [above right] node {$z_3$};
  \draw (z4)  [below right] node {$z_4$};
  \draw (T)   [below right] node {$T$};
  \draw (K+)                node {$K^+$};
  \draw (K-)                node {$K^-$};
  \draw (o)   [below left]  node {$0$};
  \draw (zz3) [above right] node {$z_3$};
  \draw (g-)  [right] node {$\gradv \gamma|_{K^-}$};
  \draw (g+)  [above right] node {$\gradv \gamma|_{K^+}$};
  \draw (p2)  [above right] node {$P$};
  \draw (T*)                node {$T^*$};
  \draw [-latex] (z1) -- (z2)  ;
  \draw [-latex, ultra thick] (z1) -- (z3)  ;
  \draw [-latex] (z1) -- (z4)  ;
  \draw (z4) -- (z3) -- (z2) ;
  \draw [-latex , thick] (o) -- (zz3) ;
  \draw  (p1) -- (p2);
  \draw [double, thick](g+) -- (g-);
  \draw [-latex] (o) -- (g+);
  \draw [-latex] (o) -- (g-);
\end{tikzpicture}
\caption{\small Face $T = [z_1, z_3]$ and its dual set $T^*$ (segment joining $\gradv \gamma|_{K^{\pm}}$). The latter lies on a straight line $P$ perpendicular to $z_3$.}
\label{F:dual-set}
\end{figure}
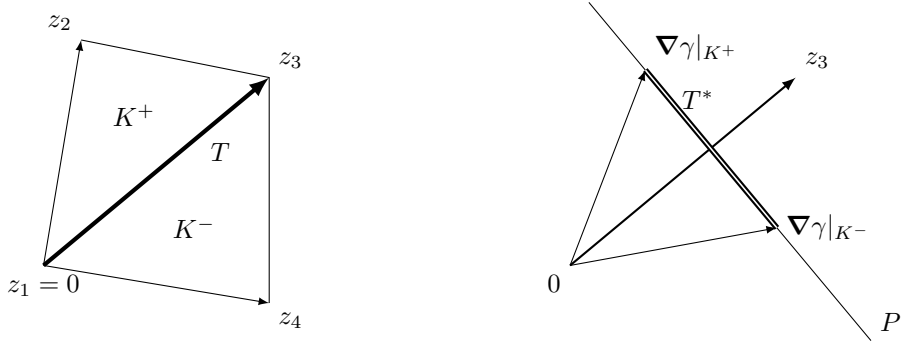
\noindent
Finally, we mention that combining claims \eqref{claim3} and \eqref{claim2} implies that 
\begin{align*}
  \gradv \gamma ( 0)
  \text{ is the (convex hull) polygon with vertices $\{ \gradv \gamma |_K , \; K \subset \omega \}$.}
\end{align*}

We now establish a characterization of dual set $T^*$ for any $n$-dim simplex $T$, which is inspired in \cite{Grunbaum03} and extends the preceding discussion to any dimension $d$.
\begin{prop}[characterization of dual set]
\label{face_subdifferential}
Let $0 \leq n < d$ and 
\begin{align*}
T  &\; \text{ be an $n$-dim simplex of $\omega$ such that $ 0 \in T$,} 
\\
 \mathscr{S} &\; \text{ be the set of $(n+1)$-dim simplices $S$ of $\omega$ such that $S \supset T$.}
\end{align*}
The dual set $T^*$ of $T$ is the convex polytope given by 
\[
T^* = \{  w \in \mathbb{R}^d: \; \innerpro  w z = \gamma ( z) \; \forall z \in T 
\quad \text{and} \quad
 \innerpro w  {z} \leq \gamma ( z) \; \forall  z \in \N T \}.
\]
Moreover, the boundary $\partial T^*$ of $T^*$ is given by $\partial T^* = \cup \{ S^* : \; S \in \mathscr{S} \}$. 
\end{prop}

  Before proving Proposition \ref{face_subdifferential}, we apply it
  to characterize the geometry of the boundary $\partial \gradv \gamma
  ( 0)$ of the sub-differential $\gradv \gamma ( 0)$ for $d = 2, 3$.
We denote by $\{ E \}$ (edges) the set of all 1-dim simplices
  for $d=3$ such that ${0} \in E$, and by $\{ F \}$ (faces)
  the set of all $(d-1)$-dim simplices for $d=2, 3$ such that ${0} \in F$. 
  We let $\{ K \}$ be the set of all $d$-dim simplices (tetrahedra for
  $d=3$ and triangles for $d=2$) such that ${0} \in K$.

\begin{corollary}[characterization of $\partial \gradv \gamma ( 0)$ for $d = 2, 3$]\label{subdifferentialfacelattice}
For $d = 2$,  the boundary  $\partial \gradv \gamma ( 0)$  of the sub-differential $\gradv \gamma ( 0)$ is the union of dual sets $F^*$ for all edges $F \subset \omega$ such that $ 0 \in F$. 
Each dual set $F^*$ is the segment with endpoints $\{\gradv \gamma |_{K^{\pm}}: \;  K^{\pm} \subset \N F\}$ and the length of $F^*$ equals the jump $J_F$.
\\
For $d = 3$, the boundary $\partial \gradv \gamma ( 0)$ is the union of dual sets $ E^*$ for all edges $E \subset \omega$ such that $ 0 \in E $.
The boundary $\partial E^*$ is the union of dual sets $F^*$ for all faces $F$ such that $E \subset F $.
Each dual set $F^*$ is a segment with endpoints $\{\gradv \gamma |_{K^{\pm}}: \;  K^{\pm} \subset \N F\}$ and the length of $F^*$ is the  jump $J_F$.
\end{corollary}
\begin{proof}
We only prove the lemma for $d= 3$;
the case $d=2$ is simpler.
To prove $\partial \gradv \gamma ( 0) = \cup E^*$, we take $T$ in Proposition \ref{face_subdifferential} to be the origin ($0$-dim simplex) and $\mathscr{S}$ to be the set of all edges ($1$-dim simplices) $E \ni  0$. 
Since $T^* = \gradv \gamma ( 0)$, the first assertion follows immediately from Proposition \ref{face_subdifferential}.

Similarly, to prove 
\[
 \partial E^* = \cup F^* \quad \forall F \supset E 
\qquad \text{and} \qquad
 \partial F^* = \cup (K^{\pm})^* 
\quad \forall K^{\pm} \supset F 
\]
we take $T$ to be either an edge $E$ or a face $F$ and $\mathscr S$ to be $\{ F:\; F\supset E \}$ or $\{K^\pm: \;  K^\pm \supset F\}$ respectively. The second assertion follows again directly from Proposition \ref{face_subdifferential}.

Finally, since $F^*$ is the line segment connecting $(K^\pm)^* = \gradv \gamma|_{K^\pm}$, we deduce that the length $\abs{F^*}$ of $F^*$ satisfies
\[
\abs{F^*} = \abs{ \gradv \gamma|_{K^+} - \gradv \gamma|_{K^-}}.
\]
The fact that $\gradv \gamma|_{K^+} - \gradv \gamma|_{K^-}$ is perpendicular to $F$, indeed equal to $J_F n_F$ with $ n_F$ being the unit normal pointing from $K^+$ to $K^-$, in conjunction with Lemma \ref{sign} (sign of $J_F$), yields $\abs {F^*} = J_F$ as asserted. 
\end{proof}

Now we proceed to prove Proposition \ref{face_subdifferential}. 
\begin{proof}[Proof of Proposition \ref{face_subdifferential}]
To show the first statement, we note that by definition \eqref{dualset}, we have $\innerpro w z \leq \gamma(z)$ for all $z \in \omega$, whence
\[
T^* \subset \{  w \in \mathbb{R}^d: \; \innerpro  w {z} = \gamma ( z) \; \forall  z \in T 
\quad \text{and} \quad
 \innerpro w  {z} \leq \gamma ( z), \; \forall z \in \N T \}.
\]
To show the reversed inclusion, we argue by contradiction: assume that $ \innerpro w  {z} \leq \gamma ( z)$ for all $ z \in \N T$ with equality for all $ z \in T$, but $ w \notin T^*$ or equivalently $ w \notin \gradv \gamma ( 0)$. Then there is a point $ z_0 \in \omega$ such that 
$
\innerpro {z_0}  w > \gamma ( z_0).
$
Let $ z_1 \in T$ be a point in the interior of $\N T$. Hence $\gamma( z_1) = \innerpro w {z_1}$ and due to the convexity of $\gamma( z)$, for $0 < \lambda < 1$,
\begin{align*}
\gamma (\lambda  z_0 + (1 - \lambda)  z_1 ) 
\leq
\lambda \gamma ( z_0 )  + (1 - \lambda) \gamma ( z_1 )
<
 \innerpro {\lambda {z}_0 + (1 - \lambda) {z}_1 }  w .
\end{align*}
Since $ z_1$ belongs to the interior of $\N T$, we have $ \lambda z_0 + (1 - \lambda) z_1 \in \N T$  for $\lambda$ small enough. Consequently, the inequality contradicts the assumption that $ \innerpro w z \leq \gamma( z)$ for all $ z \in \N T$.
This proves the first statement.

Now, we show that $\partial T^* = \cup S^*$ for all $S\in \mathscr{S}$. In view of \eqref{dualset}, this is equivalent to showing that $ w \in \partial T^*$ if and only if $ w \in \gradv \gamma ( 0)$ and the equality
\begin{align}\label{prop:eq1}
 \innerpro w {z} = \gamma( z) 
\text{ holds for all $ z \in S$ and some $(n+1)$-dim simplex $S \supset T$. }
\end{align}
Let $ z_s $ be the vertex of $S$ off the simplex $T$ ($ z_s \notin T$). Since $\gamma(x)$ is piecewise affine, \eqref{prop:eq1} is equivalent to showing that
\begin{align}\label{eq1}
 \innerpro w  {z_s} = \gamma( z_s) \quad \text{ for some $S \in \mathscr S$} .
\end{align}
We also recall from Lemma \ref{geometry_of_T*} (geometry of $T^*$) that $T^*$ is contained in the $(d - n)$-dim plane
\[
  P = \{  w \in \mathbb{R}^d : \; \innerpro w  z = \gamma ( z) \quad \forall  z \in T \}
\]
which is orthogonal to $T$. 
Consequently, a vector $ w \in T^*$ is in the interior of $ T^*$ if and only if there is a small $\epsilon >0$ such that 
$ w + \epsilon  n \in T^*$ for any unit vector $ n \; \bot \; T$. Equivalently,
\begin{align}\label{prop:eq2}
 w \in \partial T^* \iff \exists  n \; \bot \; T \text{ such that $ w + \epsilon  n \notin T^* $ for any $\epsilon > 0$.}
\end{align}

We first prove that if $ w \in \partial T^*$, then \eqref{eq1} holds. If not, then
\[
 \innerpro w {z_s} < \gamma( z_s) 
\quad 
\text{for all vertices } z_s 
\quad \text{ and } \quad
 \innerpro w {z} = \gamma( z)
\quad \text{ for all $ z \in T$.}
\]
There is $\epsilon > 0$ sufficiently small such that, 
\[
 \innerpro{ w + \epsilon {n}} {z_s} \leq \gamma( z_s) \quad \forall z_s 
\quad \text{and} \quad 
 \innerpro w  z = \gamma( z) \quad \forall  z \in T
\]
for any unit vector ${n}$ orthogonal to $T$. 
Since $\gamma( z)$ is piecewise linear, this implies that for each element $K \subset \N T$
\[
 \innerpro { w + \epsilon {n}} {z} \leq \gamma( z)          
\quad \forall  z \in K. 
\quad \text{ and } \quad
 \innerpro {w + \epsilon {n}} {z} = \gamma( z) 
\text{ for all $ z \in T$,}
\]
whence $( w + \epsilon  n) \in T^*$ for any $ n \bot T$ according to the first assertion of this Proposition. This contradicts that $ w \in \partial T^*$ in view of \eqref{prop:eq2}.

We next show that if \eqref{eq1} holds for some vector $w\in T^*$, then $ w \in \partial T^*$. 
Let ${p}( z_s)$ be the orthogonal projection of $ z_s$ onto the face $T$ and ${n}( z_s) = {z}_s - {p}( z_s)$; obviously, the vector ${n}( z_s) \neq {0}$ and ${n}( z_s) \;\bot\; T$. 
Since $ \innerpro w {z_s} = \gamma ( z_s)$, we obtain
\[
\innerpro { w + \epsilon {n}( z_s)} {z_s} = \gamma( z_s) + \epsilon | {n}( z_s) |^2 > \gamma( z_s) 
\quad \text{ for all $\epsilon > 0$ }
\] 
whence $ w + \epsilon {n}( z_s) \notin \nabla\gamma(0)$ for any
$\epsilon > 0$. This implies that $w  + \epsilon {n}( z_s) \notin T^*$
for any $\epsilon > 0$ because
$T^*\subset\nabla\gamma(0)$ according to \eqref{dualset}.
With the aid of \eqref{prop:eq2} we thus deduce $ w \in\partial T^*$,
and conclude the proof.
\end{proof}

\begin{proof}[Proof of Proposition \ref{estimate_of_subdifferential}]
The proof hinges on the isoperimetric inequality relating the measure $\abs P$ of an $n$-dim polytope $P$ with that of its perimeter $\abs {\partial P}$: there exists a constant $C = C(n)$, thereby depending on $d$, so that
\[
\abs P \leq C \abs{\partial P}^{n / (n-1)}.
\]
The proof proceeds by dimension reduction. We know that $\gradv \gamma
( 0)$ is the dual set of $T = \{ 0\}$ and, by virtue of
Proposition \ref{face_subdifferential} (characterization of dual set), that
\[
\partial \gradv \gamma( 0) = \cup \{ S_1^* :\; S_1 \in \mathscr{S}_1( 0) \}
\]
where $ \mathscr{S}_1( 0)$ is the set of all $1$-dim simplices of $\omega$ such that $ 0 \in S_1$. Therefore 
\begin{equation}\label{1st-step}
\abs { \gradv \gamma( 0) } \leq C \abs{ \partial \gradv \gamma( 0) }^{d/(d-1)} \leq C \left( \sum_{S_1 \in \mathscr{S}_1( 0)} \abs{S_1^*} \right)^{d/(d-1)}.
\end{equation}
The dual sets $S_1^*$ are convex $(d-1)$-dim polytopes orthogonal to $S_1$, according to Lemma \ref{geometry_of_T*} (geometry of $T^*$). Applying again Proposition \ref{face_subdifferential}, this time to $T=S_1$, we obtain
\[
\partial S_1^* = \cup \{ S_2^* :\; S_2 \in \mathscr{S}_2(S_1) \}
\]
where
$\mathscr{S}_2(S_1)$ stands for all $2$-dim simplices $S_2$ of $\omega$ such that $S_1 \subset S_2$. Hence
\[
\abs{S_1^*} \leq C \abs{\partial S_1^*}^{(d-1)/(d-2)} \leq C \left( \sum_{S_2 \in \mathscr{S}_2(S_1)} \abs{S_2^*} \right)^{(d-1)/(d-2)}.
\]
Inserting this in the expression for $\abs{\gradv \gamma ( 0)}$, we get 
\[
\abs{\gradv \gamma ( 0)} \leq C \left( \sum_{S_1 \in \mathscr{S}_1( 0)} \left( \sum_{S_2 \in \mathscr{S}_2(S_1)}  \abs{S_2^*} \right)^{(d-1)/(d-2)}   \right)^{d/(d-1)}.
\]
Since $\sum_{i} a_i^t \leq (\sum_i a_i)^t $ is valid for any nonnegative sequence $\{ a_i \}$ and $t \geq 1$, the preceding inequality becomes
\begin{equation}\label{2nd-step}
\abs{\gradv \gamma ( 0)} \leq C \left(  \sum_{S_1 \in \mathscr{S}_1( 0)}  \sum_{S_2 \in \mathscr{S}_2(S_1)} \abs{S_2^*}  \right)^{d/(d-2)}. 
\end{equation}
Moreover, each $2$-dim simplex $S_2$ contains exactly two $1$-dim simplices $S_1 \ni  0$. This allows us to rewrite $\abs{ \gradv \gamma ( 0)}$ with $C$ modified by a factor $2^{d/(d-2)}$ as follows:
\[
\abs{ \gradv \gamma ( 0)} \leq C \left( \sum_{S_2 \in \mathscr{S}_2( 0)} \abs{S_2^*}  \right)^{d/(d-2)}.
\]
Iterating this argument, we easily arrive at 
\[
\abs{ \gradv \gamma ( 0)} \leq C \left(  \sum_{S_{d-1} \in \mathscr{S}_{d-1}( 0)} \abs{S_{d-1}^*} \right)^d,
\]
with $C = C(d)$. The dual set $S_{d-1}^*$ of a $(d-1)$-simplex $S_{d-1} = F$ or face $F$, is a $1$-dim segment connecting $\gradv \gamma|_{K^\pm}$ where $K^\pm \in \Th$ are the elements sharing $F$ (see proof of Corollary \ref{subdifferentialfacelattice}). Consequently, 
\[
\abs{F^*} = J_F
\]
because the length $\abs{F^*}$ of $F^*$ equals the jump $J_F$. This concludes the proof. 
\end{proof}

\subsection{Proof of Theorem \ref{discrete_ABP}}
We are now ready to prove the discrete ABP estimate
for $d=2,3$ and comment on the case $d > 3$. We start with $d = 3$ for which the main difficulty is that
$\gamma=\gradv \Gamma_{x_i} (v_h)$ may not belong to $\Vh$, whence its jumps
$J_F(\gamma)$ may not be directly related to those
of $\nabla v_h$, namely $J_F(v_h)$. We proceed as in 
Proposition \ref{estimate_of_subdifferential} upon reducing the dimension.
  
Let $x_i \in \mathcal{C}_h^- (v_h)$ be a (lower) contact node
for $v_h$ and let $\gamma(x_i) = 0$ for simplicity.
In view of \eqref{1st-step}, there is a constant $C$ depending on the
dimension $d$ such that
\begin{equation*}
  |\nabla\gamma(x_i)| \le C \left( \sum_{S_j\in\mathscr{S}_1(x_i)}
  |S_j^*| \right)^{d/(d-1)},
\end{equation*}
where $\mathscr{S}_1(x_i)$ is the set of edges (or 1-dim simplices)
$S_j$ connecting nodes $x_j$ and $x_i$ and $S_j^*$ is the
dual set of $S_j$ with respect to $\gamma$
\[
S_j^* = \{w\in\nabla\gamma(x_i): \, \langle w,x_j-x_i\rangle = \gamma(x_j)\}.
\]
To estimate $|S_j^*|$ we introduce a convex function
$\gamma_j$ defined in $\omega_{ij}$ as follows:
\[
\gamma_j(x) := \sup_{L \textrm{ affine}}\{L(x): \, L=\gamma \textrm{ on } S_{j},
\, L(x_k) \le \gamma(x_k) \text{ for all } x_k \in \Nh(\omega_{ij}) \},
\]
where $\omega_{ij} := \N{S_j}$ is defined in \eqref{NT} and
$\Nh(\omega_{ij}) := \Nh\cap\omega_{ij}$. The same proof of Lemma
\ref{lce-d=2} shows that $\gamma_j\in\Vh(\omega_{ij})$. Since the sub-differential
$\nabla\gamma_j(x_i)$ is
\[
\nabla\gamma_j(x_i) = \{ w\in\mathbb{R}^d: \,
\langle w, x_j-x_i\rangle = \gamma(x_j), \,
\langle w, x_k-x_i\rangle \le \gamma(x_k) \,\, \forall x_k \in \Nh(\omega_{ij}) \},
\]
we deduce $S_j^* \subset \nabla\gamma_j(x_i)$, whence $|S_j^*| \le|\nabla\gamma_j(x_i)|$
and we have to estimate the latter. The set $\nabla\gamma_j(x_i)$ is a
convex polygon perpendicular to the edge $S_j$ and is the dual
set of $S_j$ with respect to the convex function $\gamma_j$.
Applying Proposition \ref{face_subdifferential} we get an expression for
$\partial\nabla\gamma_j(x_i)$, namely
\[
\partial\nabla\gamma_j(x_i) = \cup \{F^*: \, F\in\mathscr{S}_2(S_j)  \},
\]
where $\mathscr{S}_2(S_j)$ is the set of faces (or 2-dim simplices)
containing $S_j$. The dual sets $F^*$ are 1-dim segments connecting
the gradient $\nabla\gamma_j$ in the two elements sharing $F$, whence
$|F^*|=J_F(\gamma_j)$. Consequently, we infer that
\[
|\nabla\gamma_j(x_i)| \le C |\partial\nabla\gamma_j(x_i)|^{\frac{d-1}{d-2}}
\le C \left( \sum_{F\in\mathscr{S}_2(S_j)} J_F(\gamma_j) \right)^{(d-1)/(d-2)}
\]
and, arguing as in the proof of Proposition \ref{estimate_of_subdifferential},
we further obtain
\begin{equation}\label{subdiff}
|\nabla\gamma(x_i)| \le C \left( \sum_{S_j\in\mathscr{S}_1(x_i)}
\sum_{F\in\mathscr{S}_2(S_j)} J_F(\gamma_j) \right)^d.
\end{equation}

It remains to estimate the right-hand side of this expression. 
It is worth mentioning here that we could use induction and an
argument similar to that below to deal with dimension $d > 3$. For
simplicity, we just prove the assertion of Theorem \ref{discrete_ABP} for $d =3$.
We first recall that $J_F(\gamma_j)\ge0$ according to Lemma
\ref{sign}.
Since $\abs{F}\simeq \abs{\omega_i}^{1- \frac1d}$ and $J_F(\gamma_j)$ is
constant, we can write
\[
\sum_{F\in\mathscr{S}_2(S_j)} J_F(\gamma_j) \le C |\omega_i|^{\frac1d-1}
\sum_{F\in\mathscr{S}_2(S_j)} \int_F J_F(\gamma_j) \phi_i\phi_j,
\]
where the constant $C$ depends on the dimension $d$ and geometric quantity
\begin{equation}\label{const-G}
G:= \max_{\Th\in\mathbb{T}}\max_{x_i\in\Nh}\max_{F \ni x_i} \big\{\abs{F}^{-d}\abs{\omega_i}^{d-1} \big\}.
\end{equation}
We next exploit that $\phi_i\phi_j$ vanishes on $\partial\omega_{ij}$
to integrate by parts and thereby obtain
\[
\sum_{F\in\mathscr{S}_2(S_j)} J_F(\gamma_j) \le - C |\omega_i|^{\frac1d-1}
\int_{\omega_{ij}} \nabla\gamma_j \cdot \nabla(\phi_i\phi_j).
\]
Since $\gamma_j,\phi_j,\phi_i$ are all piecewise linear, the
right-hand side reads
\begin{align*}
\int_{\omega_{ij}} \nabla\gamma_j \cdot \nabla(\phi_i\phi_j)
& = \int_{\omega_{ij}} \nabla\gamma_j \cdot \nabla\phi_i \, \phi_j
+ \nabla\gamma_j \cdot \nabla\phi_j \, \phi_i
\\
& = \frac{1}{d+1} \int_{\omega_{ij}} \nabla\gamma_j \cdot \nabla\phi_i
+ \nabla\gamma_j \cdot \nabla\phi_j.
\end{align*}
We now resort to \eqref{localweaklyacute}, the face weakly acute
condition on $\Th$, to replace $\gamma_j$ by $I_h\gamma$. In fact, we know
that $\gamma_j(x) = \sum_{x_k\in\Nh(\omega_{ij})} \gamma_j(x_k)\phi_k$ 
and $\gamma_j(x_k) \le \gamma(x_k)$ with equality at $x_k=x_i$ and
$x_k=x_j$, whence
\[
\int_{\omega_{ij}} \nabla\gamma_j \cdot \nabla\phi_i
= \sum_{x_k\in\Nh(\omega_{ij})} \gamma_j(x_k)\int_{\omega_{ij}}\nabla\phi_k\cdot\nabla\phi_i
\ge \sum_{x_k\in\Nh(\omega_{ij})} \gamma(x_k)\int_{\omega_{ij}}\nabla\phi_k\cdot\nabla\phi_i.
\]
Since the same inequality hold for the remaining term
$\int_{\omega_{ij}} \nabla\gamma_j \nabla\phi_j$, we infer that
\[
\int_{\omega_{ij}} \nabla\gamma_j \cdot \nabla(\phi_i\phi_j)
\ge \int_{\omega_{ij}} \nabla I_h\gamma \cdot \nabla(\phi_i\phi_j).
\]
To complete the estimate of the right-hand side of \eqref{subdiff} we
must add over $S_j\in\mathscr{S}_1(x_i)$. We now make use of
$\sum_{S_j\in\mathscr{S}_1(x_i)} \phi_j = 1 -\phi_i$ together with
\eqref{lack-consistency} to obtain
\begin{align*}
\sum_{S_j\in\mathscr{S}_1(x_i)}\sum_{F\in\mathscr{S}_2(S_j)} J_F(\gamma_j)
& \le -C |\omega_i|^{\frac1d-1}
\int_{\omega_i} \nabla I_h\gamma \cdot \big(\nabla\phi_i -2\phi_i\nabla\phi_i\big)
\\
& = -C \frac{d-1}{d+1} |\omega_i|^{\frac1d-1} \int_{\omega_i} \nabla
I_h\gamma \cdot \nabla\phi_i
= C |\omega_i|^{\frac1d} \Delta_h I_h\gamma (x_i).
\end{align*}

Since
\[
I_h\gamma(x) \leq v_h(x) \qquad \text{ for all $x \in \omega_i$}
\]
with equality at $ x_i$, the monotonicity property of $\laplace_h$
in Lemma \ref{Monotonicity} yields
\[
\laplace_h I_h\gamma(x_i)
\leq 
\laplace_h v_h (x_i) .
\]
Now to prove Theorem \ref{discrete_ABP}, we only need to show that 
\[
\laplace_h v_h (x_i) \leq C f_i
\qquad \forall \, x_i \in \mathcal{C}_h^- (v_h).
\]
Since the (global) convex envelope
$\Gamma (v_h)$ touches $v_h$ at $x_i$ from below, we get
\[
0 \leq \Ie \Gamma (v_h) (x_i) \leq \Ie  v_h (x_i)
\]
where the first inequality follows from the convexity of $\Gamma(v_h)$
and the second one from the monotonicity of operator $\Ie$ in Lemma
\ref{Monotonicity}. Hence, by the definition \eqref{discrete_pde} of
discrete operator $L_h^\epsilon$ and the fact that $f_i\ge0$ for
$x_i\in\mathcal{C}_h^- (v_h)$, we obtain
\[
\laplace_h v_h (x_i) \leq \frac{2}{\lambda} L_h^\epsilon v_h (x_i)
\leq \frac{2}{\lambda} f_i  = \frac{2}{\lambda} f_i^+.
\]

Altogether, utilizing \eqref{subdiff}, we conclude that 
\[
\abs{ {\gradv \Gamma (v_h) (x_i)} } \leq
\abs{ {\gradv \gamma (x_i)}} \leq
C |f_i^+|^d |\omega_i| \qquad \forall x_i \in \mathcal{C}_h^- (v_h).
\]
Finally, invoking Proposition \ref{Alexandroff}
(discrete Alexandroff estimate), we arrive at
\begin{equation}\label{discrete-ABP}
  \sup_{\dm} v_h^- \leq C \left( \sum_{x_i \in \mathcal{C}_h^- (v_h) }
  \abs {f_i^+}^d \abs {\omega_i} \right)^{1/d},
\end{equation}
which is the desired discrete ABP estimate. 
This completes the proof
for $d=3$. The case $d=2$ is simpler because $I_h\gamma=\gamma$
and the first step above already gives
\[
|\nabla\gamma(x_i)| \le C \left( \sum_{F\in\mathscr{S}_1(x_i)}
  |F^*| \right)^{2} = C \left( \sum_{F\in\mathscr{S}_1(x_i)}
  J_F(\gamma) \right)^{2} \le C |\omega_i| \big(\Delta_h\gamma(x_i)\big)^2.
\]

The proof shows that the constant $C$ in \eqref{discrete-ABP} depends
on $\lambda^{-1}$ and the constants $C(d, \lambda, \dm)$ in
Proposition \ref{Alexandroff} (discrete Alexandroff estimate) and $G$ in
\eqref{const-G}, rather than the shape regularity constant $\sigma$.
Therefore, $C$ is independent of the number $n$ of elements within $\omega_i$,
which is an improvement over \cite{KuoTrudinger00} where $C$ depends on 
$n$. 

%
                    \section{A priori error estimates}\label{S:error-estimate}
%

In this section, we proceed as follows to
derive rates of convergence for the FEM.
In $\S$ \ref{subsec:Galerkinprojection}, we review a finite element
approximation $u_G$ of the solution $u$, commonly known as Galerkin projection. In $\S$ \ref{subsec:boundary-layer}, we introduce a boundary layer function which is instrumental to deal with points $\epsilon$-close to the boundary $\partial \dm$. 
In $\S$ \ref{subsec:erroreqn}, we derive the error equation \eqref{erroreqn} for $\ue_h - u_G$. In $\S$ \ref{subsec:estimateTi}, we examine \eqref{erroreqn} and show that the various terms exhibit a decay rate, measured in $L^d$-norm, in the region bounded $\epsilon$-away from the boundary. 
In $\S$ \ref{subsec:barrierfunction}, we develop a discrete barrier
function which is instrumental in controlling the behavior of the
error $\ue_h - u_G$ in the region $\epsilon$-close to the boundary. 
We conclude in $\S$ \ref{subsec:rates} and
\ref{subsec:pwsmooth} with pointwise rates of
convergence, which combine the discrete ABP estimate and the discrete
barrier technique.
In \S \ref{subsec:rates} we deal with
$C^{2,\alpha}$ solutions whereas in \S \ref{subsec:pwsmooth} we allow
solutions to be piecewise $C^{2,\alpha}$.
Throughout this section, we take $\epsilon = \epsilon (h) \geq Ch |\ln
h|$ so that $h/\epsilon(h)\to0$ as $h\to0$.

\subsection{Galerkin projection}\label{subsec:Galerkinprojection}
%
We have already shown in \eqref{inconsistency} that $\Delta_h I_h u(x_i)$
  does not converge to $\Delta u(x_i)$ as $h\to0$ for general meshes $\Th$. To
  circumvent this operator inconsistency, we borrow an idea from
  \cite{JensenSmears13} and consider the Galerkin projection $u_G$ of
  $u$ instead.

We recall that $\dm$ must be at least $C^{1,1}$ for the solution $u$
of \eqref{pde} to be of class $W^2_\infty(\dm)$.
Let $\dm_h$ be a polytope induced by $\Th$ with boundary nodes on
$\partial\dm$.

We define the Galerkin (or elliptic) projection $u_G \in \Vhzero$ of
$u$ as follows:
\begin{align}\label{Galerkinprojection}
\int_{\dm_h} \gradv u_G \cdot \gradv v_h = \int_{\dm_h} \gradv u \cdot\gradv v_h
= - \int_{\dm_h} \Delta u \, v_h
\quad \text{ for all $v_h \in \Vhzero$} ,
\end{align}
provided $u\in C^2(\overline{\dm})$ is suitably extended to $\dm_h$.
Upon taking $v_h = \phi_i$, we have
\begin{align}\label{consistency}
\laplace_h u_G(x_i) = \frac{ \int_{\dm_h} \phi_i \laplace u} {\int_{\dm_h} \phi_i}
\qquad
\forall x_i \in \Nh,
\end{align}
according to \eqref{lack-consistency}.
Therefore, the discrete Laplacian $\laplace_h u_G$ of $u_G$ is a
weighted mean of $\laplace u$ over the star $\omega_i$ and thus converges
to $\Delta u(x_i)$ as $h\to0$ in contrast to $\laplace_h I_h u$.

Our discretization satisfies the following three standard
assumptions \cite{SchatzWahlbin82}:
\begin{itemize}
  \item  The partition $\Th$ of $\dm_h$ is quasi-uniform and
    shape regular;
  \item  The Hausdorff distance between $\partial \dm$ and $\partial \dm_h$ satisfies 
    \[
    \text{\rm dist }(\partial\dm_h,\partial\dm)
    = \max_{x \in \partial \dm_h} \text{\rm dist }(x, \bdry) \leq C h^2;
    \]
  \item  Functions $v \in W^2_\infty(\dm)$ that vanish on $\bdry$ can
    be approximated by piecewise linear functions that vanish on
    $\partial\dm_h$ to order $h^2$ in the maximum norm. 
\end{itemize}
Therefere, the convergence rate of the Galerkin projection $u_G$ in the
$L^{\infty}$-norm is known to be quasi-optimal for
$u\in C^0(\overline{\dm})$ and is given by \cite{SchatzWahlbin82}
\begin{align}\label{Linfty}
  \inftynormh { u - u_G } \leq C  |\ln h| \inf_{v_h \in \Vhzero}
  \inftynormh { u - v_h }  .
\end{align}
Moreover, if $ u \in W^{2}_{\infty}(\dm)$, then the third bullet above implies
\begin{align}\label{Galerkininftyestimate}
\inftynormh { u - u_G } \leq C h^2  |\ln h| \, |u|_{W^2_\infty (\dm)}.
\end{align}
In view of these results, and to avoid technical difficulties, we make the
somewhat standard simplifying assumption that $\dm_h=\dm$.
Thanks to \eqref{Galerkininftyestimate}, for all $x_i \in \Nh$ such that ${\rm dist}(x_i, \bdry) \geq Q \epsilon$, we obtain
\begin{align*}
 \Abs{\delta u_G (x_i, y) - \delta u (x_i, y)} \leq 
 C h^2  |\ln h| \,  |u|_{W^2_\infty (\dm)} 
\end{align*}
which, by definition \eqref{integral} of the integral operator $\Ie$, implies
\begin{align}
\label{estimateIe}
 \Abs{\Ie u_G (x_i) - \Ie u (x_i)}
 \leq C \frac{h^2} {\epsilon^2} |\ln h| \, |u|_{W^2_\infty  (\dm)}.
\end{align}

\subsection{Boundary layer function}\label{subsec:boundary-layer}
We introduce now a boundary layer function $b:\Omega\to\mathbb{R}^-$ which is
instrumental in dealing with the boundary layer
$\omega_\epsilon:=\Omega\setminus\Omega_\epsilon$,
where $\Omega_\epsilon$ is defined in \eqref{Omega-eps}. Let
$\dist (x)$ be the distance function from
$x\in\Omega$ to $\partial\Omega$, which inherits the same
regularity as $\partial\Omega$ for $x$ close to the boundary,
that is, $\dist(x)$ is of class $C^{1,1}$ provided $\dist(x) \leq Q \epsilon$
and $\epsilon$ is small. Let
$\zeta:\mathbb{R}^+\to\mathbb{R}^-$ be
\[
\zeta(s) :=
\begin{cases}
  Q^{-2} \big( s - Q\epsilon \big)^2 - \epsilon^2 & \quad s \le Q\epsilon
  \\
  -\epsilon^2 & \quad s > Q \epsilon,
\end{cases}
\]
and note that $\zeta''(s) = 2Q^{-2}\chi_{(0,Q\epsilon)}(s)$ where $\chi_{(0,Q\epsilon)}$ is the characteristic function of $(0, Q\epsilon)$. Let the
function $b$ be given by
\[
b(x) := \zeta(\dist(x))
\quad\forall \, x\in\Omega,
\]
and observe the simple but important properties
\begin{gather*}
\nabla b(x) = \zeta'(\dist(x)) \nabla \dist(x),
\\
D^2 b(x) = \zeta''(\dist(x)) \nabla \dist(x)\otimes \nabla \dist(x) + \zeta'(\dist(x))
D^2 \dist(x).
\end{gather*}
\begin{lemma}[integral operator of $b$]\label{L:boundlayer-intopr}
There is a constant $C>0$ such that
\begin{equation}\label{boundlayer-intopr}
I_\epsilon b(x) \ge C \chi_{\omega_\epsilon}(x) \quad\forall \, x\in\Omega,
\end{equation}  
i.e. $b$ is non-negative in $\Omega$ and strictly positive in $\omega_\epsilon$.
\end{lemma}
\begin{proof}
If $x\notin\omega_\epsilon$, then $b(x)=-\epsilon^2\le b(x+y), b(x-y)$
and $\delta b(x,y) \ge 0$ whence $I_\epsilon b(x)\ge0$. Therefore, we
consider $x\in\omega_\epsilon$ and observe that, in view of
\eqref{differenceformula}, it suffices to deal with
\[
\delta^+(x,y) := \int_0^1 \int_0^1 s D^2 b(x+sty) : y\otimes y \, dt ds
\]
and
\[
I^+_\epsilon b(x) := \int_{B_1(0)} \int_0^1 \int_0^1
s D^2 b \big(x+ts\theta\epsilon M(x) z \big) : M(x) z \otimes M(x) z \,
\vphi(z) \, ds dt dz.
\]
We further decompose $I^+_\epsilon b(x)$ into two terms according to
the expression of $D^2 b(x)$
\begin{align*}
A_\epsilon (x) &:= \int_{B_1(0)} \int_0^1 \int_0^1
s \zeta''\big(\dist(x(z)) \big) \big|\nabla \dist(x(z)) \cdot M(x)z  \big|^2
\, \vphi(z) \, ds dt dz,
\\
B_\epsilon (x) &:=
\int_{B_1(0)} \int_0^1 \int_0^1
\zeta' \big( \dist(z(x)) \big) D^2 \dist(x(z)) : M(x)z \otimes M(x) z \, \vphi(z) 
\, ds dt dz,
\end{align*}  
where $x(z) := x+ts\theta\epsilon M(x) z$.
We now introduce the ellipsoid $E_\epsilon(x)$ and cone $C(x)$,
centered at $x$ and with opening $\arccos \beta<\pi/2$, defined as follows:
\begin{equation}\label{ellipsoid}
E_\epsilon(x) := \{x(z): \, z\in B_1(0)\},
\quad
C(x) := \{ y : \, \langle x-y, \nabla \dist(x) \rangle \ge \beta |x-y| \}.
\end{equation}
We point out that the set $C_\epsilon := E_\epsilon(x) \cap C(x)$
satisfies the important property
\[
|C_\epsilon(x)| \ge c |E_\epsilon(x)\cap\omega_\epsilon| \ge c |E_\epsilon(x)|.
\]
\indent
We examine $A_\epsilon(x)$ first. Since
$\big|\nabla [\dist(x(z)) - \dist(x)] \big| \le c \epsilon|M(x)z|$, we deduce
\[
\big| \nabla \dist(x(z)) \cdot M(x) z  \big| \ge c \beta |z|
\quad\forall \, x(z) \in C(x).
\]
If $D_1(0):=\{z\in B_1(0): x(z)\in C_\epsilon(x)\}$, then
$|D_1(0)|\ge c |B_1(0)|$ and
\[
A_\epsilon(x) \ge c \beta \int_{D_1(0)} \int_0^1\int_0^1 s \vphi(z) |z| ds dt dz\ge C_1>0.
\]
On the other hand, using that $|\zeta'(\dist(x(z)))|\le c\epsilon$ for
$z\in B_1(0)$ in conjunction with the uniform bound of $D^2 \dist(x(z))$ provided
$\partial\Omega \in C^{1,1}$, we readily obtain 
$|B_\epsilon(x)| \le C_2 \epsilon$. This implies
\[
I_\epsilon^+ b(x) \ge C_1 - C_2 \epsilon \ge \frac12 C_1
\quad\forall \, x\in\omega_\epsilon,
\]
which translates into $I_\epsilon b \ge c \chi_{\omega_\epsilon}$
and concludes the proof.
\end{proof}  

We now discretize the boundary layer function $b$ upon defining
\begin{equation}\label{discrete-dlf}
b_h := I_h b.
\end{equation}  
\begin{lemma}[properties of $b_h$]\label{L:discrete-dlf}
There is a constant $C$ independent of $\epsilon, h$ such that
\begin{equation}\label{properties-dlf}
  I_\epsilon b_h (x_i)\ge C \chi_{\omega_\epsilon}(x_i),
  \quad
  L^\epsilon_h  b_h (x_i) \ge C\chi_{\omega_\epsilon}(x_i)
  \quad \forall \, x_i \in \Nh.
\end{equation}
\end{lemma}  
\begin{proof}
We fix $x_i\in\Nh$ and let $\dist_i(x):=\dist(x_i) + \nabla \dist(x_i) \cdot (x-x_i)$
and $\psi$ be the convex function
\[
\psi(x) := \zeta \big( \dist_i(x) \big) - \ell(x),
\]
where $\ell$ is a linear function (corrector) with the following properties
\[
\psi(x_i) = b(x_i),
\quad
\int_{\omega_i} \nabla \psi = \int_{\omega_i} \nabla b.
\]
\indent
1 {\it Integral operator.} If $x_i\notin \omega_\epsilon$, then we
again have $I_\epsilon b_h(x_i) \geq 0$ as in Lemma \ref{L:boundlayer-intopr}.
Let's consider $x_i\in\omega_\epsilon$ and write
\[
I_\epsilon b_h(x_i) = I_\epsilon b(x_i) + I_\epsilon [b_h-b](x_i)
\]
In light of Lemma  \ref{L:boundlayer-intopr} it suffices to show that
$\big| I_\epsilon [b_h-b](x_i) \big|$ is small relative to $1$, or
equivalently $\big|b_h-b\big| (x_i)$ is small relative to
$\epsilon^2$.
Write
\[
b_h-b = I_h(b-\psi) - (b-\psi) + I_h\psi - \psi,
\]
and note that $\delta [I_h\psi - \psi](x_i, y) \ge0$, whence $I_\epsilon[I_h\psi - \psi](x_i)\ge0$, 
because $\psi$ is convex. Therefore, we only have to bound the first
two terms, namely \cite{BrennerScott},
\[
\big| (b-\psi) - I_h(b-\psi) \big|(y)
\leq C h^{2-\frac{d}{p}} \|D^2(b-\psi)\|_{L^p(B_{h}(y))}
\]
for $2-\frac{d}{p}>0$ and any $y=x_i+\epsilon M(x_i) z$ with $z \in B_1(0)$. 
This yields
\[
\big| I_\epsilon[(b-\psi)-I_h(b-\psi)](x_i) \big|
\leq C\frac{h^{2-\frac{d}{p}}}{\epsilon^2} \|D^2(b-\psi)\|_{L^p(E_{\epsilon+h}(x_i))},
\]
where $E_{\epsilon}(x_i)$ is the ellipsoid introduced in \eqref{ellipsoid}.
We need to estimate
\begin{align*}
D^2(b-\psi)(x) & =
\zeta''(\dist(x)) \nabla \dist(x)\otimes\nabla \dist(x)
\\
& - \zeta''( \dist_i(x) )  \nabla \dist(x_i)\otimes\nabla \dist(x_i)
+ \zeta'(\dist(x)) D^2 \dist(x)
\end{align*}
for $x\in E_{\epsilon+h}(x_i)$. The third term is the simplest because
$|\zeta'(\dist(x))|\le C \epsilon$ for all $x\in\omega_\epsilon$. The
first two terms are problematic because $\zeta''$ is discontinuous. We
split them as follows:
\begin{align*}
T_1(x) + T_2(x) + T_3(x) : & =
\big[\zeta''(\dist(x)) - \zeta''(\dist_i(x))\big] \nabla \dist(x) \otimes \nabla \dist(x)
\\
& + \zeta''(\dist_i(x)) [\nabla \dist(x) - \nabla \dist(x_i) ] \otimes \nabla \dist(x)
\\
& + \zeta''(\dist_i(x))  \nabla \dist(x_i) \otimes [\nabla \dist(x) - \nabla \dist(x_i) ].
\end{align*}
The function
$\zeta''(\dist(x)) - \zeta''(\dist_i(x))$ which vanishes for $x\in E_{\epsilon+h}(x_i)$
except in a set $S_\epsilon(x_i)$ with measure
$|S_\epsilon(x_i)|\le C \epsilon^{d+1}$ because
the distance function $\dist \in C^{1,1}$. Consequently, recalling that
$|\nabla \dist(x)| = 1$ and $h\le\epsilon$, we arrive at
\[
\|T_1\|_{L^p(E_{\epsilon+h}(x_i))}
\le C  \epsilon^{\frac{1+d}{p}}.
\]
The remaining
two terms are similar and, employing that $\dist\in C^{1,1}$, yield
\[
\|T_i\|_{L^p(E_{\epsilon+h}(x_i))} \le 
C h \|D^2 \dist\|_{L^p(E_{\epsilon+h}(x_i))}
\le C h \epsilon^{\frac{1}{p}}
\]
for $i=2,3$.
Collecting these estimates, and taking $p\ge d$, we obtain 
\[
\big| I_\epsilon[(b-\psi)-I_h(b-\psi)](x_i) \big|
\le \epsilon^{\frac{1}{p}}\Big(\frac{h}{\epsilon}\Big)^{2-\frac{d}{p}},
\]
which is small relative to $1$ for $\epsilon$ small because $h\le\epsilon$.

\smallskip
2 {\it Laplace operator.}
Let $x_i\notin\omega_\epsilon$. Invoking \eqref{discrete-laplace} and the
fact that $b_h(x_j) \ge b_h(x_i) = -\epsilon^2$ we deduce
\[
\Delta_h b_h(x_i) \int_{\omega_i} \phi_i =
- \int_{\omega_i} \nabla b_h \cdot \nabla \phi_i  =
\sum_j k_{ij} \big( b_h(x_i) - b_h(x_j) \big) \ge 0.
\]
Otherwise, if $x_i\in\omega_\epsilon$, then decompose $\Delta_h b_h(x_i)$
as follows
\[
\Delta_h b_h(x_i) \int_{\omega_i} \phi_i
= \int_{\omega_i} \nabla I_h(b-\psi) \cdot \nabla \phi_i
+ \int_{\omega_i} \nabla I_h\psi \cdot \nabla \phi_i
\]
and examine each term separately. We start with the last term. Since
$\psi$ is convex, we can always substract a linear function and assume
that $\psi(x_j) \ge \psi(x_i)=0$ for all $j$. This correction being
linear does not alter the last term, which becomes
\[
\int_{\dm} \nabla I_h\psi \cdot \nabla \phi_i = \sum_j k_{ij} \, \psi(x_j) \le 0.
\]
It remains to show that the first term is small relative to
$h^d \approx \int_{\omega_i} \phi_i$.
We first resort to the pointwise stability of the Lagrange interpolant
$\|\nabla I_h(b-\psi)\|_{L^\infty(\omega_i)} \le \|\nabla (b-\psi)\|_{L^\infty(\omega_i)}$,
which combined with the vanishing mean property of $\nabla(b-\psi)$ in $\omega_i$
and Poincar\'e inequality
$\|\nabla (b-\psi)\|_{L^\infty(\omega_i)} \le C h^{1-\frac{d}{p}} \|D^2(b-\psi)\|_{L^p(\omega_i)}$
yields
\[  
\big| \langle \nabla I_h(b-\psi), \nabla\phi_i \rangle \big|
\le \|\nabla (b-\psi)\|_{L^\infty(\omega_i)} \|\nabla\phi_i\|_{L^1(\omega_i)}
\le C h^{d-\frac{d}{p}} \|D^2(b-\psi)\|_{L^p(\omega_i)}.
\]
We next observe that
$\|D^2(b-\psi)\|_{L^p(\omega_i)} \le C |S_h(x_i)|^{\frac{1}{p}}$
where $S_h(x_i)$ is the subset of $\omega_i$ where
$D^2(b-\psi) \ne 0$. Since $\dist\in C^{1,1}$, we have
$|S_h(x_i)| \le C h^{d+1}$ whence
\[
\big| \langle \nabla I_h(b-\psi), \nabla\phi_i \rangle \big|
\le C h^{d+\frac{1}{p}}.
\]
Combining these estimates for $\Delta_h b_h(x_i)$
with $I_\epsilon b_h(x_i) \ge C \chi_{\omega_\epsilon}(x_i)$ 
yields $L_h^\epsilon b_h(x_i) \ge C \chi_{\omega_\epsilon}(x_i)$ and
concludes de proof.
\end{proof}

\subsection{Error equation}\label{subsec:erroreqn}
We now derive an equation for $L_h^\epsilon (u_G - \ue_h)$ assuming
$\dm=\dm_h$.
By definition \eqref{discrete_pde} of $L_h^\epsilon$ and \eqref{consistency},
we can split $L_h^\epsilon u_G (x_i)$ as follows
\begin{align*}
L_h^\epsilon u_G (x_i) \int_{\dm} \phi_i = &\; 
 -\frac{\lambda}{2} \big \langle\gradv u_G , \gradv \phi_i \big\rangle
 + \big \langle \Ie u_G (x_i) , \phi_i\big \rangle
\\
= &\;  \big\langle f + T_1 + T_2 + T_3 + T_4 , \phi_i \big\rangle ,
\end{align*}
where
\begin{align*}
T_1 = &\; \Ie u_G (x_i) - \Ie u(x_i)  ,
\\
T_2 = &\; \Ie u(x_i) - \Big(\bar{A}(x_i) - \frac{\lambda}{2} I \Big) : D^2 u(x_i) ,
\\
T_3 = &\; \Big( \bar{A}(x_i) - \frac{\lambda}{2} I \Big) : \left( D^2 u(x_i) -  D^2 u(x) \right),
\\
T_4 = &\; \left( \bar{A}(x_i) - A(x) \right): D^2 u(x) ,
\end{align*}
and $\bar{A}(x_i)$ is the mean of $A(x)$ over the star $\omega_i$ defined in \eqref{mean}.
Since $L_h^\epsilon \ue_h (x_i) = f_i$, according to \eqref{discrete_pde}, 
we thus get the following expression, for all $x_i \in \Nh$, 
\begin{align}\label{erroreqn}
L_h^\epsilon [u_G - \ue_h] (x_i)
=  \left( \int_{\dm} \phi_i \right)^{-1} \big \langle T_1
+  T_2 + T_3  + T_4 , \phi_i\big \rangle.
\end{align}

\subsection{Operator consistency and convergence}\label{subsec:estimateTi}
We now derive an upper bound for the four error terms $T_i$ in \eqref{erroreqn}; 
a lower bound follows along the same lines.
To this end, we must account for the behavior in the
boundary layer $\omega_\epsilon=\dm\setminus\dm_\epsilon$, where the definition of
the second difference \eqref{extension} changes and the operator accuracy reduces to order $1$. This leads to  convergence for $C^2$-solutions.
\begin{lemma}[estimate of $T_1$]\label{T1}
Let $u \in W^2_{\infty} (\dm)$ and $b_h \in \Vhzero$ be the discrete
boundary layer function defined in \eqref{discrete-dlf}. Then there
is a constant $C = C(\dm, \sigma)$ such that
\begin{align*}
  \Ie [ u_G - u] (x_i)
  \leq C |u|_{W^2_{\infty}(\dm)} |\ln h|
  \left(\frac{h^2}{\epsilon^2} + L_h^\epsilon b_h(x_i) \right)
\quad
\forall x_i \in \Nh,
\end{align*}
\end{lemma}
\begin{proof}
Applying the $L^{\infty}$ estimate \eqref{Galerkininftyestimate} of
$u_G-u$ yields
\begin{align*}
  \delta [ u_G  - u ] (x_i, y) \leq C  |u|_{W^2_{\infty}(\dm)}
  \, h^2 |\ln h|
  \begin{cases}
    1 \quad & \text{ if $x_i\in\Omega_\epsilon$}
    \\
    \theta^{-2} & \text{ if $x_i\in\omega_\epsilon$}
  \end{cases}  
\end{align*}
whence
\[
 \Ie [u_G - u] (x_i)
 \leq C   |u|_{W^2_{\infty}(\dm)} \left( \frac{h^2}{\epsilon^2} |\ln h|
 + \frac{h^2}{\theta^2\epsilon^2} |\ln h| \chi_{\omega_\epsilon}(x_i)\right).
\]
Since every node $x_i\in\omega_\epsilon$ is at most at distance $Ch$ to
$\partial\Omega$, i.e. $\theta\epsilon\ge Ch$, we see
that the truncation error within the layer $\omega_\epsilon$ may be of
order $|\ln h|$. We thus invoke \eqref{properties-dlf}
to replace the second term by $|\ln h| \, L_h^\epsilon
b_h(x_i)$, as asserted.
\end{proof}

To estimate the term $T_4$ in \eqref{erroreqn}, we recall the assumption \eqref{Ldassumption}: if 
$
\bar{A}(x_i) = \frac{1}{\abs{\omega_i}} \int_{\omega_i} A(x) dx
$
is the mean of $A(x)$ in the star $\omega_i$ of node $x_i$, then
\begin{align*}
\left( \sum_{x_i \in \Nh} \int_{\omega_i} \Abs{ A(x) - \bar{A}(x_i)}^d
\, dx \right)^{1/d} \leq C(A) h^{\beta}
\end{align*}
for some $0<\beta\leq 1$. Note that if $A \in W^1_d(\dm)$, then this
estimate with $\beta = 1$ follows immediately from the Poincar\'{e}
inequality.
It is weaker than $A\in L^{\gamma,d} (\dm)$, the
Campanato space with index $\gamma = d+ \beta d$  and Lebesgue
integrability $d$; embedding
theory implies $A \in C^{0,\beta}(\overline{\dm})$ \cite[Theorem 4.6.1]{Kufner},
which is consistent with Schauder theory.
We also introduce the notation 
\[
S_i = \left( \int_{\dm} \phi_i \right)^{-1} \big\langle T_4, \phi_i\big\rangle
\qquad
\text{ for all $i = 1, 2, \cdots, N$, }
\]
where $N$ is the cardinality of $\Nh$. 
Now we are ready to estimate each term $T_i$. 

\begin{lemma}[estimate of error equation]\label{lemma:estimate} 
Let the mesh $\Th$ satisfy \eqref{weaklyacute}. 
If the solution $u\in C^{2,\alpha}(\overline{\dm})$
with $0 < \alpha \leq 1$, then 
\begin{align*}
 L_h^\epsilon [ u_G - \ue_h - C |\ln h|\, b_h](x_i ) 
 \leq C_{\alpha}(u) \left( \epsilon^{\alpha} + h^{\alpha}
 +  \frac{h^2} { \epsilon^2} |\ln h| \right) + S_i .
\end{align*}
where $C_{\alpha}(u) = C \left( |u|_{C^{2,\alpha}(\overline{\dm})} + |u|_{W^{2}_{\infty}(\overline{\dm})} \right)$.
If the solution $u\in C^{3,\alpha}(\overline{\dm})$,  then 
\begin{align*}
L_h^\epsilon [ u_G - \ue_h - C |\ln h|\, b_h](x_i ) 
 \leq 
\left \{
\begin{array}{ll}
C_{1+\alpha}(u)  E_1  + S_i & \text{ for $x_i \in \dme$, }
\\
C_{1+\alpha}(u) E_2  + S_i & \text{ for $x_i \in \omega_\epsilon$. }
\end{array}
\right.
\end{align*}
where 
\begin{align*}
  E_1 = \epsilon^{1+\alpha} + h + \frac{h^2}{\epsilon^2} |\ln h| ,
\quad \quad
  E_2 =  \epsilon + h + \frac{h^2}{\epsilon^2} |\ln h|  
\end{align*}
and the constant $C_{1+\alpha}(u)= C \left( |u|_{C^{3,\alpha}(\overline{\dm})} + |u|_{W^2_{\infty}(\dm)}\right)$.
Moreover, if  \eqref{Ldassumption} is valid, then
\begin{align}
\label{S}
\left( \sum_{i=1}^N \abs{ S_i }^d \abs{\omega_i} \right)^{1/d} \leq C(A, u) \; h^{\beta}
\end{align}
with a constant $C(A, u) = C(A) |u|_{W^{2}_{\infty}(\dm)}$.
\end{lemma}

\begin{proof}
The upper bound of $T_1$ follows from Lemma \ref{T1} (estimate of $T_1$)
\[
T_1 \leq C  |u|_{W^{2}_{\infty}(\dm)} \, |\ln h|
\left( \frac{h^2} {\epsilon^2}  +  L_h^\epsilon b_h(x_i) \right).
\]  
The estimate of $T_2$ is a consequence of Lemma \ref{approximation} (approximation property of $\Ie$)
\begin{align*}
\abs {T_2}  
\leq 
C |u|_{C^{2,\alpha}(\overline{\dm})} \epsilon^{\alpha} .
\end{align*}
The estimate for $T_3$, for $u \in C^{2, \alpha}(\overline{\dm})$, reads
\[
  \abs{T_3} \leq C  |u|_{C^{2, \alpha} (\overline{\dm})} h^{\alpha}.
\]
Therefore, we conclude from the error equation \eqref{erroreqn} that
\begin{align*}
L_h^\epsilon [u_G  - \ue_h] (x_i)  \leq C |u|_{C^{2,
    \alpha}(\overline{\dm})} \left( \epsilon^{\alpha} + h^{\alpha} +
|\ln h| \frac{h^2}{\epsilon^2} 
+ |\ln h| \, L_h^\epsilon b_h(x_i)  \right) + S_i
\end{align*}
for all $x_i \in \Nh$, which proves the first estimate.

For $u \in C^{3,\alpha}(\overline{\dm})$, we only need to note that, by Lemma \ref{approximation}, 
\begin{align*}
\abs{T_2} \leq \left \{
\begin{array}{ll}
C \epsilon^{1 + \alpha} |u|_{C^{3, \alpha} (\overline{\dm})} &\qquad \text{ for $x \in \dme$, }
\\
C \epsilon |u|_{C^{2, 1} (\overline\dm)} &\qquad \text{ for $x \in \omega_\epsilon$, }
\end{array}
\right.
\end{align*}
and $\abs{T_3} \leq C h |u|_{C^{2, 1} (\overline\dm)} $.
We thus have from the error equation \eqref{erroreqn} that
\begin{align*}
L_h^\epsilon \big[u_G - \ue_h -C |\ln h| b_h \big] (x_i) \leq
C_{1+\alpha}(u)
\begin{cases}
E_1  &\quad \text{ for $x \in \dme$}, 
\\
E_2  &\quad \text{ for $x \in \omega_\epsilon$}.
\end{cases}
\end{align*}

Finally, to prove the last statement, we only need to note that by definition
\[
\abs{S_i} = \left( \int_{\omega_i} \phi_i \right)^{-1} \left|
\int_{\omega_i} \left( A(x) - \bar{A}(x_i) \right) : D^2 u(x) \;
\phi_i(x) \, dx \right|.
\] 
Since $D^2u(x)$ is bounded, invoking H\"{o}lder's inequality, we obtain
\begin{align*}
\abs{S_i}^d 
\leq &\; 
|u|_{W^{2}_{\infty}(\dm)} ^d \left( \int_{\omega_i} \phi_i
\right)^{-d} \left( \int_{\omega_i} \abs{ A(x) - \bar{A}(x_i)}^d \, dx
\right) \left( \int_{\omega_i} \phi_i(x)^{\frac{d}{d-1}} \, dx  \right)^{d-1}
\\
\leq &\; 
|u|_{W^{2}_{\infty}(\dm)}^d \left( \int_{\omega_i} \phi_i \right)^{-1}
\left( \int_{\omega_i} \abs{ A(x) - \bar{A}(x_i)}^d \, dx  \right), 
\end{align*}
due to the fact that $\phi_i \leq 1$. Hence,
we infer from assumption \eqref{Ldassumption} that
\begin{align*}
\sum_{i=1}^N \abs{ S_i }^d \abs{\omega_i} 
\leq 
|u|_{W^{2}_{\infty}(\dm)}^d (d+1) \sum_{i=1}^N\int_{\omega_i}
 \abs{ A(x) - \bar{A}(x_i)}^d \, dx
 \leq C( A)  |u|_{W^{2}_{\infty}(\dm)}^d\; h^{\beta d}.
\end{align*}
This completes the proof.
\end{proof}

\begin{corollary}[convergence for $C^{2}$ solutions]\label{convergence-C2}
Let the two scales $h$ and $\epsilon$ satisfy
$\epsilon = C_1 h \abs{\ln h}$ for any constant $C_1>0$. 
If the solution $u \in C^{2}(\overline{\dm})$, the coefficient
matrix $A\in$ VMO\,$(\Omega)$, and the mesh $\Th$ satisfies
  \eqref{localweaklyacute}, then
\begin{align*}
  \lim_{h \to 0} \|\ue_h - u\|_{L^\infty(\Omega)} = 0.
\end{align*}
\end{corollary}
\begin{proof}
Let $E_h(x_i):=L_h^\epsilon \big[ u_G - \ue_h - C |\ln h| \, b_h \big] (x_i)$.
If
\begin{equation}\label{oper-consist}
\limsup_{h \to 0} E_h(x_i) \leq 0 
\end{equation}
is valid uniformly in $x_i$ as $h, \epsilon \to 0$, then applying
Theorem \ref{discrete_ABP} (discrete ABP estimate) to
the function $u_G - \ue_h - C |\ln h| \, b_h$ with vanishing trace
on $\partial\Omega$ yields 
\[
  \sup_{\dm} \big( u_G - \ue_h - C |\ln h| \, b_h \big)^- \leq C
  \left( \sum_{x_i \in \mathcal{C}_h^-(v_h)} \abs {E_h(x_i)^+}^d
  \abs {\omega_i} \right)^{1/d} \to 0
  \quad \text{as } h\to0.
\]
Realizing that $|\ln h| \, \|b_h\|_{L^\infty(\Omega)} = \epsilon^2 |\ln h|\to 0$
as $h\to 0$, we obtain
\[
\lim_{h\to0}\|(u_G - \ue_h )^-\|_{L^\infty(\Omega)} = 0.
\]
Since a similar result is valid for $(u_G-\ue_h)^+$ the assertion follows.

It thus remains to show \eqref{oper-consist}. In view of
\eqref{erroreqn}, we just estimate each term $T_i$ for $1\le i\le 4$.
Lemma \ref{T1} (estimate of $T_1$) yields
\[
  T_1 \leq C |u|_{W^{2}_{\infty}(\dm)} |\ln h|
\left( \frac{h^2} {\epsilon^2}  +  L_h^\epsilon b_h(x_i) \right).
\]  
The estimate of $T_2$ is a consequence of Lemma \ref{approximation}
(approximation property of $\Ie$)
\begin{align*}
\abs {T_2}  \to 0 \qquad
\text{as $\epsilon \to 0$}
\end{align*}
because $u \in C^2 (\overline{\dm})$.
The latter also implies
$
  \abs{T_3} \to 0 
$
as $h \to 0$. Finally, we deduce  
\begin{align*}
  \langle T_4, \phi_i  \rangle
  = \int_{\dm} ( \bar{A}(x_i) - A(x) ) : D^2 u(x) \phi_i(x) \, dx
  \le |u|_{W^2_\infty(\dm)} \int_{\omega_i}  | \bar{A}(x_i) - A(x) | \, dx
\end{align*}  
whence
\begin{align*}
    \frac{ \langle T_4, \phi_i  \rangle}{\int_{\dm} \phi_i}  
    \leq 
    C |u|_{W^2_\infty(\dm)} \frac{1}{|\omega_i|} \int_{\omega_i}  | \bar{A}(x_i) - A(x) |\, dx 
    \le C |u|_{W^2_\infty(\dm)} \eta(h),
\end{align*}
where $\eta$ is the modulus of continuity for 
$A \in$ VMO\,$(\Omega)$ defined in \eqref{vmo}. This is not obvious
because $\bar{A}(x_i)$, defined in \eqref{mean},
is the meanvalue of $A$ over the star $\omega_i$
instead of balls of radius $\le h$. To prove such a statement, note that
\[
\bar{A}(x_i) = A_h(x_i) + \frac{1}{|\omega_i|} \int_{\omega_i}
\big(A(x) - A_h(x_i)\big) \, dx
\]
yields
\[
\frac{1}{|\omega_i|} \int_{\omega_i}  | \bar{A}(x_i) - A(x) | \, dx \le
C \frac{1}{|B_h(x_i)\cap\dm|} \int_{B_h(x_i)\cap\dm} \big| A(x) -
A_h(x_i)\big| \, dx
\le C \eta(h),
\]
where $C>0$ depends on the shape regularity of $\Th$.
We thus conclude \eqref{oper-consist} uniformly in $x_i$ as
$h,\epsilon\to0$ because $\eta(h)\to0$.
\end{proof}

\subsection{Discrete barrier functions}\label{subsec:barrierfunction}
We note that while the estimate in the interior of $\dm$ is
rather straightforward the boundary estimate is more involved,
due to the reduced rate of $E_2$ in the $\epsilon$-region
$\omega_\epsilon$ close to $\bdry$ in Lemma \ref{lemma:estimate}.
We assume that $\dm$ satisfies the {\it exterior ball property}
\cite{GilbargTrudinger01}, namely at every point $z_0\in\partial\dm$
there is a ball $B_R(y)$ of radius $R>0$ lying outside $\dm$ and tangent to
$\partial\dm$ at $z_0$; this is consistent with $\partial\dm$ being of
class $C^{1,1}$. We consider the following barrier function
\cite[p.106]{GilbargTrudinger01}
\begin{equation}\label{barrier}
\psi (x) :=
  \begin{cases}
    \tau\big( |x-y|^{-\sigma} - R^{-\sigma} \big) \quad & d\ge 3
    \\
    \tau\big( |\ln |x-y||^{-\sigma} - |\ln R|^{-\sigma} \big) \quad & d= 2,
  \end{cases}  
\end{equation}
with $\tau, \sigma>0$.
It turns out that $\psi (x) \le 0$ for all $x\in\dm$, $\psi (z_0)=0$ and
for $d\ge 3$
\begin{equation}\label{Hessian-b}
D^2 \psi(x) = \frac{\tau\sigma}{|x-y|^{\sigma+2}} \Big(
(\sigma+2) \frac{x-y}{|x-y|} \otimes \frac{x-y}{|x-y|} - I \Big),
\end{equation}
whence for $\sigma, \tau$ sufficiently large depending on $\lambda,
\Lambda$ and $R$
\begin{equation}\label{prop-barrier}
A(x) : D^2 \psi (x) \ge \frac{\tau\sigma}{|x-y|^{\sigma+2}}
\Big( (\sigma+2) \lambda - \textrm{tr} (A) \Big) \ge 2
\quad \forall \, x\in\dm.
\end{equation}
The same properties hold for $d=2$; we omit details.
\begin{lemma}[discrete barrier]\label{barrier-2}
Let $\dm$ be of class $C^{1,1}$.
Given a constant $E > 0$, for each node $z \in \Nh$ with ${\rm dist} (z , \bdry) \leq Q \epsilon$, there exists a function $p_z \in \Vh$ such that 
$
L_h^\epsilon p_z (x_i) \geq E  
$
for all $x_i \in \Nh$, $p_z \leq 0$ on $\bdry$ and 
\[
\abs{ p_z (z) } \leq C E \epsilon ,
\]
provided $h, \epsilon$ are sufficiently small and satisfy
$C h |\ln h|^2 \le \epsilon |\ln h| \le 1$.
\end{lemma}
\begin{proof}
Let $z_0 \in \bdry$ be such that $|z - z_0| = {\rm dist}(z, \bdry)$,
and let $p := E \psi$
where $\psi$ is the barrier function defined in \eqref{barrier}.
Let $p_G\in\Vh$ be the Galerkin projection of $p$ which interpolates $p$ on $\bdry$. 
 According to \eqref{Galerkininftyestimate} and \eqref{Hessian-b}, we get
\[
\inftynorm{ p - p_G } \leq C  E  h^2  |\ln h|,
\]
where $C>0$ depends on $R, \tau$ and $\sigma$. Lemma \ref{T1} (estimate of $T_1$) yields
\[
\Abs{ \Ie p_G (x_i) - \Ie p(x_i) } \leq C  E \, |\ln h|
\left(  \frac{h^2} { \epsilon^2} + L_h^\epsilon b_h(x_i) \right).
\]
Thanks to the operator consistency \eqref{consistency} of the Galerkin
projection, we obtain
\[
\Abs{ L_h^\epsilon p_G (x_i) - \Le p(x_i) } \leq C E \, |\ln h|
\left(  \frac{h^2} { \epsilon^2} + L_h^\epsilon b_h(x_i) \right).
\]
where $\Le$ is defined in \eqref{pde_approx}. Hence,
\begin{align*}
L_h^\epsilon \big[ p_G + C  E \, |\ln h| \, b_h \big] (x_i)
\geq &\; \Le p (x_i) - C E \, |\ln h| \frac{h^2} { \epsilon^2}.
\end{align*}
In view of \eqref{prop-barrier} and Lemma \ref{approximation}(3) (approximation property of $I_{\epsilon}$),
we obtain $\Le p (x_i) \ge 2E - C E \epsilon^2$,
where $C>0$ is proportional to $|\psi|_{C^{3,1}(\overline{\dm})}$.
Setting $p_z := p_G + C E \, |\ln h| b_h$, this implies 
\begin{align*}
L_h^\epsilon p_z (x_i)
\geq  2 E - C E \Big(\epsilon^2 + |\ln h| \frac{h^2} { \epsilon^2} \Big)    
\geq  E
\end{align*} 
because  $C\epsilon^2 + C |\ln h| h^2/\epsilon^2 \leq 1$ 
for $h, \epsilon$ sufficiently small. Moreover
\begin{align*}
\abs { p_z (z) }
\leq &\; 
\Abs{p(z)} + \Abs{p(z) - p_G(z)} + C E |\ln h| \Abs{b_h(z)}
\\
\leq &\;
C E \epsilon + CE h^2 |\ln h| + CE \, |\ln h| \epsilon^2 \le
CE \epsilon
\end{align*}
because $Ch|\ln h|^2 \le \epsilon |\ln h| \le 1$. This concludes the
proof.
\end{proof}

\subsection{Convergence rates for classical solutions}\label{subsec:rates}
We recall that if
$A\in \textrm{VMO}(\dm)$ and $f\in L^\infty(\dm)$, then there is a
unique strong solution $u$ satisfying \eqref{CZ} for all $1<p<\infty$
\cite{ChiarenzaFrascaLongo93}.
On the other hand, if $A,f \in C^{0,\alpha}(\overline{\Omega})$
and $\partial\dm \in C^{2,\alpha}$, then
there exists a unique classical solution $u\in C^{2,\alpha}(\overline{\Omega})$
satisfying \eqref{schauder} \cite{GilbargTrudinger01}.
Below we establish two convergence rates for $\inftynorm{u - \ue_h}$
which assume both the existence of $u\in C^{2,\alpha}(\overline{\Omega})$
and $A$ having minimal regularity
compatible with $A\in C^{0,\alpha}(\overline{\Omega})$.
\begin{corollary}[convergence rate for $C^{2, \alpha}$
  solutions]\label{convergencerate-C2Holder}
  Let the two scales $h$ and $\epsilon$ satisfy 
  $\epsilon = C_1\left( h^2 \abs{\ln h} \right)^{1/(2 + \alpha)}
  $ for an arbitrary constant $C_1>0$ and $0<\alpha\le1$.
  If the solution $u$ of \eqref{pde} belongs to
  $C^{2,\alpha}(\overline{\dm})$, the coefficient matrix $A$ satisfies
  \eqref{Ldassumption} for $\frac{2 \alpha}{2 + \alpha}
  \leq \beta \leq \alpha$, and the mesh $\Th$
  satisfies \eqref{localweaklyacute}, then
\begin{align*}
\inftynorm {u - \ue_h} \leq C \big( h^2 |\ln h| \big)^{\frac{\alpha}{2 + \alpha}}
\left( |u|_{C^{2, \alpha}(\overline{\dm})} + |u|_{W^2_{\infty}(\dm)} \right),
\end{align*}
where the constant $C$ is proportional to the constant $C(\sigma, \dm,
d, \lambda)$ in Theorem \ref{discrete_ABP} (discrete ABP estimate), the constant $C(A)$
in \eqref{Ldassumption} and $C_1$. 
\end{corollary}

\begin{proof}
Lemma \ref{lemma:estimate} (estimate of error equation) gives
\begin{align*}
L_h^\epsilon \big[u_G - \ue_h  - C |\ln h| \, b_h\big] (x_i )
\leq C \left( \epsilon^{\alpha} + h^{\alpha} + \frac{h^2} {\epsilon^2}
|\ln h| \right)  + S_i.
\end{align*}
Invoking  Theorem \ref{discrete_ABP} (discrete ABP estimate), along with \eqref{S} and $\beta \leq \alpha$, we get
\begin{align*}
  \sup_{\dm} \big(u_G - \ue_h - C |\ln h| \, b_h \big)^- \leq C \left(
  \epsilon^{\alpha} + h^{\beta} +  \frac{h^2} {\epsilon^2} |\ln h| \right).
\end{align*}
Since $\abs {b_h(x)} \leq \epsilon^2$, we obtain
\begin{align*}
  \sup_{\dm} \big(u_G - \ue_h\big)^{-} \leq C \left( \epsilon^{\alpha}
  + h^{\beta} +  \frac{h^2} {\epsilon^2} |\ln h|  + \epsilon^2  |\ln h| \right),
\end{align*}
together with a similar bound for $(u_G - \ue_h)^+$. Since $\epsilon, \alpha \leq 1$,  
we get
\begin{align*}
\inftynorm { u_G - \ue_h } \leq C \left( \epsilon^{\alpha} + h^{\beta}
+  \frac{h^2} {\epsilon^2} |\ln h| \right),
\end{align*}
and combine it with \eqref{Galerkininftyestimate} to arrive at 
\begin{align*}
\inftynorm {u - \ue_h} 
\leq
\inftynorm {u - u_G} + \inftynorm {u_G - \ue_h}
\leq 
C \left( \epsilon^{\alpha} + h^{\beta} +  \frac{h^2} {\epsilon^2} |\ln{h}|  \right).
\end{align*}
We finally set $\epsilon^{\alpha} = C\frac{h^2}{\epsilon^2}|\ln h|$
with $C>0$ arbitrary, that is $\epsilon = C_1\left( h^2 |\ln h|
\right)^{1/(2 + \alpha)} $, and use the assumption that $\beta \geq
2\alpha / (2 + \alpha)$ to infer the asserted estimate.
\end{proof}

We now examine the rate of convergence for a solution $u \in
C^{3,\alpha}(\overline{\dm})$. It is worth stressing that for $\alpha =1$, we
obtain an almost linear rate $\inftynorm{u - \ue_h} \leq C h |\ln h|$.
\begin{corollary}[convergence rate for $C^{3, \alpha}$
  solutions]\label{convergencerate-C3Holder}
Let the two scales $h$ and $\epsilon$ satisfy 
$\epsilon = C_2 h^{2/(3+\alpha)}$
for an arbitrary constant $C_2>0$ and $0<\alpha\le1$.
If the solution $u$ of \eqref{pde} belongs to
$C^{3,\alpha}(\overline{\dm})$, the coefficient matrix $A$ satisfies
\eqref{Ldassumption} for $ \frac{2 + 2 \alpha} {3 +
  \alpha} \leq \beta \leq 1 $, and the mesh $\Th$
  satisfies \eqref{localweaklyacute}, then 
\begin{align*}
\inftynorm {u - \ue_h} \leq C h^{2(1+\alpha)/(3+\alpha)} |\ln h| \,
\left( |u|_{C^{3, \alpha}(\overline{\dm})} +
|u|_{W^2_{\infty}(\dm)}\right) ,
\end{align*}
where the constant $C$ is proportional to the constant $C(\sigma, \dm,
d, \lambda)$ in Theorem \ref{discrete_ABP} (discrete ABP estimate), 
the constant $C(A)$ in \eqref{Ldassumption} and $C_2$.
\end{corollary}

\begin{proof}
We start with the estimate in Lemma \ref{lemma:estimate}
\begin{align*}
\Abs{L_h^\epsilon [ u_G - \ue_h - C |\ln h| \, b_h] (x_i )} \leq 
\left \{
\begin{array}{ll}
C E_1 + S_i &\quad \text{ for $x_i \in \dme$, }
\\
C E_2 + S_i &\quad \text{ for $x_i \in \omega_\epsilon$. }
\end{array}
\right.
\end{align*}
and carry out the proof in two steps, according to the distance of $x_i$ to $\bdry$.

\vskip0.2cm
1 (boundary behavior).
Our first goal is to show that
\[
 \big( u_G - \ue_h - C |\ln h| \, b_h \big)^{-}(z) \leq C E_2 \epsilon + C h^{\beta}  \qquad
 \text{for all nodes } z \in \omega_\epsilon.
\]
For each $z\in \omega_\epsilon$, let $p_z \in \Vh$ be the barrier function in Lemma \ref{barrier-2} with $E = C  E_2$: 
\begin{align*}
L_h^\epsilon p_z(x_i) \geq  C  E_2  \quad \text{  $\forall x_i \in \Nh$ } 
\qquad \text{and} \qquad
p_z(x) \leq  0 \quad \text{ on $\bdry$.} 
\end{align*}
Set $v_h :=  u_G - \ue_h - C|\ln h| \, b_h - p_z$, and use that $E_2\geq E_1$ to deduce
\begin{align*}
L_h^\epsilon v_h (x_i) \leq S_i  \quad \text{  $\forall x_i \in \Nh$, }
\qquad \text{and} \qquad
v_h (x) \geq 0 \quad \text{  on $ \bdry$ .}
\end{align*}
Theorem \ref{discrete_ABP} (discrete ABP estimate), coupled with \eqref{S}, yields
\[
  -v_h(z) \leq \sup_{\dm} v_h^- \leq C \left( \sum_{x_i \in \Nh} \abs{S_i^+}^d \abs{\omega_i} \right )^{1/d} \leq C h^{\beta}.
\]
Hence, we infer that
\[
p_z(z) - C h^{\beta} \leq (v_h + p_z) (z) = u_G (z) - \ue_h (z) -
C|\ln h| \, b_h(z),
\]
and the assertion now follows from the estimate on $p_z(z)$ in Lemma \ref{barrier-2}.

\vskip0.2cm
2 (interior behavior).
We consider the discrete domain $\dm_{\epsilon,h} = \cup \{T \in \Th :
T \cap \dme \neq \emptyset \}$ which is slightly larger than
$\dme$. We apply again Theorem \ref{discrete_ABP} 
(discrete ABP estimate) to $u_G - \ue_h - C|\ln h|\, b_h + C E_2 \epsilon + C h^{\beta}$, which is nonnegative on $\partial \dm_{\epsilon, h}$ according to Step 1, to obtain 
\[
\sup_\Omega \big( u_G - \ue_h - C|\ln h| \, b_h \big)^- \leq C E_2 \epsilon + C E_1 + C h^{\beta}
\]
where $CE_2 \epsilon + C h^{\beta}$ accounts for the estimate of the
boundary values established already in Step 1. 
Since $\abs {b_h(x)} \leq \epsilon^2$, we infer that
\[
  \sup_{\dm} \big( u_G - \ue_h \big)^- \leq C E_2 \epsilon + C E_1
  + C \epsilon^2 |\ln h|  + C h^{\beta}.
\]
An estimate for $\sup \big(u_G - \ue_h \big)^+$ can be proved in a similar fashion. 
This leads to
\begin{align*}
  \inftynorm {u_G - \ue_h} \leq C \left( \epsilon^{1+\alpha}
  + \epsilon^2 |\ln h|
  + h^{\beta} + \frac{h^2} {\epsilon^2} |\ln h| \right).
\end{align*}
Set $\epsilon^{1+\alpha} = C h^2/\epsilon^2$
for $C>0$ arbitrary and any $\alpha\le1$,
that is $\epsilon = C_2 h^{2/(3+\alpha)}$, and recall that $\beta \geq (2+ 2 \alpha) /(3+ \alpha)$ to deduce the asserted rate of convergence.
\end{proof}

\begin{remark}[linear rate is sharp]\label{R:linear-rate}
It is worth mentioning that the estimate of Corollary \ref{convergencerate-C3Holder} for $\alpha = \beta =1$ is quasi-optimal.
To see this, we consider $d =1$, $\dm = (-1, 1)$, the solution $u(x) = x^4 + x^2 -2$ and $f(x) = 2 u'' (x)  = 24x^2 + 4$; thus $A(x) = 2$. Let $\Th$ be uniform and $u_G$ be the Galerkin projection of $u$.
Then
\begin{align*}
L_h^\epsilon u_G (x_i) = &\;  
\Big( - \big\langle u_G' , \phi_i' \big\rangle
+ \big \langle \Ie u_G (x_i) , \phi_i \big\rangle \Big)
\left( \int_{\dm} \phi_i \right)^{-1}
=
\frac 12 f_i + \Ie u_G(x_i)
\\
= &\;
 f_i + \frac 12 ( f(x_i) -  f_i )  + \Ie u(x_i) - \frac 12 f(x_i) + \Ie u_G(x_i) - \Ie u(x_i).
\end{align*}
Since $u''(x) = 12 x^2 + 2$ is quadratic, a simple calculation based on Lemma \ref{approximation} (approximation property of $\Ie$) for $\alpha, k =1$ yields
\begin{align*}
 \Ie u(x_i) - \frac 12 f(x_i) \geq \left \{
\begin{array}{ll}
2 \epsilon^2 \quad &\text{ for $x_i \in \Omega_\epsilon$,}
\\ 
0  \quad &\text{ for $x_i \in \omega_\epsilon$.}
\end{array}
\right.
\end{align*}
Since $u_G$ is exactly the Lagrange interpolant $I_h u$ for $d =1$, we
have that $v := u_G - u$ vanishes at $x = x_i$ and
$ v(x) \geq (x- x_i) ( x_{i+1} -x) \geq 0$ for all $x \in [x_i,x_{i+1}]$
because $u''(x) \ge 2$ for $x\in\Omega$.
Using this expression for $v$ we readily get
\begin{align*}
 \Ie v(x_i) \geq \left \{
\begin{array}{ll}
\frac{h^2}{2\epsilon^2} \quad &\text{ for $x_i \in \Omega_\epsilon$,}
\\ 
0  \quad &\text{ for $x_i \in \omega_\epsilon$.}
\end{array}
\right.
\end{align*}
Moreover, using that $f$ is quadratic and the symmetry of the integral below,
we note that
\begin{align*}
  f(x_i) - f_i  = &\; \left( \int_\Omega \phi_i
  \right)^{-1} \int_{-h}^{h} \big( f(x_i) - f(x_i + s) \big) \phi_i(s)
  \, ds 
  \\
= &\; \left( \int_\Omega \phi \right)^{-1} \int_{-h}^{h}  -\frac 12
\delta f(x_i, s) \phi_i(s) \, ds =
\left( \int_\Omega \phi \right)^{-1} \int_{-h}^{h}  -24 s^2 \phi_i(s) \, ds;
\end{align*}
hence $\abs {f(x_i) - f_i} \leq C h^2$.
Therefore, for $\epsilon\ge C h$ we conclude that
\begin{align*}
L_h^\epsilon \big[u_G-\ue_h\big](x_i) \geq &\; 2\epsilon^2
  + \frac{h^2}{2\epsilon^2} - C h^2
  \ge \frac12 \left(\epsilon^2 + \frac{h^2}{\epsilon^2}\right) =: E
\end{align*}
because $L_h^\epsilon \ue_h (x_i) = f_i$.
For $a>0$ to be chosen, 
let $p(x) =   \min\{ 0, a E (x^2 - (1 - \epsilon)^2 )\}$ and $p_G$ be
its Galerkin projection. Since $p(x)=0$ for $1-\epsilon\le x \le 1$,
its interpolant $p_G$ vanishes for all $x_i\in\omega_\epsilon$.
Moreover, $I_\epsilon$ being exact for quadratics implies
\begin{align*}
 L_h^\epsilon p_G(x_i) \leq \left \{
\begin{array}{ll}
 4 a E (1 + \frac{h^2}{2\epsilon^2}) \quad &\text{ for $x_i \in\Omega_\epsilon$,}
\\ 
0  \quad &\text{ for $x_i \in \omega_\epsilon$.}
\end{array}
\right.
\end{align*}
Take the constant $a$ sufficiently small so that $4 a (1 + 
  \frac{h^2}{2\epsilon^2} ) \leq 1$ to get $L_h^\epsilon \big[u_G -
  \ue_h -p_G \big] (x_i) \geq 0$. Since $u_G - \ue_h - p_G = 0$ on $\bdry$,
applying Corollary \ref{dmp} (discrete maximum principle), we then infer that 
$u_G - \ue_h - p_G \le 0$ in $\Omega$, whence
\[
\sup_{\dm} \big(\ue_h - u_G \big)^- \geq \sup_{\dm} p_G^- \geq C \Big(\epsilon^2 
+ \frac{h^2}{\epsilon^2} \Big) \geq C h,
\]
for any choice of $\epsilon$.

This example shows that, even for smooth $u$, $A$ and $f$, we can not expect the optimal rate of convergence to be better than order one. 
\end{remark}

%
\subsection{Convergence rates for piecewise $C^{2,\alpha}$-solutions}
\label{subsec:pwsmooth}
%
We have already mentioned in Section \ref{S:introduction} that there are
fundamental obstructions for the development of a PDE theory for
\eqref{pde} with general discontinuous coefficients.
In the absence of general supporting theory, we dwell now on a
practically significant case of discontinuous coefficients $A$ across a
$(d-1)$-dimensional manifold $\Sigma$; we refer to \cite{Kim07,Lorenzi72}
for partial existence and uniqueness results of strong solutions.
We assume that the domain $\dm$ splits into a finite union
of disjoint Lipschitz subdomains $\dm_j$
\begin{equation*}
  \overline{\dm} = \cup_{j=1}^J \overline{\dm}_j,
  \qquad
  \dm_j \cap \dm_i = \emptyset
  \quad
  j \ne i,
\end{equation*}
and denote the discontinuity set by $\Sigma$
\begin{equation*}
  \Sigma := \cup_{j=1}^J \partial\dm_j \cap \dm.
\end{equation*}
We further make the following assumptions:
there exists $1/d \le \alpha \le 1$ such that
\begin{equation}\label{matA}
\left( \sum_{\omega_i\subset\dm_j} \int_{\omega_i} \big| A(x) -
\bar{A}(x_i) \big|^d \, dx \right)^{1/d} \le C(A) h^\beta
\quad
  \textrm{for all } 1 \le j \le J
\end{equation}  
with $\frac{2\alpha}{2+\alpha} \le \beta \le \alpha$, and there exists
a solution $u \in W^2_\infty(\dm)$ of \eqref{pde} satisfying
\begin{equation}\label{existence}
  u \in C^{2,\alpha}(\overline{\dm}_j)
  \quad
  \textrm{for all }1\le j\le J.
\end{equation}
We now exploit that operator consistency is measured in
$L^d(\dm)$ rather than $L^\infty(\dm)$, according to
Theorem \ref{discrete_ABP} (discrete ABP estimate), to
explore the consequences of \eqref{matA}-\eqref{existence}. We do not
require that $\Sigma$ is aligned with the mesh $\Th$.
\begin{corollary}[convergence rate for piecewise $C^{2,\alpha}$-solutions]\label{convergencerate-C11}
Let $\Th$ satisfy \eqref{localweaklyacute} and let $A$ and $u$ satisfy
\eqref{matA} and \eqref{existence} with
$\frac{2\alpha}{2+\alpha} \le \beta \le \alpha$.
If the two scales $h$ and $\epsilon$ satisfy 
$\epsilon = C_3 \left( h^2 \abs{\ln h}\right)^{d/(1+2d)}$ with $C_3>0$
arbitrary, then
\begin{align*}
\inftynorm {u - \ue_h} \leq 
C \big( h^{2} |\ln h| \big)^{\frac{1}{1+2d}}
\end{align*}
where the constant $C$ is proportional to the constant $C(\sigma, \dm,
d, \lambda)$ in Theorem \ref{discrete_ABP} (discrete ABP estimate), 
$|u|_{W^2_\infty\dm)}$, $|u|_{C^{2,\alpha}(\overline{\dm}_j)}$ for
$1\le j\le J$, 
and the constant $C(A)$ in \eqref{matA}.
\end{corollary}
\begin{proof}
We divide the domain $\dm$ into two subdomains $\dm \setminus\Sigma_{\epsilon}$
and $\Sigma_{\epsilon}$, where
\begin{equation}\label{Gamma-eps}
\Sigma_\epsilon := \{x\in\dm: \dist (x,\Sigma) \le Q \epsilon \},
\qquad\Rightarrow\qquad
|\Sigma_\epsilon| \le C\epsilon
\end{equation}
because $\partial\dm_j$ is at least Lipschitz for all $j$.
We study these sets separately. 
If $x_i \in \dm \setminus \Sigma_{\epsilon}$, then Lemma
\ref{lemma:estimate} (estimate of error equation) yields
\begin{align*}
   L_h^\epsilon [ u_G - \ue_h - C |\ln h| \, b_h] (x_i ) \leq 
  C_\alpha(u,A) \Big(\epsilon^\alpha + h^\beta+ \frac{h^2}{\epsilon^2}
  |\ln h| \Big) \le C_\alpha(u,A) \epsilon^\alpha,
\end{align*}
where we have used \eqref{matA} and \eqref{existence} in each $\dm_j$
as well as the relations $\frac{2\alpha}{2+\alpha} \le \beta \le\alpha$
and
$\epsilon = C \left( h^2 \abs{\ln h}\right)^{\frac{d}{1+2d}} \ge
C \left( h^2 \abs{\ln h}\right)^{\frac{1}{2+\alpha}}$ to derive the
last inequality.

Let now $x_i \in \Sigma_{\epsilon}$. Lemma \ref{T1} (estimate of $T_1$) gives 
$
T_1 \leq C |\ln h| \left (  \frac{h^2}{\epsilon^2} + L_h^\epsilon b_h(x_i)  \right ) .
$
Since $u\in W^2_\infty(\Omega)$ and $A$ is bounded, we get
$T_3, T_4 \leq C$. Moreover, a simple variant of Lemma
\ref{approximation} (approximation property of $I_\epsilon$) implies
$\big|I_\epsilon u(x_i) -
\big(\bar A(x_i)-\frac{\lambda}{2} I\big) : D^2u(x_i) \big|\le C$ for
$x_i\in \Sigma_\epsilon$, whence $T_2\leq C$. Altogether, we have 
\[
  L_h^\epsilon [ u_G - \ue_h - C |\ln h| \, b_h] (x_i) \leq C
  \quad
\text{for $x_i \in \Sigma_{\epsilon}$}
\]

Upon applying Theorem \ref{discrete_ABP} (discrete ABP estimate) we obtain
\begin{align*}
 \sup_\dm \big( u_G - \ue_h - C |\ln h| \, b_h \big)^- 
 \leq &\; C \left( \sum_{x_i \in \dm \setminus \Sigma_{\epsilon}}
 \epsilon^{\alpha d} |\omega_i|
 + \sum_{x_i \in \Sigma_{\epsilon}} |\omega_i| \right)^{1/d} . 
\end{align*}
Invoking \eqref{Gamma-eps} yields
$\sum_{x_i \in \Sigma_{\epsilon}} |\omega_i| \le C |\Sigma_\epsilon| \le C\epsilon$, whence
\[
\sup_\dm \big( u_G - \ue_h - C |\ln h| \, b_h \big)^-
\le C \big(\epsilon^{\alpha d} + \epsilon \big)^{1/d} \le C \epsilon^{1/d}
\]
because $\alpha d \ge 1$.
Since $|\ln h| \, |b_h(x)| \leq \epsilon^2 |\ln h| \le\epsilon^{1/d}$,
we deduce
\[
\sup_\dm \big( u_G - \ue_h \big)^- \leq C \epsilon^{1/d}
= C \big( h^{2} |\ln h| \big)^{1/(1+2d)},
\]
which is the desired lower bound.  We can finally obtain the upper bound
in a similar fashion, and complete the proof.
\end{proof}

%
\section{Numerical experiments}\label{S:numerics}
%

In this section, we discuss the implementation of the two-scale method
for $d=2$ and present numerical experiments that explore convergence
rates for smooth solutions, discontinuous
coefficients, and $C^{2,\alpha}$-solutions within and beyond theory.

\subsection{Implementation}\label{S:implementation}

The change of variables $y = \epsilon M(x) z$ transforms the
integral operator \eqref{Ie}, or its modification \eqref{extension}
near the boundary, into
\[
  \Ie v(x) = \int_{B_1(0)} \frac{\delta v(x, \epsilon\theta M(x) z)}{\epsilon^2\theta^2
  } \varphi ( z) \, dz ,
\]
where $B_1(0)$ is the unit ball in $\mathbb{R}^d$. We recall that,
to preserve the essential properties of the scheme,
a quadrature formula of the form
\[
  \Qe v(x) = \sum_{k} w_k \; \frac{\delta v(x,\epsilon\theta M(x) q_k)}{\epsilon^2\theta^2
  } \; \varphi ( q_k)
\]
must satisfy the three conditions in subsection \ref{subsec:quadrature}.
We thus use a quadrature formula for the integral in $B_1(0)$ and
$d =2$ with the following weights $\{w_k\}_{k=1}^6$
and nodes $\{q_k\}_{i=1}^6=\{(\rho_k,\theta_k)\}_{k=1}^6$
(in polar coordinates) \cite{Stroud71}:
\[
w_k = \frac{\pi} 6,
\quad
\rho_k = \frac{\sqrt{2}}{2},
\quad
\theta_k = \frac{\pi k}{3}.
\]
This formula is exact for cubic polynomials because it is exact
for quadratics and is also symmetric.

If $u_h^\epsilon = \sum_j U_j \phi_j$ where $\phi_j$ is the hat function at
node $x_j$, then we have
\[
Q_{\epsilon} u_h^\epsilon (x_i) =
\sum_{k, j} w_k U_j \frac{\delta  \phi_j(x_i, \epsilon\theta M(x_i)
  q_k )}{\epsilon^2\theta^2} \varphi(q_k)
= \sum_{j} M_{ij} U_j
\]
for all $x_i\in\Nh$, where
\[
M_{ij} = \sum_k w_k \frac{\delta  \phi_j(x_i, \epsilon\theta M(x_i) q_k )}{\epsilon^2\theta^2} \varphi(q_k)
\]

We implement the two-scale method within the MATLAB package FELICITY
\cite{Walker13}. To evaluate $\phi_j\big(x_i+\epsilon\theta M(x_i) q_k\big)$, we
resort to the search routine in  FELICITY to find all the elements
containing the quadrature points $y_k=x_i+\theta\epsilon M(x_i)q_k$,
namely for each node $x_i$
we find the basis functions $\phi_j$ which are non-zero at $y_k$,
and evaluate them at $y_k$.

\begin{remark}[meshes]\label{rmk:meshes}
  The domains are rectangles and the meshes are made of right triangles obtained from cutting cartesian rectangles along their main diagonal.
\end{remark}

\subsection{Example 1: smooth solution}
%
We consider the domain $\dm = [0,1]^2$, solution $u$,
and anisotropic matrix $A$ with moderate aspect ratio of size $5$
given by
\begin{align}\label{example-2}
u(x,y) = \frac{y}{2} \sin(2 \pi x ) + \frac{y}{5} \sin(5\pi y),
\qquad
A(x,y) = 
\bigl(\begin{smallmatrix}
3&-2\\ -2 &3
\end{smallmatrix} \bigr).
\end{align}
We take $\epsilon = \frac 12 \sqrt{h}$, as suggested by
Corollary \ref{convergencerate-C3Holder}.
Figure \ref{fig:LAR} displays a linear asymptotic convergence
rate. This validates Corollary \ref{convergencerate-C3Holder}
and also complements Remark
\ref{R:linear-rate}, thereby showing that the two-scale method cannot
be better than first-order also for dimension $d=2$.
In addition, we stress that the PDE
\eqref{pde_a} can be written in divergence form as div $(A\nabla u)=f$,
but the use of monotone FEMs with weakly acute
meshes is prohibitive with aspect ratio $5$.
\begin{figure}[h!]\label{fig:LAR}
\centering
\includegraphics[width=0.49\textwidth]{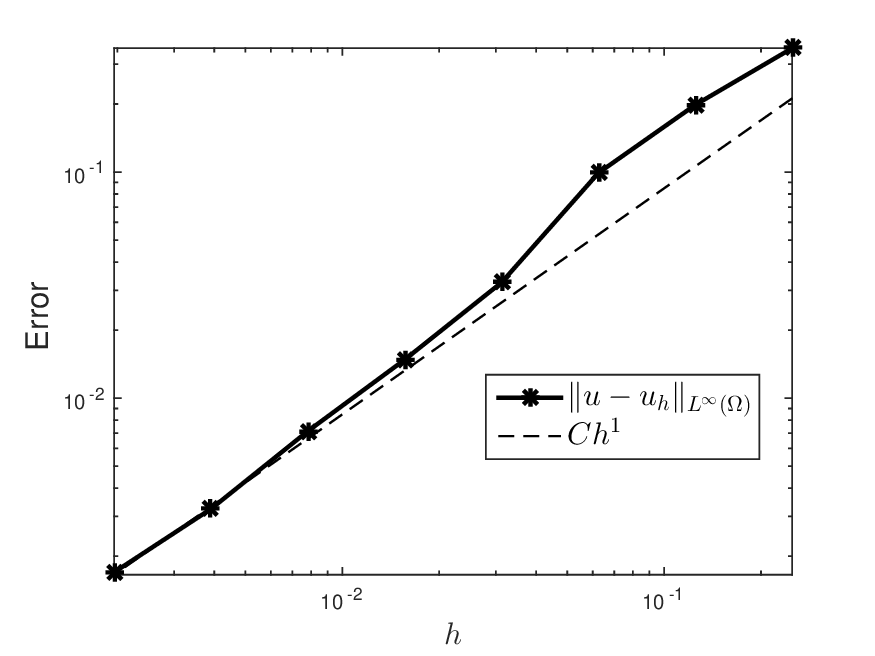}
\caption{Example 1: smooth solution $u$ and anisotropic $A$ with
aspect ratio $5$. The choice $\epsilon = \frac 12 \sqrt{h}$
yields a linear asymptotic rate which
is consistent with both Corollary \ref{convergencerate-C3Holder} and
Remark \ref{R:linear-rate}.}
\end{figure}

We point out that on average, it takes about $30\%$ of the computing time to assemble the
matrix, mostly due to the FELICITY search routine
to evaluate second differences. It takes about $50\%$ of the computing
time to solve the (non-symmetric) linear systems using MATLAB
backslash. However, solving the system requires significantly more time than
assembling the matrix for finer meshes.
The finest meshsize is $h = 2^{-9}$, which corresponds to about
$2.6\times10^{5}$ degrees of freedom and a relative pointwise error of
about $0.3 \, \%$.

\subsection{Example 2: discontinuous coefficients}
%
Let $\dm = [-1,1]^2$, the coefficient matrix $A$ exhibit the checkerboard structure
\begin{equation}\label{example-1A}
A (x, y) = 
\bigl(\begin{smallmatrix}
2 & 1\\ 1 & 2
\end{smallmatrix} \bigr)
\quad \text{ if $ x y > 0$,}
\qquad
A (x, y) = 
\bigl(\begin{smallmatrix}
2 & -1\\ -1 & 2
\end{smallmatrix} \bigr)
\quad \text{ if $ x y \leq 0 $,}
\end{equation}
with discontinuities across the axes,
and the exact solution be given by
\begin{align}\label{example-1u}
  u (x, y) = \phi(x) \phi(y) 
  \quad
  \text{ where }
  \quad
  \phi(x) = (x e^{1 - |x|} - x).
\end{align}
A simple calculation yields
\[
\gradv \phi(x) = \big(1 - |x| \big) e^{1 - |x|}  - 1
\quad
\text{and}
\quad
D^2 \phi(x) = \big(x - 2 \, \text{sgn}(x) \big) e^{1 - |x|}.
\]
Since
$u \in W^3_\infty(\Omega\setminus\Sigma), A\in W^1_\infty(\Omega\setminus\Sigma)$
with discontinuity set $\Sigma$
being the two coordinate axes, we can take $\alpha=1$,
choose $\epsilon = 0.5 h^{4/5}$
and expect a convergence rate $2/5$ according to Corollary \ref{convergencerate-C11}.
Figure \ref{fig:checkboard} (a) displays an experimental order of convergence
approximately $0.74$, which is much higher than predicted. 
\begin{figure}[h!]
\begin{center}
\includegraphics[width=0.49\textwidth]{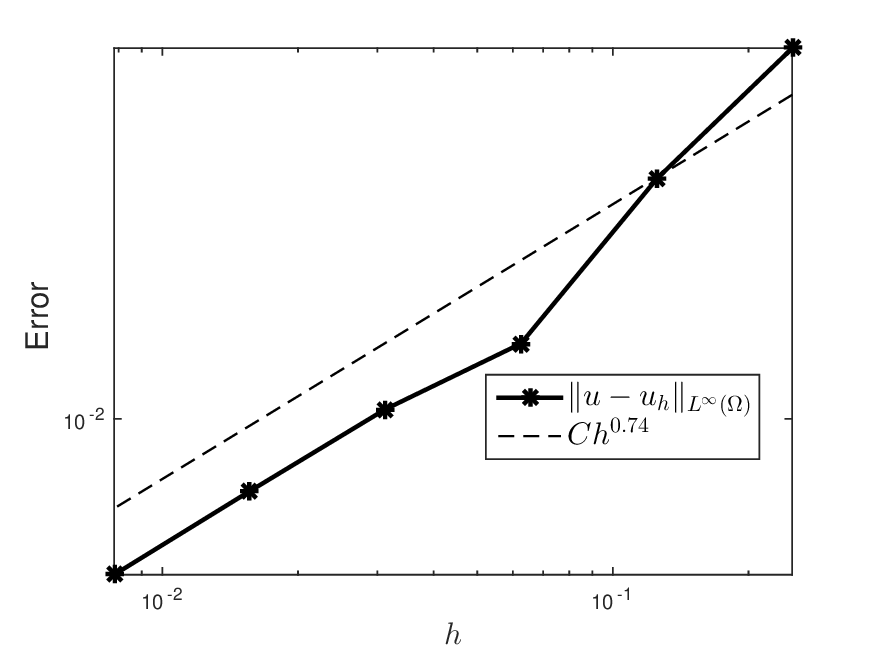}
\includegraphics[width=0.49\textwidth]{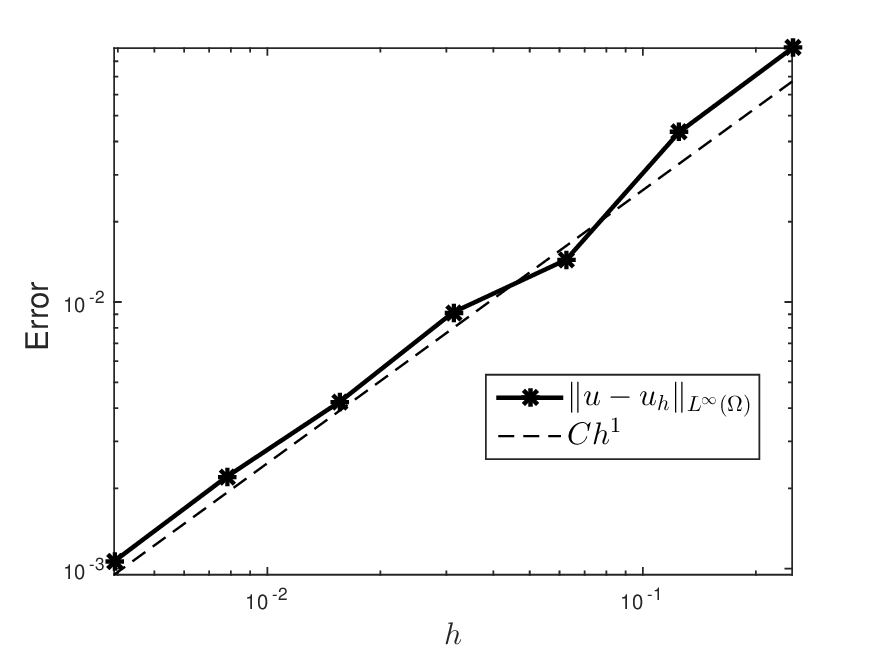}
\end{center}
\caption{\small  Example 2. The figure on the left shows that when
  $\epsilon = \frac 12 h^{4/5} $, the convergence rate is $0.74$,
  better than the rate $2/5$ predicted in Corollary
  \ref{convergencerate-C11}.  The figure on the right shows that when
  $\epsilon = h^{1/2} $, the convergence rate is $1$. This example
  shows that the discrete ABP estimate may overestimate the
  $L^{\infty}$-error when $D^2u$ is discontinuous. }
\label{fig:checkboard}
\end{figure}
The finest meshsize is $h=2^{-7}$, which corresponds to about
$6.5\times 10^4$ degrees of freedom and a relative pointwise accuracy of
about $1.3 \, \%$.

To explain this better convergence rate, we note that the consistency error
\[
E_h^\epsilon(x_i) := f_i - L_h^\epsilon I_h u(x_i)
= L_h^\epsilon [u_h^\epsilon-I_h u] (x_i)
\]
is concentrated along the $x$ and $y$-axis, where $I_h u$ is the piecewise
linear interpolant of $u$ on mesh $\Th$.
We have found computationally that the error $e_h^\epsilon := \ue_h - I_h u$ changes
rapidly (of order $O(1)$) in the direction perpendicular to $\Sigma$
and smoothly (of order $O(h^2)$) along $\Sigma$.
In fact, if node $x_i$ belongs to the $y$-axis, we observe
\begin{equation}\label{2nd-difference}
  \frac{\big|\delta e_h^\epsilon (x_i, h v_1)\big|}{h^2} =  O(1),
  \quad
  \frac{\big|\delta e_h^\epsilon (x_i, h v_2)\big|}{h^2} =  O(h^2).
\end{equation}
where $v_1 = (1, 0)$ and $v_2 = (0, 1)$.
We see a similar behavior with $v_1$ and $v_2$ exchanged
if $x_i$ belongs to the $x$-axis.
We believe that the discrete ABP estimate of Theorem \ref{discrete_ABP}
overestimates the pointwise error in this case.

In order to give a plausible explanation,
we start with Proposition \ref{Alexandroff} (discrete Alexandroff estimate) applied to
$e_h^\epsilon$
\begin{align}
  \label{numerics:A-estimate}
  \sup_{\dm} \, (e_h^\epsilon)^{-}
  \leq
  C \left( \sum_{x_i \in \mathcal{C}_h^-(e_h^\epsilon)} | \gradv e_h^{\epsilon} (x_i)| \right)^{1/2}.
\end{align}
Applying the definition of sub-differential,
we deduce $\gradv e_h^{\epsilon} (x_i) \subset R(x_i)$ where
\begin{align*}
  R(x_i) = \{ 
    w \in \mathbb R^d,  
    &\; \pm  w \cdot h v_1 \leq e_h^{\epsilon} (x_i \pm h v_1) - e_h^{\epsilon} (x_i) 
  \\
  \; \text{and} \;
  &\; \pm  w \cdot h v_2 \leq e_h^{\epsilon} (x_i \pm h v_2) - e_h^{\epsilon} (x_i) \}
\end{align*}
It is easy to check that
$|R(x_i)| = \frac{\big|\delta e_h^\epsilon(x_i, h v_1 )\big| \,
\big|\delta e_h^\epsilon(x_i, h v_2)\big|}{h^2}$ which yields
\[
  | \gradv e_h^{\epsilon} (x_i)| 
  \leq 
  \frac{\big|\delta e_h^\epsilon(x_i, h v_1 )\big| \,
  \big|\delta e_h^\epsilon(x_i, h v_2 )\big|}{h^2}.
\]
Hence, since $|\omega_i|\approx h^2$ for $d=2$, \eqref{numerics:A-estimate} yields
\[
  \sup_{\dm} \, (e_h^\epsilon)^{-}
  \leq 
  C \left( \sum_{x_i \in \mathcal{C}_h^-(e_h^{\epsilon})} \left(
  \frac{\delta e_h^\epsilon(x_i, h v_1 )}{h^2}  +  \frac{\delta
    e_h^\epsilon(x_i, h v_2)}{h^2}  \right)^2 |\omega_i| \right)^{1/2} ,
\]
because $\delta e_h^{\epsilon}(x_i,h v_j )\ge0$ for
$x_i\in\mathcal{C}_h^-(e_h^{\epsilon})$ and $j=1,2$.
We deal with meshes $\Th$ for which the discrete Laplacian satisfies
$\Delta_h  v_h (x_i) = \frac{\delta v_h(x_i,
h v_1 )}{h^2}  +  \frac{\delta v_h(x_i, h v_2)}{h^2}$ for any piecewise linear
function $v_h$; see Remark \ref{rmk:meshes}. Consequently, 
\[
  \sup_{\dm} \, (e_h^\epsilon)^{-}
  \leq 
  C \left( \sum_{x_i \in \mathcal{C}_h^-(e_h^{\epsilon})} (\Delta_h e_h^\epsilon (x_i) )^2 |\omega_i| \right)^{1/2}
  \leq
  C \left( \sum_{x_i \in \mathcal{C}_h^-(e_h^{\epsilon})}
  (L_h^\epsilon e_h^\epsilon (x_i) )^2 |\omega_i|\right)^{1/2}.
\]
Applying \eqref{2nd-difference} gives
\[
\Delta_h e_h^{\epsilon} (x_i) = \frac{\delta e_h^{\epsilon}(x_i, h v_1 )}{h^2} + \frac{\delta e_h^{\epsilon}(x_i, h v_2 )}{h^2} = O(1)
\]
and 
\[
\frac{\delta e_h^{\epsilon}(x_i, h v_1 )}{h^2}  \frac{\delta
   e_h^{\epsilon}(x_i,  h v_2 )}{h^2} = O(h^2)
\] 
for nodes $x_i \in \Sigma_\epsilon\cap\mathcal{C}_h^-(e_h^\epsilon)$,
where $\Sigma_{\epsilon}$ is defined in \eqref{Gamma-eps}.
Therefore, setting $\mathcal{C}_h^- := \mathcal{C}_h^-(e_h^\epsilon)$
and accounting for the correct contribution of $\Sigma_{\epsilon}$,
\eqref{numerics:A-estimate} implies
\begin{align*}
  \Big| \sup_{\dm} \, (e_h^\epsilon)^{-} \Big|^2
  & \leq  
  C \sum_{x_i \in \mathcal{C}_h^-  \setminus \Sigma_{\epsilon}}
  (L_h^\epsilon e_h^\epsilon (x_i) )^2 |\omega_i|
  \\
  & + C\sum_{x_i \in  \mathcal{C}_h^- \cap \Sigma_{\epsilon}} \frac{\delta e_h^{\epsilon}(x_i, h v_1 )}{h^2}  \frac{\delta e_h^{\epsilon}(x_i, h v_2 )}{h^2} |\omega_i|
  \leq \; 
  C \left( \Big( \epsilon^2 + \frac{h^2}{\epsilon^2} \Big)^{2} + h^2\epsilon \right),
\end{align*}
while the discrete ABP estimate overestimates
$\sup_{\dm} \, (e_h^\epsilon)^{-}$.
If we now choose $\epsilon = \sqrt{h}$, then the rate of convergence
is order $h$ which is consistent with Figure \ref{fig:checkboard} (b)
The finest meshsize in such figure
is $h=2^{-8}$, which leads to about $2.6 \times 10^5$
degrees of freedom and a pointwise relative error of about
$0.4 \, \%$.

\subsection{Example 3: $C^{2,\alpha}$-solution and
    $C^{0,\alpha}$-coefficients}\label{SS:singular}
%
We finally consider $\dm=(-1,1)^2$
and the following solution $u$ and coefficient matrix $A$
\begin{align}\label{example-3}
u(x) = |x|^{2+\alpha} 
\quad
\text{ and }
\quad
A(x) = I + |x|^{\alpha} \frac{x}{|x|} \otimes \frac{x}{|x|}
\end{align}
with $0 < \alpha < 1$; we choose $\alpha = 0.4$.
Since $ u \in C^{2,\alpha}(\overline{\dm})$ and
$A \in C^{0,\alpha}(\overline{\dm})$, we take $\epsilon = 1.5
h^{2/(2+\alpha)}$, and expect a convergence rate  $O(h^{2\alpha/ (2+\alpha)}) = O(h^{1/3})$
according to Corollary \ref{convergencerate-C2Holder}. This prediction
is verified in Figure \ref{fig:singular} (a) which shows an
approximate rate $1/3$. The finest meshsize is $h=2^{-8}$, which gives
rise to about $2.6 \times 10^5$ degrees of freedom and a relative
pointwise accuracy of about $2.3 \, \%$.
\begin{figure}[h!]
\begin{center}
\includegraphics[width=0.49\textwidth]{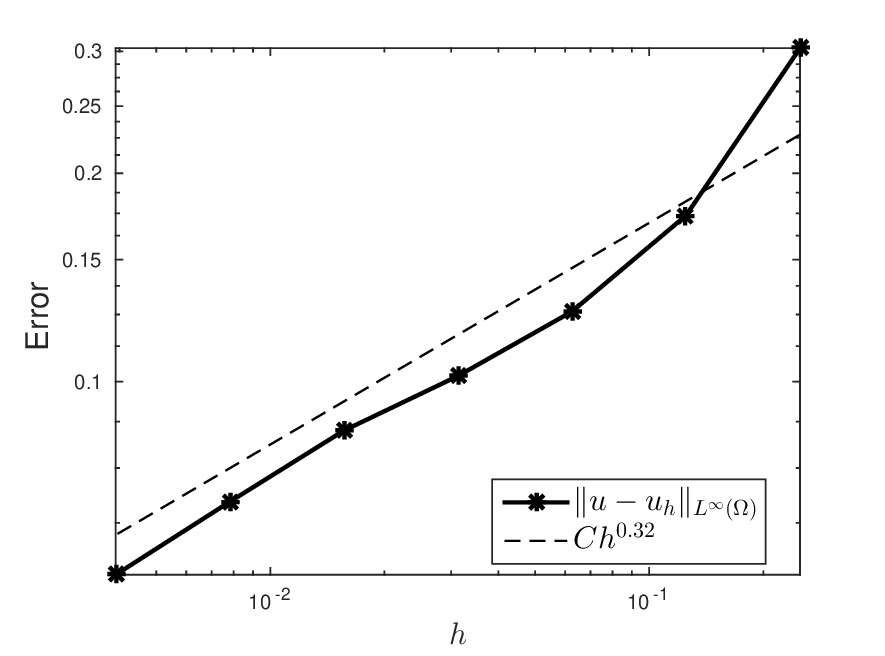}
\includegraphics[width=0.49\textwidth]{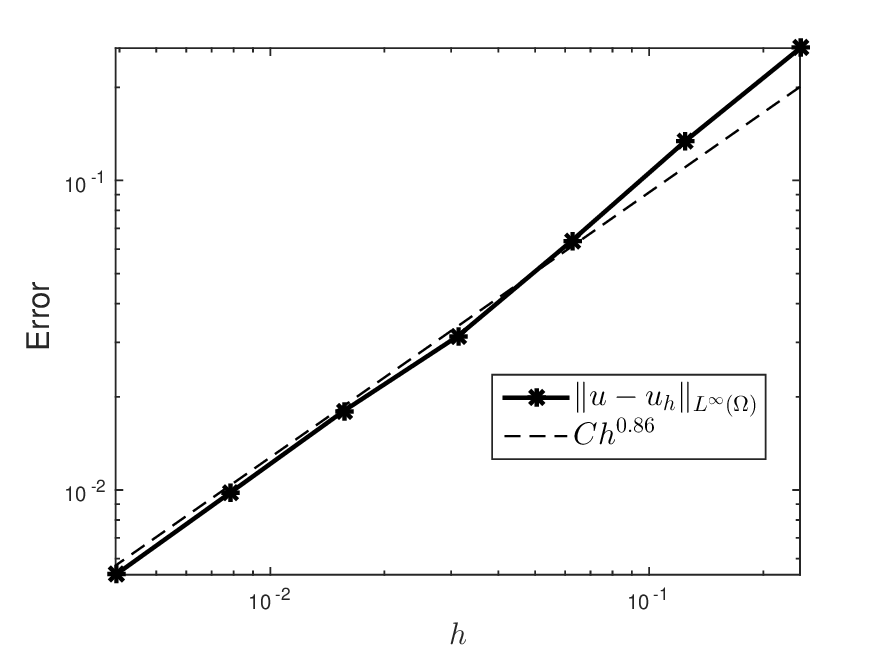}
\end{center}
\caption{\small  Example 3: $C^{2,\alpha}$-solution and
    $C^{0,\alpha}$-coefficients.
(a) If $\epsilon = 1.5 h^{2/(2 + \alpha)} $, then the convergence rate
is $1/3$, which is consistent with Corollary \ref{convergencerate-C2Holder}.
(b) If $\epsilon = 1.5 h^{2/(3 + \alpha)} $, then the convergence rate
is about $0.86$, which is better but is not supported by 
$C^{2,\alpha}$-regularity of $u$.}
\label{fig:singular}
\end{figure}

The error $u_h^\epsilon - u_G$ in the $L^{\infty}$-norm
is bounded by the operator consistency error
\[
E_h^\epsilon(x_i) := f_i - L_h^\epsilon u_G(x_i)
= L_h^\epsilon [u_h^\epsilon - u_G] (x_i)
\]
in the discrete  $L^d$-norm,
according to Theorem \ref{discrete_ABP} (discrete ABP estimate).
Since $u \in H^{3+\alpha}(\dm)$ and $d=2$, we conjecture that the quantity
\[
\left( \sum_{x_i \in \Nh}
\big|E_h^\epsilon(x_i)\big|^2 |\omega_i| \right)^{1/2} = O\left(\epsilon^{1+\alpha} + \frac {h^2}{\epsilon^2} \right),
\]
dictates the pointwise convergence rate of $u - u^{\epsilon}_h$.
Assuming this behavior, choosing $\epsilon = O(h^{2/(3+\alpha)})$,
and applying Theorem \ref{discrete_ABP}, we deduce 
\[
\norm{ u - u^{\epsilon}_h }_{L^{\infty} (\dm)} \leq O(h^{\frac{2+2\alpha}{3+\alpha}}) \approx O(h^{0.82})
\]
which is faster than the rate from Corollary
\ref{convergencerate-C2Holder}. In Figure \ref{fig:singular} (b), we
observe that the computational order of convergence is about
$0.86$, which confirms this heuristic explanation;
we are currently exploring this issue \cite{NochettoZhang16b}.
The finest meshsize in Figure \ref{fig:singular} (b) is $h=2^{-8}$, which leads to about
$2.6\times 10^5$ degrees of freedom and a pointwise relative accuracy
of about $0.23 \, \%$.

\bigskip
{\bf Acknowledgements}: We would like to thank L. Caffarelli for bringing
up the integro-differential approach of
\cite{CaffarelliSilvestre10} to us, as well as C. Gutierrez for
mentioning the discrete ABP estimate of \cite{KuoTrudinger00}.
We would also like to thank Tengfei Su for implementing the
 two-scale method and the referees for their incisive comments and
 suggestions which led to a much better exposition of techniques and results.


\bibliographystyle{plain}
%


\def\cprime{$'$}


\end{document}